\documentclass[10pt, a4paper, oneside]{article}

\usepackage[T1]{fontenc}
\usepackage{todonotes}
\usepackage[english]{babel}
\usepackage[utf8]{inputenc}
\usepackage{indentfirst}				
\usepackage{booktabs}				
\usepackage{multirow}					
\usepackage{tabularx}					
\usepackage{graphicx}					
\usepackage{caption}					
\usepackage{amsmath,amssymb,amsthm,mathrsfs,mathtools}	
\usepackage{bm,braket,stmaryrd,esint}
\usepackage{hyperref}					
\usepackage{bookmark}					
\usepackage{fullpage}
\usepackage[capitalise,noabbrev]{cleveref}
\usepackage{enumitem, color, comment}

\newtheorem{thm}{Theorem}[section]
\newtheorem{prop}[thm]{Proposition}

\newtheorem{rem}[thm]{Remark}
\newtheorem{lem}[thm]{Lemma}

\newtheorem{cor}[thm]{Corollary}
\newtheorem*{ex}{Example}

\newtheorem*{que*}{\textcolor{BrickRed}{Question}}

\def\bd{\begin{defn}}
\def\ed{\end{defn}}
\def\br{\begin{rem}}
\def\er{\end{rem}}
\def\bex{\begin{ex}}
\def\eex{\end{ex}}
\theoremstyle{plain}
\newtheorem{defn}[thm]{Definition}
\newcommand{\dd}{\textup{d}}
\DeclareMathOperator*{\Res}{Res}

\newcommand{\todoSeverin}[1]{\textcolor{red}{~\\SC: #1\\}}

\newcommand{\resp}{respectively{ }}

\begin{document}

\title{\textsc{Topological recursion for Orlov--Scherbin tau functions, and constellations with internal faces}}

\date{\vspace{-5ex}}
 
\author{
	\setcounter{footnote}{0}
	Valentin Bonzom\footnote{Universit\'{e} Sorbonne Paris Nord, LIPN, 99 avenue Jean-Baptiste Cl\'ement, 93430 Villetaneuse, France.}\,\,,
	\setcounter{footnote}{1}
	Guillaume Chapuy\footnote{Universit\'{e} Paris Cit\'{e}, IRIF, 8 Place Aur\'{e}lie Nemours, 75205 Paris Cedex 13, France.}\,\,,
	S\'everin Charbonnier\footnotemark[2]\,\,,
	Elba Garcia-Failde\footnote{Sorbonne Universit\'{e} and Universit\'{e} Paris Cit\'{e}, CNRS, IMJ-PRG, Place Jussieu, 75252 Paris Cedex 05, France.\newline This project (GC, SC, EGF) has received funding from the European Research Council (ERC) under the European Union's Horizon 2020 research and innovation programme (grant agreement No.~ERC-2016-STG 716083 ``CombiTop'').  GC is supported by the grants ANR-19-CE48-0011 ``COMBIN\'{E}'' and ANR-18-CE40-0033 ``DIMERS''. VB is supported by the grants ANR-20-CE48-0018 ``3DMaps'' and ANR-21-CE48-0017 ``LambdaComb''. VB is grateful to IRIF for the working environment which he enjoys as an associated member.}
}

\maketitle

\begin{abstract}
	We study the correlators $W_{g,n}$ arising from Orlov--Scherbin 2-Toda tau functions with rational content-weight $G(z)$, at arbitrary values of the two sets of time parameters. Combinatorially, they correspond to generating functions of weighted Hurwitz numbers and $(m,r)$-factorisations of permutations. When the weight function is polynomial, they are generating functions of constellations on surfaces in which two full sets of degrees (black/white) are entirely controlled, and in which internal faces are allowed in addition to boundaries. 

	We give the spectral curve (the ``disk'' function $W_{0,1}$, and the ``cylinder'' function $W_{0,2}$) for this model, generalising Eynard's solution of the 2-matrix model which corresponds to $G(z)=1+z$, by the addition of arbitrarily many free parameters.
	Our method relies both on the Albenque--Bouttier combinatorial proof
	of Eynard's result 
	by slice decompositions, which is strong enough to handle the polynomial case, and on algebraic arguments. 

	Building on this, we establish the topological recursion (TR) for the model.  Our proof relies 
on the fact that TR is already known at time zero (or, combinatorially, when the underlying graphs have only boundaries, and no internal faces) by work of Bychkov--Dunin-Barkowski--Kazarian--Shadrin (or Alexandrov--Chapuy--Eynard--Harnad  for the polynomial case),	and on the general idea of deformation of spectral curves due to Eynard and Orantin, which we make explicit in this case.
As a result of TR, we obtain strong structure results for all fixed-genus generating functions.

Our techniques also cover the case where $G(z)$ is a rational function times an exponential (containing in particular the case of classical Hurwitz numbers).
\end{abstract}

\section{Introduction}

\subsection{Context}

In the last decades, there has been an immense interest in combinatorics, mathematical physics, and enumerative geometry, in the study of Hurwitz numbers and their variants, which enumerate various families of branched coverings of the sphere according to their ramification profile, see e.g.~\cite{GouldenJackson1997,Okounkov2000,EkedahlLandoShapiroVainshtein2001,GouldenJacksonVakil2005,OkounkovPandharipande2006, BousquetSchaeffer2000,AlexandrovChapuyEynardHarnad2020, GouldenGuayPaquetNovak:polynomiality}.
A remarkable feature of the field is that grand generating functions of Hurwitz numbers of various kinds give rise to tau functions of integrable hierarchies (such as the KP or the 2-Toda hierarchy), see~\cite{Okounkov2000, GouldenJackson2008,Guay-PaquetHarnad2017}.
Branched coverings are also in bijection with certain factorisations of permutations in the symmetric group, and with certain graphs embedded on orientable surfaces called \emph{maps}, whose enumeration has been a question of interest in combinatorics much before the connection to Hurwitz numbers was noticed (e.g.~\cite{Tutte1954,Tutte1962a,BenderCanfield1986,BenderCanfield1991,BenderCanfield1994}), and is still a subject of active study today, see~\cite{BouttierDiFrancescoGuitter2004, Chapuy2009, BouttierGuitterMiermont2021, BernardiFusy2012, AlbenquePoulalhon2015, AlbenqueLepoutre2019, AlbenqueBouttier2022, Chapuy:PTRF,Eynard:book} for entry points.
Today, Hurwitz numbers have become a unifying topic among these fields, being at the origin of many new developments but also of transfers of techniques between those different domains.

A central question of the field is to understand the structure of these enumerative problems, when one fixes the underlying surface. A unifying answer to such questions has been progressively given in the light of the Chekhov--Eynard--Orantin \emph{Topological Recursion (TR)}~\cite{EO}, a universal procedure which enables one to compute certain invariants, usually denoted by $\omega_{g,n}$, attached to the surface of genus $g$ with $n$ boundaries, from a small initial data consisting of an algebraic object called the \emph{spectral curve}. It has been gradually observed along the years that many natural problems in combinatorics or enumerative geometry were actually instances of the topological recursion, in the sense that the correlation functions $W_{g,n}$ naturally associated to these counting problems could be recovered by this procedure, with an appropriate spectral curve~\cite{EO, EynardOrantinWeilPetersson, BorotEynard2011, EMS11, Borot13, AC14, KZ15, DOPS17, AC18, ABCDGLW19, BCG21}. Although it was initially developed to provide solutions to the topological expansion of matrix models~\cite{Eynard04, AMM05, CE06, ChekhovEynardOrantin2006}, and in the closely related context of the enumeration of maps in combinatorics~\cite{Eynard:book}, the topological recursion has grown into a vast field of research with deep mathematical connections in geometry and mathematical physics~\cite{BKMP09, Eynard2014, DOSS, BE15, BE17, FLZ20, EGMO21}.
It is important to notice that, although many examples of models giving rise to TR have now been proved to exist, the proofs are often model-dependent, and the question of the unification of methods, of results, and of giving conceptual explanations for TR to hold in vast generality is a subject of active interrogation.

\smallskip

A vast class of models which contains many previously studied variants of Hurwitz numbers was introduced a few years ago by Guay-Paquet and Harnad under the name \emph{weighted Hurwitz numbers}~\cite{Guay-PaquetHarnad2017}. These numbers are parametrised by a certain \emph{weight function}, denoted by $G$. The grand partition function of weighted Hurwitz numbers is defined by an explicit expansion as a sum over integer partitions involving Schur functions (see~\eqref{eq:tau} below), in which $G$ controls a certain product-weight attached to the content of the underlying partitions. In this framework, generating functions of weighted Hurwitz numbers are archaetypal examples of hypergeometric tau functions of the 2-Toda and KP hierarchies, also called Orlov--Scherbin tau functions, see~\cite{OrlovScherbin2000}. At the combinatorial level, weighted Hurwitz numbers count the number of solutions to certain factorisation problems in symmetric groups, which have a natural branched covering interpretation, see~\cite{Guay-PaquetHarnad2017} or Appendix~\ref{sec:WHN}. 
Many examples of weighted Hurwitz numbers are accessible by simple choices of the function $G$, for example: bipartite maps (also called hypermaps, dessins d'enfants, Belyi curves, in the literature) for $G(z)=(1+z)$, constellations (also called Bousquet-M\'elou--Schaeffer numbers after~\cite{BousquetSchaeffer2000}) for $G(z)=(1+z)^m$ or more generally a polynomial $G(z)$, or monotone Hurwitz numbers~\cite{GouldenGuayPaquetNovak:polynomiality} for $G(z)=(1\pm z)^{-1}$. 
We also mention that weighted Hurwitz numbers have been generalised to the context of $\beta$-deformations and non-orientable surfaces~\cite{ChapuyDolega2020, BonzomChapuyDolega2021} and have become a tool to approach Jack polynomials~\cite{BenDali:Integrality}. Studying the structure of topological expansions of these deformations is a very natural question, which however goes much beyond the present work.

At this stage it is important to notice that in all these models, two natural families of ``time'' variables (noted $p_i$ and $q_i$ in this paper) give control on two families of ``degree parameters'' on the associated combinatorial objects, corresponding to so-called \emph{double} (weighted) Hurwitz numbers.
However, working with the two families of parameters makes problems much more difficult, even in genus zero, and many of the works connecting the subject to TR restrict attention to \emph{single} numbers or variants, in which one of the two families is degenerated. Eynard's solution of the 2-matrix model \cite{Eynard2002} in the planar case, and with TR at all genus \cite{EO05, ChekhovEynardOrantin2006}, is a notable exception.

\smallskip

Motivated by this discussion, a natural goal for the unification of topological recursion approaches to Hurwitz problems, would be to prove a general TR statement for weighted Hurwitz numbers. This would encapsulate many of the previous results, giving in particular a unified proof for them. This project was initiated in the papers~\cite{AlexandrovChapuyEynardHarnadJ2018,AlexandrovChapuyEynardHarnad2020}, which proved TR for weighted Hurwitz in the case where $G(z)$ is polynomial, \emph{under the assumption that the first family of time parameters is equal to zero}. Combinatorially, this assumption amounts to forbidding the associated objects to have internal faces in addition to the $n$ boundaries accounted for in the generating function $W_{g,n}$. The unification was further advanced in~\cite{BDBKS1, BDBKS2} which covers the case of a rational function $G$, with the same assumption on times.

\subsection{Overview}

In this paper, we prove the topological recursion for weighted Hurwitz numbers of rational weight-function $G(z)$, for arbitrary values  of the time parameters, under minimal simplicity assumptions (without which  the regular version of TR is not expected to hold). This is among the most general\footnote{At least if one wants to stay in the realm of algebraic curves, i.e.~to keep a spectral curve that is defined by polynomial equations. In Section \ref{sec:exp} we address the case of $G$ being the product of a rational function and an exponential. In this case $X$ is no longer a polynomial but $X'/X$ still is.} result one could expect regarding TR and weighted Hurwitz numbers, completing the project initiated in~\cite{AlexandrovChapuyEynardHarnadJ2018,AlexandrovChapuyEynardHarnad2020} and concluding the efforts of dozens of papers in the last two decades addressing different particular cases. We note that this goal of research was also explicitly proposed recently  as part of Conjecture 4.4 in~\cite{BDBKS3}.

 More precisely, the main results of this paper are  the expression of the disk generating function $W_{0,1}$ (Theorem~\ref{thm:Discrat}),  the expression of the cylinder generating function $W_{0,2}$ (Theorem~\ref{thm:cylinder}),  the topological recursion for this model (Theorem~\ref{thm:main:result}) and the structure result it implies for the fixed-genus generating functions $W_{g,n}$ (Corollary~\ref{cor:structure}).

Beyond the results, maybe it is worth commenting about the structure of our paper and our proof. Our proofs for $W_{0,1}$ and $W_{0,2}$ when $G(z)$ is polynomial (Section~\ref{sec:constellations}) are purely combinatorial. They rely on the combinatorial expertise developed by the bijective school of map enumeration over the years. In particular, we use the combinatorial approach of Albenque and Bouttier \cite{AlbenqueBouttier2022}, originally developed to prove combinatorially Eynard's solution of the 2-matrix model \cite{Eynard2002} (which is $G(z)=1+z$ in our language) in the planar case. While the calculations of generating functions done in \cite{AlbenqueBouttier2022} to derive Eynard's solution from their bijections are not trivial, they are fortunately, easily promoted to the general polynomial case (as explained in Section \ref{sec:constellations}). 
Maybe surprisingly, we are able to deduce the general case from the polynomial one (Section~\ref{sec:rational}) using an algebraic approach.

Finally, let us comment on how we obtain TR. In most works on the subject, TR is proved by first obtaining strong structure results on the generating functions (showing that they have poles only at the branchpoints of the spectral curve, once expressed in the spectral variables), and on a number of additional equations (typically the linear and quadratic loop equations). In particular, the structure of generating functions is not deduced from TR, but proved before, or simultaneously. In our work, we proceed differently. We use the idea of \emph{deformation of spectral curves}, taken from Eynard and Orantin's original paper~\cite{EO} (but which we make mathematically explicit in our case). It proves TR for arbitrary times from the fact that it was already proved when times are equal to zero (and from our solution in genus $0$). In particular, we obtain the structure of generating functions as a byproduct of TR, and not the converse.

We hope that our paper, using techniques going from pure bijective combinatorics, algebra, to fine analysis of the topological recursion and to deformation of spectral curves, does justice to the idea that Hurwitz counting problems are an inexhaustible source of connections between those fields.

On the day the first version of this paper was released, another article by Bychkov, Dunin-Barkowski, Kazarian and Shadrin was made public \cite{BDBKS4}, containing results strongly related to ours, obtained with different techniques. See also Section~\ref{sec:exp}.

We now proceed with the formal definition of our main objects of study, before stating formally our main results.

\subsection{Definition of the model}

Throughout the paper we let $m, r\geq0$ with $m+r\geq 1$, and $M=m+r$. 
We let $\mathbf{u}=(u_0,\dots,u_{M-1})$ be a sequence of indeterminates which will serve as free parameters in our model. We consider the rational function
\begin{align}\label{eq:tildeG}
G(\cdot)  =  \frac{\prod_{i=0}^{m-1} ( 1+\cdot\, u_i)}{\prod_{j=m}^{M-1} (1+ \cdot\, u_j)} =  \frac{\prod_{i \in I} ( 1+\cdot\, u_i)}{\prod_{j\in J} (1+ \cdot\, u_j)},
\end{align}
where everywhere in the paper we will write $I=\{0,1,\dots,m-1\}$ and $J=\{m,m+1,\dots,M-1\}$. Note that $m=|I|$ and $r=|J|$ are respectively the number of parameters appearing in the numerator and denominator of $G$. We can have  $r=0$ or $m=0$, in which case $J$ and $I$ are empty, respectively (corresponding to $G$ being a polynomial or the inverse of a polynomial).

Given indeterminates $\mathbf{p}=(p_1,p_2,\dots)$, $\mathbf{q}=(q_1,q_2,\dots)$ and $t$, we form the formal power series
\begin{align}\label{eq:tau}
\tau \equiv \tau^G(\mathbf{p},\mathbf{q}; \mathbf{u}; t) \coloneqq
\sum_{\lambda} t^{|\lambda|} s_\lambda(\mathbf{p}) s_\lambda(\mathbf{q}) \prod_{\Box \in \lambda} G(c(\Box)),
\end{align}
where the sum is taken over all integer partitions, where the symbol $s_\lambda$ denotes the Schur function of index $\lambda$ expressed as a polynomial in its power sums variables, where the product is taken over all boxes of the partition $\lambda$ and $c(\Box)$ is the content of the box $\Box$ inside $\lambda$, see~\cite{Guay-PaquetHarnad2017}. This expression makes $\tau$ a hypergeometric Orlov--Sherbin tau function of the 2-Toda (and KP) hierarchy, see~\cite{OrlovScherbin2000,GouldenJackson2008}.

The coefficients of $\tau$, i.e.~the numbers $H(\lambda, \mu;\ell_I; \ell_J)=H(\lambda, \mu;\ell_0,\dots,\ell_{m-1}; \ell_m \dots,\ell_{M-1})$, defined by the series expansion in the $p_k$s and $q_l$s (here for a partition $\lambda$ we denote $p_\lambda = \prod_{k=1}^{\ell(\lambda)} p_{\lambda_k}$ the product of the $p_k$s corresponding to the parts of $\lambda$, and similarly for $q_\mu = \prod_{l=1}^{\ell(\mu)} q_{\mu_l}$, where $\ell(\lambda)$ and $\ell(\mu)$ denote the lengths of the partitions)
$$
\tau \equiv \tau(\mathbf{p},\mathbf{q}; u_0,\dots,u_{M-1}; t)
\coloneqq 1+ \sum_{d\geq 1}\frac{t^d}{d!} \sum_{\substack{\lambda,\mu\vdash d\\ \ell_0,\dots,\ell_{M-1}}}
H(\lambda, \mu;\ell_I; \ell_J)
p_\lambda q_\mu u_0^{\ell_0} \dots u_{M-1}^{\ell_{M-1}},
$$
are the \emph{unconnected weighted Hurwitz numbers} of weight-function $G$, whose combinatorial and geometric interpretation we recall at the end of the introduction.
Of even greater interest are the (connected) \emph{weighted Hurwitz numbers}, defined as the coefficients in the expansion of the logarithm:
$$\log \tau 
=\sum_{d\geq 1}\frac{t^d}{d!} \sum_{\substack{\lambda,\mu\vdash d\\ \ell_0,\dots,\ell_{M-1}}}
H^\circ(\lambda, \mu;\ell_I; \ell_J)
p_\lambda q_\mu u_0^{\ell_0} \dots u_{M-1}^{\ell_{M-1}}.
$$

For $g\geq0$, we define the generating function
\begin{align}\label{eq:Fg}
	{F}_g 
	&= [N^{2g-2}] \big(\mathrm{log} \tau \big) \Big|_{p_i\rightarrow p_i/N, q_i\mapsto q_i/N, u_i\mapsto u_i N}
	\\
	&=  \sum_{d\geq 1}\frac{t^d}{d!} \sum_{\substack{\lambda,\mu\vdash d\\ \ell_0,\dots,\ell_{M-1}\\ 2g=\sum_{k=0}^{M-1}\ell_k+2-\ell(\lambda)-\ell(\mu)}}
H^\circ(\lambda, \mu;\ell_0,\dots,\ell_{M-1})
p_\lambda q_\mu u_0^{\ell_0} \dots u_{M-1}^{\ell_{M-1}}.
\end{align}
In the interpretation as coverings (or maps) $F_g$ is the contribution of surfaces of genus $g$ to the connected function $\log \tau$ (this is an instance of the Euler--Riemann--Hurwitz formula).
We introduce the following ``rooting operator'', parametrised by a formal variable $x$,
$$
\nabla_x \coloneqq \sum_{i\geq 1} i x^{-i-1} \frac{\partial}{\partial p_i}.
$$
On constellations (see Section~\ref{sec:constellations} and also Appendix~\ref{sec:WHN}), 
we interpret the operator $\nabla_x$ as marking a face of given colour (white in our conventions) together with a ``root'' vertex inside it, and changing the weighting convention from $p_i$ to $x^{-1-i}$ if this face has degree $i$. Such a face is thought of as a ``boundary'' (as opposed to an ``internal face'' counted with weight $p_i$). 
The use of the inverse variable $1/x$ may be surprising to combinatorialists but is convenient for comparison with the literature in enumerative geometry. 

Throughout the paper, we fix two arbitrary positive integers $D_1,D_2$, which will play the role of maximum allowed degrees for internal faces of each colour.
For $n\geq 1$, and variables $x_1,\dots, x_n$, we introduce the generating function
\begin{align}\label{eq:Wgn}
W_{g,n}(x_1,\dots,x_n) \coloneqq  \Big( \nabla_{x_n} \dots \nabla_{x_1} F_g \Big) \Big|_{\substack{p_{i}=0, i> D_1\\q_{j}=0, j> D_2}}.
\end{align}
Anticipating on the next paragraph, one can think combinatorially of $W_{g,n}$ as the generating function of weighted Hurwitz numbers on a surface of genus $g$ with $n$ boundaries (but arbitrarily many internal faces, whose degrees are bounded by $D_1, D_2$, depending on their colour). We view it as a formal power series in $t$ whose coefficients are polynomials in the other variables, i.e.~an element of $\mathbb{Q}[\mathbf{p},\mathbf{q}, \mathbf{u}, \bar{\mathbf{x}}][[t]]$ in standard notation recalled below.
It depends implicitly on the parameters $m,r,\mathbf{u}$ (i.e.~on the function $G$) and on $D_1, D_2$, although we will not indicate it in the notation.

\medskip
{\noindent \bf Combinatorial interpretation of the generating functions.}
We briefly recall here the combinatorial interpretation of weighted Hurwitz numbers~\cite{Guay-PaquetHarnad2017}
(see Appendix~\ref{sec:WHN} for details).
Call a \emph{monotone run of transpositions} of length $\ell$,  an $\ell$-tuple of transpositions in $\mathfrak{S}_d$ of the form
$$
\rho = ((j_1,i_1), \dots, (j_\ell,i_\ell)),
$$
with $j_k<i_k$ for all $k$, and $i_1\leq \dots \leq i_\ell$.
Then the number $H(\lambda, \mu;\ell_I;\ell_J)$ defined above is equal to $(-1)^{\sum_{j\in J} \ell_j}$ times the  number of tuples 
$$(\sigma_{-2}, \sigma_{-1}, \sigma_0,\dots,\sigma_{m-1}, \rho^{(m)}, \dots, \rho^{(M-1)}),$$
where each $\sigma_k$ is a permutation in $\mathfrak{S}_d$ (here $d=|\lambda|=|\mu|$) with $\sigma_{-2}$ of cycle-type $\lambda$, $\sigma_{-1}$ of type $\mu$, with $\sigma_i$ having $d-\ell_i$ cycles for $i\in I$, where each $\rho^{(j)}$ for $j\in J$ is a monotone run of transpositions of length $\ell_j$, and where the total product is equal to the identity,
$$\sigma_{-2} \sigma_{-1} \sigma_0\dots \sigma_{m-1} \underline{\rho^{(m)}} \dots \underline{\rho^{(M-1)}} = \mathbf{1}_{\mathfrak{S}_d},$$
where the underlined notation $\underline{\rho}=(j_1,i_1) \dots (j_\ell,i_\ell)$ denotes the product of elements in the run.
We call such a tuple an \emph{$(m,r)$-factorisation}.

The connected number $H^\circ(\lambda, \mu;\ell_I;\ell_J)$ counts the same objects with the additional requirement that the subgroup of $\mathfrak{S}_d$ generated by all permutations $\sigma_k$ and all transpositions appearing in the runs $\rho^{(j)}$ acts transitively on $[d]$. It is well known that factorisations of the identity in permutations are in bijection with branched coverings of the sphere (see e.g.~\cite{LandoZvonkin2004}), and this property corresponds to the connectedness of the covering surface, hence the terminology.

All the generating functions defined above can thus be considered as generating functions of $(m,r)$-factorisations with different weights. In $\tau$ (or $\log\tau$ for the connected case) the weights $p_i$, $q_i$ are attached respectively to the cycles of length $i$ of $\sigma_{-2}$ and $\sigma_{-1}$, while the variable $u_i$ for $i\in I$ is attached to ``missing'' cycles of $\sigma_i$ compared to the identity (weight $u_i^{n-\ell}$ if $\sigma_i$ has $\ell$ cycles), and $u_j$ for $j\in J$ to transpositions in the monotone run $\rho^{(j)}$. The series $F_g$ is the contribution to $\log\tau$ of coverings whose underlying surface has genus $g$. In the function $W_{g,n}$ the same objects are considered, but $n$ cycles of $\sigma_{-2}$ have been distinguished, and carry a distinguished element, and receive a different weighting.

In the case $r=0$, these factorisations have a well-known combinatorial interpretation as certain embedded graphs called \emph{constellations}, which will recall and use in Section~\ref{sec:constellations}.
In this interpretation, cycles of $\sigma_{-2}$ and $\sigma_{-1}$ are regarded as \emph{white/black faces} respectively. The distinguished cycles/faces are regarded as \emph{boundaries} and the other ones as \emph{internal}. By analogy we will sometimes use this terminology also in the case of general  $(m,r)$-factorisations.


\medskip
{\noindent \bf Notation.}
We will use, for any variable $x$, the notation $\bar{x}\coloneqq 1/x$.
We use bold letters to denote (finite or infinite) sequences of variables, for example $\mathbf{p}=(p_1,p_2,\dots)$, $\mathbf{q}=(q_1,q_2,\dots)$, $\mathbf{u}=(u_0,\dots,u_{M-1})$, 
$\mathbf{x}=(x_1,x_2,\dots)$, $\bar{\mathbf{x}}=(\bar{x}_1,\bar{x}_2,\dots)$, etc.
We also write $\mathbf{u}_S = (u_i, i\in S)$, with $S$ a set of indices, and we will repeatedly use it with $\mathbf{u}_I$ and $\mathbf{u}_J$.

Throughout the paper we use the notations $[\cdot],(\cdot),[[\cdot]],((\cdot))$ respectively for polynomials, rational functions, formal power series, formal Laurent series, over a ring or field. We let $\mathbb{K}$ be an algebraic closure of $\mathbb{Q}(\mathbf{p},\mathbf{q},\mathbf{u})$ and we let $\mathbb{K}((t^*))$ be the (algebraically closed) field of Puiseux series (formal Laurent series in fractional powers of $t$) over this field.

We write $[x^k] f(x)$ for the coefficient of $x^k$ in the formal series $f(x)$ (which can be a formal series in $x$ or a formal series in another variable whose coefficients are Laurent polynomials in $x$). 
If $f(z)=\sum_{k\geq k_0} f_k z^k$ is a Laurent series in $z$, we write
$$[f(z)]^{<} \coloneqq \sum_{k=k_0}^{-1} f_k z^k.$$
This notation will always be used with the variable ``$z$'' in this paper. We use similar notation $[f(z)]^{>}$, $[f(z)]^{\leq}$, $[f(z)]^{\geq}$ with clear analogous definitions. If $g(z)=\sum_{k\geq k_0} g_k z^{-k}$ is a Laurent series in $z^{-1}$, we write 
\begin{equation*}
    \{g(z)\}^{\geq} \coloneqq \sum_{k=k_0}^{0} g_k z^{-k}.
\end{equation*}

\medskip
{\noindent \bf Acknowledgements.} We thank Marie Albenque and Jérémie Bouttier for sharing with us the results of the paper~\cite{AlbenqueBouttier2022}. We recently learned that Boris Bychkov, Petr Dunin-Barkowski, Maxim Kazarian and Sergey Shadrin are preparing the paper entitled {\it ``Symplectic duality for topological recursion''}, where they obtain results strongly related to ours with very different techniques. The latter was released \cite{BDBKS4} on the same day as this article. We are grateful to them for mentioning this project in preparation to us.

\section{Main results} \label{sec:MainResults}

\noindent The main results of this paper are:  the expression of the disk generating function $W_{0,1}$ (Theorem~\ref{thm:Discrat});  the expression of the cylinder generating function $W_{0,2}$ (Theorem~\ref{thm:cylinder});  the topological recursion for this model (Theorem~\ref{thm:main:result}) and the structure result it implies for the fixed-genus generating functions $W_{g,n}$ (Corollary~\ref{cor:structure}).

\subsection{Main result I: Disk generating function $W_{0,1}$ and spectral curve}

Our first result is an explicit algebraic parametrisation of the function $W_{0,1}$. In the case $G(z)=(1+ u z)$, what follows is Eynard's leading-order solution
of the 2-matrix model~\cite{Eynard2002} (which can be given different formulations, see~\cite{AlbenqueBouttier2022,BouttierCarrance}). The notation and the general form of the equations in this section follow closely the ones of \cite{AlbenqueBouttier2022}.

To state the result, we need to introduce the following quantities (which have nice combinatorial interpretation at least in the polynomial case\footnote{In the polynomial case, $A_k^{(c)}, B_k^{(c)}$ are \resp the generating functions of white-based and black-based elementary slices of increment $mk-1$ and initial colour $c$, see Section~\ref{sec:constellations}. This interpretation is essentially due to \cite{AlbenqueBouttier2022}.}).

\begin{defn}[Polynomials $A^{(c)}(z), B^{(c)}(z)$ and functions $A_k^{(c)}, B_k^{(c)}$]\label{def:WkBk}
For $c\in I \cup J$ we define the Laurent polynomials
$$A^{(c)}(z) = \sum_{k=0}^{D_2} A_k^{(c)} z^k \, ,  \ \ 
B^{(c)}(z) = 1 + \sum_{k=1}^{D_1} B^{(c)}_k z^{-k},$$ 
uniquely determined as elements of 
$\mathbb{Q}[\mathbf{p},\mathbf{q},\mathbf{u}][[t]][z]$ 
and 
$\mathbb{Q}[\mathbf{p},\mathbf{q},\mathbf{u}][[t]][z^{-1}]$ 
respectively, by the following system of equations:
\begin{align}
  \label{eq:defWratconv2}
    A^{(c)}(z) &= 1 + u_c \sum_{s=1}^{D_2} q_{s} t^s   \biggl\{z^s \frac{\prod_{i\in I}  B^{(i)}(z)^s}{\prod_{j\in J}  B^{(j)}(z)^s}\frac{1}{B^{(c)}(z)}\biggr\}^{\geq},\\
    \label{eq:defBratconv2}
    B^{(c)}(z) &= 1 + u_c \sum_{s=1}^{D_1} p_s \biggl[z^{-s} \frac{\prod_{i \in I} A^{(i)}(z)^s}{\prod_{j \in J} A^{(j)}(z)^s}\frac{1}{A^{(c)}(z)}\biggr]^{<},
\end{align}
where we recall that the nonnegative and negative parts $\{\cdot\}^\geq$, $[\cdot]^<$ are  with respect to a Laurent series in $z^{-1}$ and to a Laurent series in $z$, respectively.
\end{defn}
Note that the previous definition can be reformulated as a polynomial system of equations relating the series $A^{(c)}_k, B^{(c)}_k$ together from which all coefficients of these series can be computed recursively, order by order in $t$ (which justifies the definition).

Of great importance for the present paper is the series $Z(x)$ defined by the following equation
\begin{align} \label{eq:ZasExcursionsrat}
Z(x) = \bar{x} \frac{\prod_{i\in I} A^{(i)}(Z)}{\prod_{j\in J} A^{(j)}(Z)}.
\end{align}
The series $Z$ (element of $\mathbb{Q}[\mathbf{p},\mathbf{q},\mathbf{u},\bar{x}][[t]]$) has an expansion of the form
\begin{align}\label{eq:expansionZrat}
Z = \bar{x} + O(t).
\end{align}

\begin{defn}\label{def:Hrat}
	Define the Laurent polynomial 
	$H^{(c)}(z) \in \mathbb{Q}[\mathbf{p},\mathbf{q},\mathbf{u},\bar{\mathbf{u}}][[t]][z,z^{-1}]$ given by 
\begin{equation*}
H^{(c)}(z) \coloneqq \bar{u}_c\left( A^{(c)}(z) B^{(c)}(z) - 1 \right).
\end{equation*}
\end{defn}
\begin{rem}\label{rem:Hsymmetry}
The quantity $H^{(c)}$ involves no negative power of $u_c$, it is symmetric in the variables $\mathbf{u}_I$, symmetric in the $\mathbf{u}_J$, and moreover it is independent of the chosen value of $c \in I\cup J$. These properties are not trivial  and will be proved in Section~\ref{sec:Hsymmetry}.
\end{rem}

We have now defined all quantities needed to introduce our spectral curve.
\begin{defn}[Spectral curve of our model] \label{def:SpectralCurve}
	We consider the system of polynomial equations\footnote{When $\mathbf{p}, \mathbf{q}, \mathbf{u}$ are fixed complex numbers and $|t|$ is small enough, all generating functions $A_k^{(c)}, B_k^{(c)}$ converge and these polynomial equations are defined over complex numbers (see Remark~\ref{rem:convergenceDomain}). One can also consider these polynomial equations formally over the field $\mathbb{Q}(\mathbf{p}, \mathbf{q}, \mathbf{u})[[t]]$.}
	defined by
	\begin{align}\label{eq:spectralX}
		z X(z)&=  \frac{\prod_{i  \in I} A^{(i)}(z)}{\prod_{j \in J} A^{(j)}(z)}
		, \\ \label{eq:spectralY}
		X(z) Y(z) &= H^{(c)}(z),
\end{align}
	where we recall that the quantities $A^{(i)}$ and $H^{(c)}$ are given by Definition~\ref{def:WkBk} and~\ref{def:Hrat}, and where $c$ is any value in $I\cup J$.
\end{defn}

Note that the equation $X(Z)=x$ defines a unique power series $Z\in \mathbb{Q}[\mathbf{p},\mathbf{q},\mathbf{u},\bar{x}][[t]]$, 
which is nothing but the one we introduced in~\eqref{eq:ZasExcursionsrat}, with expansion~\eqref{eq:expansionZrat}. By substitution, the expression~\eqref{eq:spectralY} also defines a unique valid formal series $Y(Z(x))$ in 
$\mathbb{Q}[\mathbf{p},\mathbf{q},\mathbf{u},\bar{\mathbf{u}},x,\bar{x}]((t))$.

The following theorem is proved in \cite{AlbenqueBouttier2022} in the case $M=m=1$. Their approach is in fact the main tool we use to prove the theorem below in the polynomial case, i.e.~for $M=m\geq 1$. In fact (see Section~\ref{sec:constellations}), the proof in this case follows relatively easily for a reader familiar with \cite{AlbenqueBouttier2022} and combinatorics of paths, once understood that the combinatorial objects in the polynomial case can be encoded into constellations with two sets of face weights (Proposition \ref{thm:FactorizationsConstellations}).

\begin{thm}[Disk generating function $W_{0,1}$] \label{thm:Discrat}
	The disk generating function $W_{0,1}(x)$ is, up to an explicit shift, given by the parametrisation $(Y(z),X(z))$ given above. Namely, we have:
	$$W_{0,1}(x) + \sum_{k=1}^{D_1}p_kx^{k-1} = Y(Z(x)),$$
	in $\mathbb{Q}[\mathbf{p},\mathbf{q},\mathbf{u},x,\bar{x}][[t]]$.
\end{thm}

The following remark will be of great importance.
\begin{rem}[Artificial poles]\label{rem:artificialPoles}
Note that if $u_i = u_j$ for some $(i,j) \in I \times J$, then by definition we have $A^{(i)}=A^{(j)}$ and $B^{(i)}=B^{(j)}$, and the contribution of these quantities to equations~\eqref{eq:defWratconv2}-\eqref{eq:defBratconv2}-\eqref{eq:ZasExcursionsrat} simplifies. Consequently, if the expression of $G(\cdot)$ is artificially replaced by $G(\cdot ) \times \frac{1+u\cdot }{1+u \cdot}$, one obtains the same definition for all the quantities defined in these equations, and for the spectral curve. Therefore our spectral curve depends only on the rational function $G$, and not of a particular way to express it (with possible artificial extra poles).
\end{rem}

\subsection{Main result II: Cylinder generating function $W_{0,2}$}

Once formulated in the ``change of variables'' $x \leftrightarrow Z$, the cylinder generating function has the universal expression already encountered in many other models of enumerative geometry. Indeed we have:
\begin{thm}[Cylinder generating function $W_{0,2}$]\label{thm:cylinder}
	The cylinder generating function $W_{0,2}$ is
	$$W_{0,2}(x_1,x_2)= \frac{Z'(x_1)Z'(x_2)}{(Z(x_1)-Z(x_2))^2} - \frac{1}{(x_1-x_2)^2},$$
	in $\mathbb{Q}[\mathbf{p},\mathbf{q}, \mathbf{u}](x_1,x_2)[[t]]$.
\end{thm}

\subsection{Main result III: Topological recursion}

In abbreviated form, our main result states that the correlators $W_{g,n}(x_1,\dots,x_n)$ obey the Eynard--Orantin topological recursion, with spectral curve~\eqref{eq:spectralX}--\eqref{eq:spectralY}. To state this properly (Theorem~\ref{thm:main:result} in Section~\ref{sec:toprecresult}), we need some preliminary discussion.

\subsubsection{Discussion on ramification points and convergence}
\label{sec:tr:rampoints}
We start by a general remark about convergence of formal series.

\begin{rem}[Convergence]\label{rem:convergenceDomain}
For arbitrary sequences of complex numbers $\mathbf{p},\mathbf{q},\mathbf{u}$, when the variable $t$ is in a small enough neighbourhood of zero, all series $A_k^{(c)}, B_k^{(c)}$ are absolutely convergent (as follows from the well-founded system of algebraic equations~\eqref{eq:defWratconv2}-\eqref{eq:defBratconv2}).
In particular, the equations~\eqref{eq:spectralX}-\eqref{eq:spectralY} define a rational parametrisation $(X(z), Y(z))$ of a complex algebraic curve -- the \emph{spectral curve} from which we will define the topological recursion.

Moreover, the generating function
$W_{g,n}(x_1,\dots,x_n) \in \mathbb{Q}[\mathbf{p}, \mathbf{q}, \mathbf{u}, \bar{\mathbf{x}}][[t]]$ is absolutely convergent when $t,\bar{x}_1,\dots,\bar{x}_n$ are in a small-enough neighbourhood of $0 \in \mathbb{C}$, which can be taken independently of~$n$ and~$g$ (this is a consequence of the asymptotic growth  of weighted Hurwitz numbers of fixed genus, see Lemma~\ref{lemma:absoluteConvergence} in the appendix).
This neighbourhood can be chosen uniformly in $\mathbf{p},\mathbf{q}, \mathbf{u}$ in a compact set.

\end{rem}

In particular, the quantities
$X(z)$
and
$W_{g,n}(X(z_1),\dots,X(z_n))$
can be considered as analytic functions of $z$ and $z_1,\dots,z_n$ in a neighbourhood of zero (provided $t$ is in such a neighbourhood).
This justifies the following definition:
\begin{defn}[Differential forms $\omega_{g,n}$]\label{def:omegagn}
We define, for $\mathbf{p},\mathbf{q},\mathbf{u}$ in a compact set and $t$ in a neighbourhood $V$ of zero,
$\omega_{0,1}(z_1) = Y(z_1)\dd X(z_1)$, and for $g\in\mathbb{Z}_{\geq 0}$, $n\in \mathbb{Z}_{\geq 1}$ such that $(g,n)\neq(0,1)$:
\[
\omega_{g,n}(z_1,\dots,z_n) = W_{g,n}(X(z_1),\dots,X(z_n))\dd X(z_1)\dots \dd X(z_n) +\delta_{g,0}\delta_{n,2} \frac{\dd X(z_1)\,\dd X(z_2)}{(X(z_1)-X(z_2))^2}.
\]
We view these objects as differential forms for $z_i$ in a neighbourhood of zero, which are moreover analytic in such neighbourhood (with the exception of the diagonal $z_1=z_2$ in the case of $\omega_{0,2}$).
\end{defn}

We now need a discussion about (formal or complex) ramification points of the spectral curve. First consider the polynomial equation $X'(b)=0$, or equivalently (by logarithmic differentiation),
\begin{align}\label{eq:ramification points}
-b^{-1} + \sum_{i \in I} \frac{{A^{(i)}}'(b)}{A^{(i)}(b)} -\sum_{j \in J} \frac{{A^{(j)}}'(b)}{A^{(j)}(b)}=0. 
\end{align}

The ramification points of this equation near $t=0$ will be important for us:
\begin{defn}[Initial ramification points]\label{def:initialramification points}
An initial ramification point is a zero $a$ of the polynomial equation
\begin{align}\label{eq:ramification points0}
-a^{-1} + \sum_{i\in I} \frac{\sum_{k=1}^{D_2} k q_k a^{k-1}}{1/u_i + \sum_{k=1}^{D_2} q_k a^k}-\sum_{j\in J} \frac{\sum_{k=1}^{D_2} k q_k a^{k-1}}{1/u_j + \sum_{k=1}^{D_2} q_k a^k} =0,
\end{align}
or equivalently
\begin{align}\label{eq:ramification points0bis}
\frac{d}{da} \left(a^{-1} G(\sum_{k=1}^{D_2} q_k a^k) \right) =0.
\end{align}
We number these zeroes $a_1,a_2, \dots, a_{MD_2}$.
In the formal setting, we view the $a_i$ as (distinct) elements of~$\mathbb{K}$. In the complex setting, the $a_i$ are complex numbers.
\end{defn}

\begin{prop}[Formal ramification points]
The ramification point equation~\eqref{eq:ramification points} has $MD_2$ solutions $b^{\textup{formal}}_i(t), 1\leq i \leq MD_2$ which are distinct formal Laurent series, with expansion
\begin{align}\label{eq:devramification point}
b^{\textup{formal}}_i (t) = \frac{a_i}{t} +O(1) \ \ \in \mathbb{K}((t)),
\end{align}
where $a_i\in \mathbb{K}$ are (formal) initial ramification points.
\end{prop}
\begin{proof}
Observe that $A_k^{(c)} = u_c q_k t^k + O(t^{k+1})$ for $k\geq 1$ and $A_0^{(c)} = 1 + O(t)$.
The equation~\eqref{eq:ramification points} thus gives
\begin{align}\label{eq:ramification points0ter}
-b^{-1} + \sum_{i\in I} \frac{\sum_{k=1}^{D_2} k q_k b^{k-1} ( t^k + O(t^{k+1}) )}{\bar{u}_i + O(t) +  \sum_{k=1}^{D_2} q_k b^k (t^k + O(t^{k+1}) )}-\sum_{j\in J} \frac{\sum_{k=1}^{D_2} k q_k b^{k-1} ( t^k + O(t^{k+1}) )}{\bar{u}_j + O(t) +  \sum_{k=1}^{D_2} q_k b^k (t^k + O(t^{k+1}) )} =0.
\end{align}
Therefore for each $i\in\{1,\dots,MD_2\}$, a full expansion of $b^{\textup{formal}}_i(t)$ with initial seed $b^{\textup{formal}}_i (t) = \frac{a_i}{t}+\dots$ can be computed recursively.
\end{proof}

In the complex setting,~\eqref{eq:devramification point} and absolute convergence of Puiseux expansions of roots of polynomials also imply that, when $t$ goes to zero, the complex ramification points (complex zeroes of~\eqref{eq:ramification points}, which we denote by $b_1,\dots,b_{MD_2}$) behave as 
\begin{align}\label{eq:complexdevramification point}
b_i (t) \sim \frac{a_i}{t}, \ \ \ t\rightarrow 0.
\end{align}
In particular, if the $a_i$ are distinct, so are the $b_i$ for $t$ in a neighbourhood of zero.

\begin{rem}
For parameters $\mathbf{p},\, \mathbf{q},\,\mathbf{u}$ chosen generically, $Y'(b_i)\neq 0$ (equivalently $H^{(c)}{}^{'} (b_i)\neq 0$) for all $i\in\{1,\dots,MD_2\}$. Indeed, at first orders in $t$, we have $B_\ell^{(c)} = u_c p_\ell + O(t)$ for $\ell\geq 1$, so
\[\begin{split}
    u_c H^{(c)}(z) &= 
    \Big(1+O(t)+u_c\sum\limits_{k=1}^{D_2}q_k z^k(t^k+O(t^{k+1}))\Big)\Big(1+u_c\sum\limits_{\ell=1}^{D_1} p_{\ell}z^{-\ell}(1+O(t))\Big)-1.
\end{split}\]
Substituting $z=b_i(t)$ in the derivative of this expression (which is a Laurent polynomial in $z$), we have by~\eqref{eq:complexdevramification point} that  $zu_c {H^{(c)}}'(z) \Big|_{z=b_i(t)}$ is equal to
\[\begin{split}
    \Big(u_c\sum\limits_{k=1}^{D_2}kq_k a_i^{k} +O(t)\Big)\Big(1+u_c\sum\limits_{\ell=1}^{D_1} p_{\ell}a_i^{-\ell}t^\ell+O(t)\Big)
    -
        \Big(1+\sum\limits_{k=1}^{D_2}q_k a_i^k+O(t)\Big)\Big(u_c\sum\limits_{\ell=1}^{D_1} \ell p_{\ell}a_i^{-\ell}t^\ell +O(t)\Big).
\end{split}\]
When $t$ goes to zero this quantity tends to 
\begin{align}\label{eq:nozeroesY}
u_c\sum\limits_{k=1}^{D_2}kq_k a_i^{k},
\end{align}
which is indeed nonzero generically.
\end{rem}

We are now ready to clarify the analytic assumptions needed for our results:

\begin{defn}[Analytic assumptions]\label{def:assumptions}
By \emph{the analytic assumptions}, we will refer to the following properties:
\begin{itemize}
    \item 
The variables $\mathbf{p},\mathbf{q},\mathbf{u}$ are complex numbers living in an arbitrary (but fixed) compact of $\mathbb{C}$.  The variables in $\mathbf{u}$ are nonzero\footnote{This last assumption is not strictly necessary but makes intermediate equations easier to write.}.
\item The variables $t$, $z_1,\dots, z_n, \dots$ are constrained to a small enough neighbourhood $V$ of zero which is such that all series $A_k^{(c)}, B_k^{(c)}$, and $W_{g,n}(X(z_1),\dots,X(z_n))$ are absolutely convergent.
\item The variables $\mathbf{q}, \mathbf{u}$ are such that the initial ramification points $a_i$ (Definition~\ref{def:initialramification points}) are all distinct and different from zero. Moreover, they are such that $\sum_{k=1}^{D_2}kq_k a_i^{k}\neq 0$. This is true in particular if these variables are chosen generically.
\item The neighbourhood $V$ is reduced, if necessary, so that all ramification points $b_i(t)$ (Equation~\eqref{eq:complexdevramification point}) are distinct, for any $t$ in $V\setminus\{0\}$, and such that $Y'(b_i(t))\neq 0$ everywhere in this neighbourhood.
\end{itemize}
\end{defn}
Note that the last assumption is made possible by the last remark, and by the fact that we assume that $\sum_{k=1}^{D_2}kq_k a_i^{k}\neq 0$.

\subsubsection{Main statement, topological recursion}
\label{sec:toprecresult}

\begin{thm}[Main result -- topological recursion for $\omega_{g,n}$]\label{thm:main:result}
Under the analytic assumptions, the following is true.

For $g\in\mathbb{Z}_{\geq 0}$, $n\in \mathbb{Z}_{\geq 1}$ such that $2g-2+n\geq 1$, $\omega_{g,n}$ can be analytically continued to a meromorphic $n$-differential on $\mathbb{CP}^1$, with possible poles only 
%
at the ramification points $b_1,\dots,b_{MD_2}$.

They satisfy the topological recursion~\cite{EO} with spectral curve 
$$
\mathcal{S} =\Big(\mathbb{P}^1,\, X(z),\, Y(z),\, \omega_{0,2}(z_1,z_2) = \frac{\dd z_1 \dd z_2}{(z_1-z_2)^2} \Big).
$$
\end{thm}
We refer to Definition~\ref{def:differentials} below for the expanded statement of the topological recursion (the $\omega_{g,n}$ are denoted by $\tilde{\omega}_{g,n}$ there).

\medskip

Topological recursion implies strong structure results for the differentials $\omega_{g,n}$, see e.g.~\cite[Proposition 4.1]{Eynard04} or \cite[Theorem 3.7]{DOSS}.
For our most combinatorially inclined readers, we provide the following formal version. It vastly generalises structure results from~\cite{BenderCanfield1986, BenderCanfield1991, Gao1993, Chapuy2009, ChapuyFang, GouldenGuayPaquetNovak:polynomiality}. Note that it appeals for a combinatorial interpretation.

\begin{cor}[Structure of generating functions]
\label{cor:structure}
For $2g-2+n\geq 1$,  the formal series 
\[
W_{g,n}(X(z_1),\dots, X(z_n))X'(z_1) \dots X'(z_n) 
\]
can be written as a rational function; in each variable, the poles are located at the ramification points $b_i^{\textup{formal}}(t)$ and are of order at most $6g-4+2n$. More precisely, using the short hand notations $Y^{(k)}/Y'$, $X^{(k)}/X''$ for $Y^{(k)}(b_i^{\textup{formal}}(t))/Y'(b_i^{\textup{formal}}(t))$ and $ X^{(k)}(b_i^{\textup{formal}}(t))/X''(b_i^{\textup{formal}}(t))$ respectively, it belongs to the polynomial ring:
\[\mathbb{Q}\Big[\frac{1}{z_j-b_i^{\textup{formal}}(t)},\,b_i^{\textup{formal}}(t),\,\frac{1}{b_i^{\textup{formal}}(t)-b_{k}^{\textup{formal}}(t)},\, \frac{1}{Y'(b_i^{\textup{formal}}(t))},\, \frac{Y^{(k)}}{Y'},\,\frac{1}{X''(b_i^{\textup{formal}}(t))},\, \frac{X^{(k)}}{X''}\Big],\]
where for each $(g,n)$ a finite number of derivatives of $X$ and $Y$ at the ramification points contribute.
\end{cor}
\begin{proof}
The statement is proved by an elementary induction on $2g-2+n$ from the formula of topological recursion which involves residues at the ramification points (see Definition \ref{def:differentials}). Near the ramification point $b_i(t)$, the expansions of the local involution $\tilde{\sigma}_i(z)$ and $Y$ are encoded in coefficients $(\beta^i_k)_{k\geq 1}$ and $(t^{i}_k)_{k\geq 2}$ defined by:
\[\begin{split}
\tilde{\sigma}_i(z)-b_i(t) &= -(z-b_i(t))\Big(1+\sum\limits_{k\geq 2}\beta^{i}_{k}(z-b_i(t))^k\Big),\\
Y(z)-Y(\tilde{\sigma}_i(z))&= 2Y'(b_i(t))\Big(z-b_i(t)+\sum\limits_{k\geq 2}t^{i}_k(z-b_i(t))^k\Big).
\end{split}\]
Those coefficients belong to the following rings:
\[
\beta^{i}_{k}\in\mathbb{Q}\Big[\frac{X^{(j)}(b_i(t))}{X''(b_i(t))}\Big],\qquad t^{i}_k\in\mathbb{Q}\Big[\frac{Y^{(j)}(b_i(t))}{Y'(b_i(t))},\,\,\frac{X^{(j)}(b_i(t))}{X''(b_i(t))}\Big].
\]
From this the induction follows and one obtains the statement at the complex-analytic level. The statement at the formal level follows.
\end{proof}

\section{Deformation of spectral curves and proof of topological recursion}
\label{sec:toprecproof}

Our proof of the topological recursion (TR) relies on the technique of deformations of spectral curves, which is generically introduced in Eynard and Orantin's original paper~\cite{EO}.
We start from the fact that TR is already known when the variables $p_i$ are equal to zero, and from our knowledge of the spectral curve (and in particular of $W_{0,1}$,  $W_{0,2}$) for general $\mathbf{p}$. 
The idea is then to perform a certain ``Taylor expansion'' of the correlators near $\mathbf{p}=0$ (or more precisely near $\alpha=0$ for a certain rescaling parameter $\alpha$ introduced below). On the combinatorial side, this Taylor expansion is obtained through repeated insertions of internal faces via an appropriate operator (see Definition~\ref{def:chi}), while on the TR side, coefficients of this Taylor expansion are accessible \textit{via} iterated residues  (Definition~\ref{def:chitilde}). We then prove by induction that the coefficients on the two sides agree (Proposition~\ref{prop:egalite:taylor}).

We rely on several ideas of~\cite{EO} (although the technique of deformations was not introduced specifically to handle such situations in that reference) but we provide self-contained proofs. Moreover, to make our proofs more broadly accessible, we refrain from using the diagrammatic formalism of TR (which some experienced readers might find useful to interpret some of the calculations below).

\medskip

In this section we fix the parameters $\mathbf{p}$, $\mathbf{q}$, $\mathbf{u}$, and $t\neq 0$ 
under the analytic assumptions (Definition~\ref{def:assumptions}).
We also introduce a new parameter $\alpha\in \mathbb{C}$, $|\alpha|\leq 1$, and everywhere in this section we (implicitly) rescale all parameters $p_i$ by $\alpha$: 
$$p_i\to \alpha\, p_i.$$ 
This parameter will not be indicated in the notation, except in some cases to insist that it is equal to zero.



\subsection{The spectral curve at $\alpha=0$}

We now consider the model where $\alpha$ is equal to zero. In order to make the value of $\alpha$ explicit, the spectral curve is denoted by $X_0$, $Y_0$, and the generating series are denoted $X_{g,n}\coloneqq W_{g,n}\Big|_{\alpha=0}$. 
Note that Definition~\ref{def:WkBk} implies, for $\alpha=0$,
$$
B^{(c)}(z) = 1\quad \text{and}\quad A^{(c)}(z) = 1 + u_c \sum_{k=1}^{D_2} q_k t^k z^k.
$$
The spectral curve~\eqref{eq:spectralX}--\eqref{eq:spectralY} thus takes the form
\begin{align}\label{eq:spectralX0}
	z X_0(z)&=  \frac{\prod_{i  \in I} \left(1 + u_i\sum_{r=1}^{D_2} q_r t^r z^{r}\right)}{\prod_{j  \in J} \left(1 + u_j\sum_{r=1}^{D_2} q_r t^r z^{r}\right)}
=G(Q(t z)),
	\\
	\label{eq:spectralY0}
X_0(z) Y_0(z) &=
 \sum_{i=1}^{D_2}t^i  q_i z^{i}
	= Q(tz),
\end{align}
where $Q(z)\coloneqq\sum_{r= 1}^{D_2} q_{r}z^{r}$. 
Note that this is the spectral curve given in~\cite{BDBKS2} up to simple identifications of variables. Indeed, the curve given in that reference reads (we use hats for the quantities in that paper to avoid confusion with ours)
$$
\hat Z \hat X^{-1}  =  G(Q(\hat Z)), \ \  
\hat Y = Q (\hat Z), 
$$
(identifying our quantities $q_i, Q, G$ respectively with $y_i,\hat y, e^{\psi}$ in that reference).
Therefore the dictionary between the two papers is given by
$$
\hat Z = t Z_0 ,  \ \  \hat X = X_0^{-1},  \ \  \hat Y = X_0 Y_0, \  \ 1=t.
$$
(the variable $t$ can be reinserted by the scaling $q_r \mapsto t^r q_r$).
Note also that the initial branchpoints $a_1,\dots,a_{MD_2}$ defined above are precisely the zeroes of the equation $X_0'(a) =0$.

For $g\in \mathbb{Z}_{\geq 0}$, $n\in \mathbb{Z}_{\geq 1}$, define the $n$-differential 
\[\chi_{g,n}(z_1,\dots,z_n) = X_{g,n}(X_0(z_1),\dots,X_0(z_n))\dd X_0(z_1) \dots \dd X_0(z_n)+\delta_{g,0}\delta_{n,2}\frac{\dd X_0(z_1)\,\dd X_0(z_2)}{(X_0(z_1)-X_0(z_2))^2}.\] 
Those differentials are a priori defined only in a neighbourhood of zero. Then, we have:
\begin{prop}[{\cite[Theorem 5.3]{BDBKS2}}, Topological recursion at  $\alpha=0$]\label{thm:TR:time:zero}
For $g\in \mathbb{Z}_{\geq 0}$, $n\in\mathbb{Z}_{\geq 1}$, with $2g-2+n\geq 1$, $\chi_{g,n}(z_1,\dots,z_n)$ can be analytically continued to a meromorphic differential in each variable $z_i\in \mathbb{C}$, with poles only at the  ramification points
$a_1/t,\dots,a_{MD_2}/t$.
These differentials $(\chi_{g,n})_{g,n}$ satisfy the topological recursion for the spectral curve
\[
\mathcal{S}_0 = \Bigg(\mathbb{P}^1 ,\, X_0(z),\, Y_0(z),\, \chi_{0,2}(z_1,z_2)= \frac{\dd z_1 \dd z_2}{(z_1-z_2)^2}\Bigg).
\]
\end{prop}
Namely, for $g\in\mathbb{Z}_{\geq 0}$, $n\in\mathbb{Z}_{\geq 1}$, with $2g-2+n\geq 1$
we have, with $L=\{z_2,\dots,z_n\}$, 
\begin{equation}\label{eq:TR:time:zero}
\chi_{g,n}(z_1,L) = \sum\limits_{i=1}^{M\, D_2} \underset{z=a_i/t}{\Res}\, K_i(z_1,z)\Big(\chi_{g-1,n+1}(z,\sigma_i(z),L)+\sum\limits_{\substack{h+h'=g \\ C\sqcup C'=L}}^{'}\chi_{h,1+|C|}(z,C)\chi_{h',1+|C'|}(\sigma_i(z),C')\Big),
\end{equation}
where 
\begin{itemize}
	\item $\sigma_i$ denotes the local involution around the branchpoint $a_i/t$:
	\[
	X_0(\sigma_i(z))=X_0(z),\qquad \sigma_i(a_i/t)=a_i/t,\qquad \sigma_i(z)\neq z\,\, \textup{for }z\neq a_i/t;
	\]
	\item the recursion kernel is denoted by $K_i(z_1,z)$ and reads
	\[
	K_i(z_1,z) \coloneqq \frac{1}{2}\frac{\int_{w=\sigma_i(z)}^{z}\chi_{0,2}(z_1,w)}{\chi_{0,1}(z)-\chi_{0,1}(\sigma_i(z))};
	\]
	\item $\sum^{'}$ means that the terms $(h,C)=(0,\emptyset),\, (g,L)$ are excluded from the sum.
\end{itemize}

From now on, we use the notation $\tilde{a}_i \coloneqq a_i/t$ for $i\in\{1,\dots,MD_2 \}$.

\subsection{Taylor expansions of differentials}
In Definition~\ref{def:omegagn} we have defined the families of differentials $(\omega_{g,n})_{\substack{g\geq 0\\ n\geq 1}}$ from the combinatorial series $W_{g,n}$. They are a priori defined only for $z_i$ in a neighbourhood of zero.
We now define another family $(\tilde{\omega}_{g,n})_{\substack{g\geq 0\\ n\geq 1}}$ through the topological recursion.
\begin{defn}\label{def:differentials}
Let $\tilde{\omega}_{0,1}(z)=Y(z)\dd X(z)$.
Let $\tilde{\mathcal{S}}$ be the spectral curve:
\[
\tilde{\mathcal{S}} = \Big(\mathbb{P}^1,\,X(z),\,Y(z),\, \tilde{\omega}_{0,2}(z_1,z_2)=\frac{\dd z_1 \dd z_2}{(z_1-z_2)^2}\Big).
\]
The $M D_2$ ramification points of $X$ are the roots of~\eqref{eq:ramification points} (with expansion~\eqref{eq:complexdevramification point}) and denoted $ b_1,\dots, b_{M D_2}$. 
\begin{itemize}
    \item The local involution near $b_i$ is denoted $\tilde{\sigma}_i$:
    \[
    X(\tilde{\sigma}_i(z))=X(z),\qquad \tilde{\sigma}_i(b_i)=b_i, \qquad \tilde{\sigma}_i(z)\neq z, \ \textup{ for }z\neq b_i.
    \]
    \item The recursion kernel near $b_i$ is denoted by $\tilde{K}_i(z_1,z)$:
    \[
    \tilde{K}_i(z_1,z)\coloneqq \frac{1}{2}\frac{\int_{w=\tilde{\sigma}_i(z)}^{z}\tilde{\omega}_{0,2}(z_1,w)}{\tilde{\omega}_{0,1}(z)-\tilde{\omega}_{0,1}(\tilde{\sigma}_i(z))}.
    \]
\end{itemize}

For $2g-2+n\geq 1$, we define the differentials $\tilde{\omega}_{g,n}$ as the meromorphic differentials built by running topological recursion of~\cite{EO} on the spectral curve $\tilde{\mathcal{S}}$, namely,
\begin{equation}\label{eq:TR:generic:alpha}
\tilde{\omega}_{g,n}(z_1,L) = \sum\limits_{i=1}^{M\, D_2} \underset{z=b_i}{\Res}\, \tilde{K}_i(z_1,z)\Big(\tilde{\omega}_{g-1,n+1}(z,\tilde{\sigma}_i(z),L)+\sum\limits_{\substack{h+h'=g \\ C\sqcup C'=L}}^{'}\tilde{\omega}_{h,1+|C|}(z,C)\tilde{\omega}_{h',1+|C'|}(\tilde{\sigma}_i(z),C')\Big),
\end{equation}
where $L=\{z_2,\dots,z_n\}$ and the notation for $\sum^{'}$ is as above.
\end{defn}

From Theorems \ref{thm:Discrat} and \ref{thm:cylinder}, we have $\omega_{0,1}(z) = \tilde{\omega}_{0,1}(z)$ and 
$\omega_{0,2}(z_1,z_2) = \frac{\dd z_1\, \dd z_2}{(z_1-z_2)^2}=\tilde{\omega}_{0,2}(z_1,z_2)$ in a neighbourhood of zero. This can be used to analytically continue $\omega_{0,1}$ and $\omega_{0,2}$, which gives the equality between spectral curves $\mathcal{S} = \tilde{\mathcal{S}}$. Note that $\omega_{0,2}$ is the fundamental bidifferential of the second kind on the Riemann sphere, typically called Bergman kernel in the topological recursion literature.

The differentials $\tilde{\omega}_{g,n}$, for $2g+n>2$, are defined through the formula of the topological recursion and therefore are defined globally (for all complex $z_i$ avoiding branchpoints). The rest of Section~\ref{sec:toprecproof} is dedicated to prove that $\omega_{g,n}(z_1,\dots,z_n)=\tilde{\omega}_{g,n}(z_1,\dots,z_n)$, when $z_i$ are close to zero. This will prove Theorem \ref{thm:main:result}. 

\medskip

\begin{defn}[Insertion operator $\Gamma_{x}$]
Let $f(x)$ be a series in $\bar{x}$ which is absolutely convergent in a neighbourhood of $\bar{x}=0$.
We define 
\[
\Gamma_x f(x)= \sum\limits_{k=1}^{D_1} \frac{p_k}{k}\, [x^{-k-1}] f(x).
\]
%
Equivalently,
\begin{equation}\label{eq:insertion:operator:res}
\Gamma_{x} f(x)= \sum\limits_{k=1}^{D_1}\frac{p_k}{k}\, \underset{z=0}{\Res}\, X_0(z)^k\, f\big(X_0(z)\big)\, \dd X_0(z).
\end{equation}
\end{defn}
The second expression of $\Gamma_x f(x)$ comes from the fact that $X_0(z)^{-1} =z + O(z^2)$ and from the change of variables $x=X_0(z)$ in Cauchy's integral formula (which is made possible by absolute convergence).

\begin{rem}
When $\Gamma_x$ acts on $X_{g,n+1}(x_1,\dots,x_n,x)$, one obtains the generating series of $(m,r)$-factorisations of genus $g$ with $n$ boundaries and one internal white face (carrying the weight $p_j$ recording its degree). Therefore, $\Gamma_x$ transforms a boundary into an internal white face, hence its name (it ``inserts'' an internal face).
\end{rem}

We now introduce auxiliary sets of differentials $(\chi^{\ell}_{g,n})_{\substack{g,\ell\geq 0\\ n\geq 1}}$ and $(\tilde{\chi}^{\ell}_{g,n})_{\substack{g,\ell\geq 0\\ n\geq 1}}$ on $\mathbb{P}^1$ that will be related to the Taylor expansions (in $\alpha$) of the differentials $\omega_{g,n}$, $\tilde{\omega}_{g,n}$ in Propositions \ref{prop:taylor:combi}, \ref{prop:taylor:TR} respectively.
\begin{defn}[Differentials $\chi^{\ell}_{g,n}$]\label{def:chi}
For $g\geq 0$, $n\geq 1$, $\ell\geq 0$,
we define
\[\begin{split}
\chi^{\ell}_{g,n}(z_1,\dots,z_n) \coloneqq &\,\dd X_0(z_1)\dots\dd X_0(z_n)\, \Gamma_{x_1}\dots\Gamma_{x_{\ell}} X_{g,n+\ell}\big(X_0(z_1),\dots,X_0(z_n),x_1,\dots,x_{\ell}\big)\\
& +\delta_{g,0}\delta_{n,1}\delta_{\ell,1} \dd X_0(z_1) \Gamma_{x_1} \frac{1}{(X_{0}(z_1)-x_1)^2} + \delta_{g,0}\delta_{n,2}\delta_{\ell,0} \frac{\dd X_0(z_1)\dd X_0(z_2)}{(X_0(z_1)-X_0(z_2))^2}.
\end{split}
\]
\end{defn}
\begin{defn}[Differentials $\tilde{\chi}^{\ell}_{g,n}$]\label{def:chitilde}
We define $(\tilde{\chi}^{\ell}_{g,n})_{\substack{g,\ell\geq 0\\ n\geq 1}}$ by the following formulas. First set 
\[
\tilde{\chi}^{\ell}_{0,1}(z_1)\coloneqq \chi^{\ell}_{0,1}(z_1),\qquad \tilde{\chi}^{\ell}_{0,2}(z_1,z_2)\coloneqq \chi^{\ell}_{0,2}(z_1,z_2), \qquad \tilde{\chi}^{0}_{g,n}(z_1,\dots,z_n)\coloneqq \chi_{g,n}(z_1,\dots,z_n).
\]
Then, for $2g-2+n\geq 1$ and $\ell\geq 1$, use the following recursive definition -- the recursion being on $(2g-2+n,\ell)$ with the lexicographic order:
\begin{equation}\label{eq:chi:TR}
\begin{split}
    \tilde{\chi}_{g,n}^{\ell}(z_1,L) = \sum\limits_{i=1}^{M\, D_2} \underset{z= \tilde{a}_i}{\Res}\, K_i(z_1&,z) \Bigg(\sum\limits_{k=1}^{\ell}{\ell \choose k}\big(\tilde{\chi}^{k}_{0,1}(z)\, \tilde{\chi}^{\ell-k}_{g,n}(\sigma_i(z),L)+\tilde{\chi}^{k}_{0,1}(\sigma_i(z)\, \tilde{\chi}^{\ell-k}_{g,n}(z,L)\big)\\ 
    &+\tilde{\chi}^{\ell}_{g-1,n+1}(z,\sigma_i(z),L) + \sum\limits_{\substack{h+h'=g\\ C\sqcup C'=L\\ k+k'=\ell}}^{'}{\ell \choose k} \tilde{\chi}_{h,1+|C|}^{k}(z,C) \tilde{\chi}_{h',1+|C'|}^{k'}(\sigma_i(z),C')\Bigg),
\end{split}
\end{equation}
where $L=\{z_2,\dots,z_n\}$.
\end{defn}

\begin{rem}
Being constructed from the case $\alpha=0$, for which the  topological recursion is known (Proposition~\ref{thm:TR:time:zero}) the differentials $\chi^{\ell}_{g,n}$ and $\tilde{\chi}^{\ell}_{g,n}$ are defined globally for the parameters $z_i$ (not only in a neighbourhood of zero). In particular they are defined near the $\tilde{a}_i$ and the recursive definition makes sense.
\end{rem}

To state the Taylor expansions, we first need a lemma.
\begin{lem}[Local inversion of $X(z)$ and $X_0(z)$]\label{lemma:inversion}
The function $z\mapsto X(z)$ defines a bijection from a pointed neighbouhood of zero to a pointed neighbourhood of infinity. The same is true for $z\mapsto X_0(z)$. 
\end{lem}
\begin{proof}
The series $Z(x) \in \mathbb{C}[[\bar x]]$ is absolutely convergent and has an expansion of the form $Z(x) = \bar{x} + O(\bar{x}^{2})$,
and the result about $X(z)$ follows. The argument for $X_0(z)$ is precisely the same.
\end{proof}
\begin{prop}[Combinatorial expansion]\label{prop:taylor:combi}
For small enough $z_1,\dots,z_n$ and $\alpha$:
\[
\omega_{g,n}(\tilde{z}_1,\dots,\tilde{z}_n) = \sum\limits_{\ell\geq 0} \frac{\alpha^{\ell}}{\ell!}\, \chi^{\ell}_{g,n}(z_1,\dots,z_n),
\]
where $z_i$ and $\tilde{z}_i$ are related by the equation $X(\tilde{z}_i)= X_{0}(z_i)$. 
\end{prop}
Note that the relation between $z_i$ and $\tilde{z}_i$ makes sense by Lemma~\ref{lemma:inversion}.
\begin{proof}
Let $\tilde{x}_1,\dots,\tilde{x}_n \in \mathbb{P}^1$ near $\infty$ such that $W_{g,n}(\tilde{x}_1,\dots,\tilde{x}_n)$ is absolutely converging. Since $W_{g,n}$ is the generating series of $(m,r)$-factorisations with internal faces, it is equal to the following sum:
\[
W_{g,n}(\tilde{x}_1,\dots,\tilde{x}_n)=\sum\limits_{\ell\geq 0} W^{\ell}_{g,n}(\tilde{x}_1,\dots,\tilde{x}_n),
\]
where $W^{\ell}_{g,n}(\tilde{x}_1,\dots,\tilde{x}_n)$ is the generating series of $(m,r)$-factorisations with $\ell$ internal white faces. From the definition of the insertion operator $\Gamma_{x}$, we get:
\[
W^{\ell}_{g,n}(\tilde{x}_1,\dots,\tilde{x}_n)=\frac{\alpha^{\ell}}{\ell!}\, \Gamma_{x_1}\dots\Gamma_{x_{\ell}} X_{g,n+\ell}(\tilde{x_1},\dots,\tilde{x}_n,x_1,\dots,x_{\ell}).
\]
The factor $\frac{1}{\ell!}$ is there to compensate the order in which internal cycles are inserted.

Let us take $\tilde{z}_1,\dots,\tilde{z}_n\in \mathbb{P}^1$ and $z_1,\dots,z_n\in \mathbb{P}^1$ near $0$ so that $X(\tilde{z}_i)=X_{0}(z_i)=\tilde{x}_i$. For $(g,n)\neq (0,1),\, (0,2)$, we get
\[
\begin{split}
    \omega_{g,n}(\tilde{z}_1,\dots,\tilde{z}_n)=& W_{g,n}(\tilde{x}_1,\dots,\tilde{x}_n) \dd X(\tilde{z}_1)\dots \dd X(\tilde{z}_n) \\
    =& W_{g,n}(\tilde{x}_1,\dots,\tilde{x}_n) \dd X_0(z_1)\dots \dd X_0(z_n)\\
    =& \sum\limits_{\ell\geq 0}\dd X_0(z_1)\dots \dd X_0(z_n) \frac{\alpha^{\ell}}{\ell!}\, \Gamma_{x_1}\dots\Gamma_{x_{\ell}} X_{g,n+\ell}(X_0(z_1),\dots,X_0(z_n),x_1,\dots,x_{\ell}) \\
    =& \sum\limits_{\ell \geq 0} \frac{\alpha^{\ell}}{\ell!}\, \chi^{\ell}_{g,n}(z_1,\dots,z_n).
\end{split}
\]
For $(g,n)=(0,1)$, we have on one hand:
\[\begin{split}
\omega_{0,1}(\tilde{z}_1)&=\Big(W_{0,1}(\tilde{x_1})+\alpha \sum\limits_{k=1}^{D_1}p_k \tilde{x}_1^{k-1}\Big)\dd X(\tilde{z}_1) \\
&= \Big(W_{0,1}(X_0(z_1))+\alpha \sum\limits_{k=1}^{D_1}p_k X_0(z_1)^{k-1}\Big)\dd X_0(z_1).
\end{split}\]
On the other hand, using that $\Gamma_{x_1}\frac{1}{(\tilde{x}_1-x_1)^2}= \sum\limits_{k=1}^{D_1}p_k \, \tilde{x}_1^{k-1}$:
\[
\begin{split}
    \sum\limits_{\ell\geq 0}\frac{\alpha^{\ell}}{\ell!}\, \chi_{0,1}^{\ell}(z_1)=&\Big( \sum\limits_{\ell\geq 0} \frac{\alpha^{\ell}}{\ell!}\, \Gamma_{x_1}\dots \Gamma_{x_{\ell}} X_{0,1+\ell}(\tilde{x}_1,x_1,\dots,x_{\ell}) +\alpha\,\Gamma_{x_1}\frac{1}{(X_0(z_1)-x_1)^2}\Big)\dd X_0(z_1)\\
    =&\Big(W_{0,1}(X_0(z_1))+\alpha \sum\limits_{k=1}^{D_1}p_k X_0(z_1)^{k-1}\Big)\dd X_0(z_1),
\end{split}
\]
so we get $\omega_{0,1}(\tilde{z}_1)= \sum\limits_{\ell\geq 0}\frac{\alpha^{\ell}}{\ell!}\, \chi_{0,1}^{\ell}(z_1)$.\\
Last, for $(g,n)=(0,2)$:
\[\begin{split}
\omega_{0,2}(\tilde{z}_1,\tilde{z}_2)&= \Big(W_{0,2}(\tilde{x}_1,\tilde{x}_2)+\frac{1}{(\tilde{x}_1-\tilde{x}_2)^2}\Big)\dd X(\tilde{z}_1) \dd X(\tilde{z}_2)\\ &=\Big(W_{0,2}(X_0(z_1),X_0(z_2))+\frac{1}{(X_0(z_1)-X_0(z_2))^2}\Big)\dd X_0(z_1) \dd X_0(z_2)
\end{split}
\]
and
\[\begin{split}
\sum\limits_{\ell\geq 0} \frac{\alpha^{\ell}}{\ell!}\, \chi_{0,2}^{\ell}(z_1,z_2) =& \Big(\sum\limits_{\ell\geq 0} \frac{\alpha^{\ell}}{\ell!} \Gamma_{x_1}\dots \Gamma_{x_{\ell}}X_{0,2+\ell}(X_0(z_1),X_0(z_2),x_1,\dots,x_{\ell})\\ & + \frac{1}{(X_0(z_1)-X_0(z_2))^2}\Big)\dd X_0(z_1) \dd X_0(z_2)\\
=&\Big(W_{0,2}(X_0(z_1),X_0(z_2))+\frac{1}{(X_0(z_1)-X_0(z_2))^2}\Big)\dd X_0(z_1) \dd X_0(z_2),
\end{split}\]
so we have $\omega_{0,2}(\tilde{z}_1,\tilde{z}_2)= \sum\limits_{\ell\geq 0}\frac{\alpha^{\ell}}{\ell!}\, \chi_{0,2}^{\ell}(z_1,z_2)$.
\end{proof}

\begin{prop}[$\alpha$-Expansion of topological recursion]\label{prop:taylor:TR}
For small enough $\alpha$, we have: 
\[
\tilde{\omega}_{g,n}(\tilde{z}_1,\dots,\tilde{z}_n) = \sum\limits_{\ell\geq 0} \frac{\alpha^{\ell}}{\ell!}\, \tilde{\chi}^{\ell}_{g,n}(z_1,\dots,z_n),
\]
where $z_i$ and $\tilde{z}_i$ are related by the equation $X(\tilde{z}_i)= X_{0}(z_i)$. 
\end{prop}

In order to prove this proposition, we need the following lemmas, which first require a definition.
In this definition (and there only), in order to make the dependence on $\alpha$ explicit, we use the notation $X(\alpha;z)$ and $b_j(\alpha)$  for $X$ and $b_j$, respectively.
\begin{defn}[Differentiation at $X$ fixed] \label{def:differentiation}
Let $f(\alpha;z_1,\dots,z_n)$ be a function or a differential, depending on $\alpha$ and variables $z_1,\dots,z_n$. Let $\alpha\in\mathbb{C}$ and $z_1,\dots, z_n \in \mathbb{CP}^1$ such that $z_i\neq b_j(\alpha)$ for all $i\in\{1,\dots,n\}$, $j\in\{1,\dots,MD_2\}$. For $i\in\{1,\dots,n\}$, let $\phi_{\alpha;z_i}:\, \mathbb{CP}^1 \to \mathbb{CP}^1$ be the differentiable function defined in a neighbourhood of $\alpha$ by:
\[
\phi_{\alpha;z_i}(\alpha)=z_i,\qquad X(\alpha';\phi_{\alpha;z_i}(\alpha'))=X(\alpha;z_i) .
\]
We define $\mathbf{D}$ to be the operator of differentiation with respect to $\alpha$ at $X$ fixed:
\begin{equation}
\mathbf{D}f(\alpha;z_1,\dots,z_n)=\frac{\partial}{\partial \alpha'}f(\alpha'; \phi_{\alpha;z_1}(\alpha'),\dots,\phi_{\alpha;z_n}(\alpha'))\big|_{\alpha'=\alpha}.
\end{equation}
\end{defn}
To justify the existence of $\phi_{\alpha;z_i}$ for $z_i\neq b_j(\alpha)$ in this definition, note that outside of the branchpoints, $X$ is a local coordinate (a statement which generalises Lemma \ref{lemma:inversion}). 
\begin{rem}\label{rem:derivation:insertion}
On the generating function $W_{0,1}(X(z_1))$, the derivation operator $\mathbf{D}$ acts as marking an internal cycle of an $(m,r)$-factorisation. Equivalently, $\mathbf{D}W_{0,1}(X(z_1))$ is obtained by applying the insertion operator to $W_{0,2}(X(z_1),x_2)$:
\[
\mathbf{D}W_{0,1}(X(z_1))= \Gamma_{x_2}W_{0,2}(X(z_1),x_2).
\]
We also have $\mathbf{D}\Big(\alpha\sum_{k=1}^{D_1}p_k\, X(z_1)^{k-1}\Big)= \Gamma_{x_2}1/(X(z_1)-x_2)^2 $, so:
\[
\mathbf{D}Y(z_1) = \Gamma_{x_2} \Big(W_{0,2}(X(z_1),x_2)+\frac{1}{(X(z_1)-x_2)^2}\Big).
\]
\end{rem}

\begin{lem}[Rauch variational formula]\label{lem:Rauch} 
Suppose that $\mathbf{D}\tilde{\omega}_{0,1}(z)/\dd Y(z)$ does not have poles at ramification points. The variation of the Bergman kernel $\tilde{\omega}_{0,2}$ is given by
\begin{equation}\label{eq:deformation:Bergman}
\mathbf{D} \tilde{\omega}_{0,2}(z_1,z_2)=-\sum\limits_{j=1}^{M\, D_2} \underset{v= b_j}{\Res}\, \frac{\mathbf{D}\tilde{\omega}_{0,1}(v)\,\tilde{\omega}_{0,2}(v,z_1)\,\tilde{\omega}_{0,2}(v,z_2)}{\dd X(v)\, \dd Y(v)}.
\end{equation}
\end{lem}

This result is based on the original work of Rauch \cite{Rauch} (see also \cite[p. 57]{Fay}) and appeared in the form of Lemma~\ref{lem:Rauch} in the context of topological recursion already in \cite[equation (5-4)]{EO}. A complete proof covering the case we need here is given in~\cite[Theorem 4.4]{Bertola--Korotkin19} (see also \cite[Proposition 7.1]{Baraglia--Huang17}, \cite[Lemma A.1]{Chaimanowong--Norbury--Swaddle--Tavakol20} for different proofs in various settings).

In our case, it is true that $\mathbf{D}\tilde{\omega}_{0,1}(z)/\dd Y(z)$ has no pole at ramification points, indeed,
\begin{itemize}
    \item from the analytic assumptions (Definitions \ref{def:assumptions}) $Y'$ does not vanish at ramification points;
    \item from Remark \ref{rem:derivation:insertion} and Equation \eqref{eq:insertion:operator:res}:
\[\mathbf{D}\tilde{\omega}_{0,1}(z) = \sum\limits_{k=1}^{D_1} \frac{p_k}{k} \, \underset{z'=0}{\Res}\, \tilde{\omega}_{0,2}(z,z') \, X(z')^k = \dd z \sum\limits_{k=1}^{D_1}\frac{p_k}{k!} \left.\frac{\dd^{k-1}}{\dd z'^{k-1}} \frac{G\big(Q(tz')\big)^k}{(z-z')^2}\right|_{z'=0}.\]
Since the ramification points are supposed to be different from zero, all the terms of the sum are finite when $z=b_i(t)$, so $\mathbf{D}\tilde{\omega}_{0,1}(z)$ does not have any pole at the ramification points. 
\end{itemize}
Therefore the conclusion~\eqref{eq:deformation:Bergman} of the last lemma holds.

\begin{lem}\label{lem:deformation:kernel}
Let $f$ be a symmetric bidifferential on $\mathbb{CP}^1$. Then, for all $i\in\{1,\dots,M D_2\}$:
\begin{equation}\label{eq:deformation:kernel}
\begin{split}
\underset{z=b_i}{\Res}\,\mathbf{D}\,\Big( \tilde{K}_i(z_1&,z) f(z,\tilde{\sigma}_i(z))\Big) = \underset{z= b_i}{\Res}\,\tilde{K}_{i}(z_1,z) \mathbf{D} f(z,\tilde{\sigma}_i(z)) \\
&+\sum\limits_{j=1}^{M\, D_2} \underset{v= b_j}{\Res}\,\underset{z= b_i}{\Res}\, \tilde{K}_j(z_1,v)\Big(\mathbf{D}\big(\tilde{\omega}_{0,1}(v) \big) \tilde{K}_{i}(\tilde{\sigma}_{j}(v),z)+\mathbf{D}\big(\tilde{\omega}_{0,1}(\tilde{\sigma}_j(v))\big) \tilde{K}_{i}(v,z)\Big) f(z,\tilde{\sigma}_i(z)).
\end{split}
\end{equation}
\end{lem}
\begin{proof}
The proof follows the same lines as in \cite[Lemma 5.1]{EO}. We have:
\begin{equation}\label{eq:variation:kernel}\begin{split}
    \mathbf{D}\, \Big( \tilde{K}_i(z_1,z) f(z,\tilde{\sigma}_i(z))\Big) &= \tilde{K}_i(z_1,z) \, \mathbf{D}\, f(z,\tilde{\sigma}_i(z)) -\tilde{K}_i(z_1,z)\frac{\mathbf{D}\tilde{\omega}_{0,1}(z)-\mathbf{D}\tilde{\omega}_{0,1}(\tilde{\sigma}_i(z))}{\tilde{\omega}_{0,1}(z)-\tilde{\omega}_{0,1}(\tilde{\sigma}_i(z))} f(z,\tilde{\sigma}_i(z)) \\
    & + \frac{1}{2}\frac{f(z,\tilde{\sigma}_i(z))}{\tilde{\omega}_{0,1}(z)-\tilde{\omega}_{0,1}(\tilde{\sigma}_i(z))} \mathbf{D}\Big(\int_{w=\tilde{\sigma}_i(z)}^{z}\tilde{\omega}_{0,2}(z_1,w)\Big). 
\end{split}\end{equation}
To treat the third term, we first re-write \eqref{eq:deformation:Bergman} as
\begin{align}
    \mathbf{D} \tilde{\omega}_{0,2}(z_1,w)& =-\sum\limits_{j=1}^{M\, D_2} \underset{v= b_j}{\Res}\, \frac{\mathbf{D}(\tilde{\omega}_{0,1}(v))\,\tilde{\omega}_{0,2}(v,z_1)\,\tilde{\omega}_{0,2}(v,z_2)}{\dd X(v)\, \dd Y(v)}\label{eq:def_Berg}\\
    &=\sum\limits_{j=1}^{M\, D_2} \underset{v= b_j}{\Res}\, \frac{\mathbf{D}(\tilde{\omega}_{0,1}(v))\,\tilde{\omega}_{0,2}(v,z_1)\,\tilde{\omega}_{0,2}(\tilde{\sigma}_j(v),w)}{\dd X(v)\, \dd Y(v)}\nonumber\\
    &=2\sum\limits_{j=1}^{M\, D_2} \underset{v= b_j}{\Res}\, \tilde{K}_j(z_1,v)\,\mathbf{D}(\tilde{\omega}_{0,1}(v))\,\tilde{\omega}_{0,2}(\tilde{\sigma}_j(v),w)\nonumber\\
    &=\sum\limits_{j=1}^{M\, D_2} \underset{v= b_j}{\Res}\, \tilde{K}_j(z_1,v)\,(\mathbf{D}(\tilde{\omega}_{0,1}(v))\,\tilde{\omega}_{0,2}(\tilde{\sigma}_j(v),w)+\mathbf{D}(\tilde{\omega}_{0,1}(\tilde{\sigma}_j(v)))\,\tilde{\omega}_{0,2}(v,w)),\nonumber
\end{align}
where in the second equality we use that $\tilde{\omega}_{0,2}(v,w)+\tilde{\omega}_{0,2}(\tilde{\sigma}_j(v),w)=0$ at $v=b_j$, in the third equality we use $\tilde{K}_j(z_1,v)=\frac{1}{2}\frac{\tilde{\omega}_{0,2}(v,z_1)}{\dd X(v)\, \dd Y(v)}(1+O(v-b_j))$ and in the last equality we use $\tilde{K}_j(z_1,v)=\tilde{K}_j(z_1,\tilde{\sigma}_j(v))$. 

By integrating \eqref{eq:def_Berg} from $w=\tilde{\sigma}_i(z)$ to $w=z$ along a contour that does not pass through the ramification point $b_i$, we obtain 
\begin{align}
& \frac{1}{2}\frac{f(z,\tilde{\sigma}_i(z))}{\tilde{\omega}_{0,1}(z)-\tilde{\omega}_{0,1}(\tilde{\sigma}_i(z))} \int_{w=\tilde{\sigma}_i(z)}^{z}\mathbf{D}\tilde{\omega}_{0,2}(z_1,w) = \nonumber\\
& \sum\limits_{j=1}^{M\, D_2} \underset{v= b_j}{\Res}\, \tilde{K}_j(z_1,v)\Big(\mathbf{D}\tilde{\omega}_{0,1}(v) \tilde{K}_i(\tilde{\sigma}_j(v),z) + \mathbf{D}\tilde{\omega}_{0,1}(\tilde{\sigma}_j(v)) \tilde{K}_i(v,z)\Big) f(z,\tilde{\sigma}_i(z)). \label{eq:third:term}
\end{align}
We can rewrite the second term on the right hand side of \eqref{eq:variation:kernel} by using Cauchy formula:
\[
\begin{split}
    -\tilde{K}_i(z_1,z)\frac{\mathbf{D}\tilde{\omega}_{0,1}(z)}{\tilde{\omega}_{0,1}(z)-\tilde{\omega}_{0,1}(\tilde{\sigma}_i(z))} =& \,\underset{v=z}{\Res}\,\tilde{K}_i(z_1,v)\mathbf{D}\tilde{\omega}_{0,1}(v) \tilde{K}_i(\tilde{\sigma}_i(v),z) \\
    &+\underset{v=\tilde{\sigma}_i(z)}{\Res}\,\tilde{K}_i(z_1,v)\mathbf{D}\tilde{\omega}_{0,1}(\tilde{\sigma}_i(v)) \tilde{K}_i(v,z), \\
    \tilde{K}_i(z_1,z)\frac{\mathbf{D}\tilde{\omega}_{0,1}(\tilde{\sigma}_i(z))}{\tilde{\omega}_{0,1}(z)-\tilde{\omega}_{0,1}(\tilde{\sigma}_i(z))}  =&\,\underset{v=z}{\Res}\,\tilde{K}_i(z_1,v)\mathbf{D}\tilde{\omega}_{0,1}(\tilde{\sigma}_i(v)) \tilde{K}_i(v,z) \\
    &+\underset{v=\tilde{\sigma}_i(z)}{\Res}\,\tilde{K}_i(z_1,v)\mathbf{D}\tilde{\omega}_{0,1}(v) \tilde{K}_i(\tilde{\sigma}_i(v),z).
\end{split}
\]
In the end, the second term of \eqref{eq:variation:kernel} can be expressed as:
\begin{equation}\label{eq:second:term}
\Big(\underset{v=z}{\Res}+\underset{v=\tilde{\sigma}_i(z)}{\Res}\Big)\tilde{K}_i(z_1,v)\Big(\mathbf{D}\tilde{\omega}_{0,1}(v) \tilde{K}_i(\tilde{\sigma}_i(v),z)+\mathbf{D}\tilde{\omega}_{0,1}(\tilde{\sigma}_i(v)) \tilde{K}_i(v,z)\Big) f(z,\tilde{\sigma}_i(z)).
\end{equation}
Gathering formulas \eqref{eq:variation:kernel}, \eqref{eq:third:term} and \eqref{eq:second:term} and using 
\[
\sum\limits_{i,j=1}^{M\, D_2}\underset{z= b_i}{\Res}\,\underset{v= b_j}{\Res}\, + \sum\limits_{i=1}^{M\,D_2} \Big(\underset{z= b_i}{\Res}\,\underset{v=z}{\Res}\, + \underset{z= b_i}{\Res}\,\underset{v=\tilde{\sigma}_i(z)}{\Res}\Big) =  \sum\limits_{i,j=1}^{M\, D_2}\underset{v= b_j}{\Res}\,\underset{z= b_i}{\Res},
\]
we get the lemma.
\end{proof}

\begin{proof}[Proof of Proposition \ref{prop:taylor:TR}]
First, the cases $(g,n)=(0,1), \, (0,2)$ undergo the same treatment as in the proof of Proposition \ref{prop:taylor:combi}, since the definitions of $\omega_{0,i}$, $\chi_{0,i}^{\ell}$ on one hand, and of $\tilde{\omega}_{0,i}$, $\tilde{\chi}_{0,i}^{\ell}$ on the other hand, coincide for $i=1,\,2$.\\

Second, let $\alpha_0$ small enough so that there exist $M\, D_2$ pairwise disjoint closed contours $\gamma_{1},\dots,\gamma_{M\, D_2}\subset \mathbb{P}^1$ not depending on $\alpha$, such that $\gamma_i$ encloses $ b_i$ (and no other $ b_j$ for $j\neq i$) for all $|\alpha|\leq |\alpha_0|$. Such contours exist because for $\alpha=0$, the ramification points $b_i|_{\alpha=0} = \tilde{a}_i$ are disjoint, and because they are continuous functions at $\mathbf{p}=0$ (since all sums converge absolutely).
With this assumption, we can then write the residues at $b_i$ as integrals over contours that do not depend on $\alpha$:
\[
\underset{z= b_i}{\Res}\, f(z) = \int_{z\in\gamma_i} f(z).
\]
Therefore, the operator $\mathbf{D}$ passes through the residue:
\[
\mathbf{D}\,\underset{z= b_i}{\Res}\, f(z) = \underset{z= b_i}{ \Res}\, \mathbf{D} \,f(z).
\]
We show by induction on $(2g-2+n,\ell)$ the following identity:
\begin{equation}\label{eq:derivation:TR}
\begin{split}
    \mathbf{D}^{\ell}\tilde{\omega}_{g,n}(z_1,L)= \sum\limits_{i=1}^{M\, D_2}&\underset{z= b_i}{\Res}\,\tilde{K}_{i}(z_1,z)\Bigg(\sum\limits_{k=1}^{\ell}{\ell \choose k}\Big( \mathbf{D}^{k}\tilde{\omega}_{0,1}(z)\mathbf{D}^{\ell-k}\tilde{\omega}_{g,n}(\tilde{\sigma}_i(z),L)\\
    &+\mathbf{D}^{k}\tilde{\omega}_{0,1}(\tilde{\sigma}_i(z))\mathbf{D}^{\ell-k}\tilde{\omega}_{g,n}(z,L)\Big) \\
    & \mathbf{D}^{\ell}\Big(\tilde{\omega}_{g-1,n+1}(z,\tilde{\sigma}_i(z),L) + \sum\limits_{\substack{h+h'=g\\ C\sqcup C'=L}}^{'}\tilde{\omega}_{h,1+|C|}(z,C)\tilde{\omega}_{h',1+|C'|}(\tilde{\sigma}_i(z),C')\Big)\Bigg).
\end{split}
\end{equation}
For $\ell=0$ we recover the formula of topological recursion for the spectral curve $\tilde{\mathcal{S}}$, so we get $\mathbf{D}^{0}\tilde{\omega}_{g,n}=\tilde{\omega}_{g,n}$ by induction on $2g-2+n$. Let us now assume that the formula holds up to order $(2g-2+n,\ell)$. We write:
\[
    \mathbf{D}^{\ell+1}\tilde{\omega}_{g,n}(z_1,\dots,z_n) = \mathbf{D}\cdot \mathbf{D}^{\ell}\tilde{\omega}_{g,n}(z_1,\dots,z_n).
\]
We can use the induction hypothesis -- equation \eqref{eq:derivation:TR} -- to write $\mathbf{D}^{\ell}\tilde{\omega}_{g,n}(z_1,\dots,z_n)$. Then, passing $\mathbf{D}$ through the residues, and using lemma \ref{lem:deformation:kernel} and the identity ${\ell \choose k}+{\ell \choose k-1} = {\ell+1 \choose k}$, we recover equation \eqref{eq:derivation:TR} for the order $(2g-2+n,\ell+1)$. This completes the proof of equation \eqref{eq:derivation:TR}.\\
To prove proposition \ref{prop:taylor:TR}, it is enough to prove that $\mathbf{D}^{\ell}\tilde{\omega}_{g,n}(z_1,\dots,z_n)|_{\alpha = 0} =\tilde{\chi}_{g,n}^{\ell}(z_1,\dots,z_n)$. Since they satisfy the same recursive relation (equations \eqref{eq:chi:TR} and \eqref{eq:derivation:TR}) and they coincide for $(g,n)=(0,1),\, (0,2)$, $\ell\geq 0$, the equality holds for all $g,\,n,\,\ell$. 
\end{proof}


We now prove that the Taylor expansions are actually the same.
\begin{prop}\label{prop:egalite:taylor}
For all $g\geq 0,\, n\geq 1,\, \ell\geq 0$ and $z_1,\dots,z_n\in \mathbb{CP}^1$, we have:
\[
\chi_{g,n}^{\ell}(z_1,\dots,z_n)=\tilde{\chi}_{g,n}^{\ell}(z_1,\dots,z_n).
\]
\end{prop}
\begin{proof}
First, notice that for any $\ell\geq 0$, the cases $(g,n)=(0,1),\, (0,2)$ are true by definition. Also, from the definition of $\tilde{\chi}_{g,n}^{0}$, we have $\tilde{\chi}_{g,n}^{0}=\chi_{g,n}=\chi_{g,n}^{0}$. \\
We assume that the equality is proved up to order $(2g-2+n,\ell-1)$, where $2g-2+n\geq 1$ and $\ell\geq 1$, and we prove that the equality holds at order $(2g-2+n,\ell)$. From equation \eqref{eq:chi:TR}, and the induction hypothesis, we have:
\[
\begin{split}
    \tilde{\chi}_{g,n}^{\ell}(z_1,L) = \sum\limits_{i=1}^{M\, D_2} \underset{z= \tilde{a}_i}{\Res}\,& K_i(z_1,z) \Bigg(\sum\limits_{k=1}^{\ell}{\ell \choose k}\big(\chi^{k}_{0,1}(z)\, \chi^{\ell-k}_{g,n}(\sigma_i(z),L)+\chi^{k}_{0,1}(\sigma_i(z)\, \chi^{\ell-k}_{g,n}(z,L)\big)\\ 
    &\!\!\!\!+\chi^{\ell}_{g-1,n+1}(z,\sigma_i(z),L) + \sum\limits_{\substack{h+h'=g\\ C\sqcup C'=L\\ k+k'=\ell}}^{'}{\ell \choose k} \chi_{h,1+|C|}^{k}(z,C) \chi_{h',1+|C'|}^{k'}(\sigma_i(z),C')\Bigg).
\end{split}
\]
For a differential $f(z')$ defined near $z'=0$, we define the modified insertion operator $\widehat{\Gamma}_{z'}$:
\[
\widehat{\Gamma}_{z'}f(z')\coloneqq \sum\limits_{k=1}^{D_1}\frac{p_k}{k}\, \underset{z'=0}{\Res}\, f(z')\, X_0(z')^k.
\]
From the definition of $\chi_{g,n}^{\ell}$ and formula \eqref{eq:insertion:operator:res} for the insertion operator, we can write:
\[\begin{split}
\chi_{g,n}^{\ell}(z_1,\dots,z_n) =& \sum\limits_{k_1,\dots,k_{\ell}=1}^{D_1}\frac{p_{k_1}}{k_1}\dots\frac{p_{k_{\ell}}}{k_{\ell}}\underset{z'_1=0}{\Res}\dots\underset{z'_{\ell}=0}{\Res}\, \chi_{g,n+\ell}(z_1,\dots,z_n,z'_1,\dots,z'_{\ell})\, X_{0}(z'_1)^{k_1}\dots X_{0}(z'_{\ell})^{k_{\ell}}\\
=&\,\, \widehat{\Gamma}_{z'_1}\dots\widehat{\Gamma}_{z'_{\ell}}\, \chi_{g,n+\ell}(z_1,\dots,z_n,z'_1,\dots,z'_{\ell}). 
\end{split}\]
Hence:
\[\begin{split}
\sum\limits_{k=0}^{\ell}{\ell\choose k} \chi_{h,1+|C|}^{k}(z,C)\chi_{h',1+|C'|}^{\ell-k}(\sigma_i(z),C') &=\\ &\!\!\!\!\!\!\!\!\!\!\!\!\widehat{\Gamma}_{z'_1}\dots\widehat{\Gamma}_{z'_{\ell}} \sum\limits_{\substack{U\sqcup U' =\\ \{z'_1,\dots,z'_{\ell}\}}}\chi_{h,1+|C\sqcup U|}(z,C,U)\,\chi_{h',1+|C'\sqcup U'|}^{\ell-k}(\sigma_i(z),C',U').
\end{split}\]
Applying this identity in the formula for $\tilde{\chi}_{g,n}^{\ell}$, we have:
\[
\begin{split}
    \tilde{\chi}_{g,n}^{\ell}(z_1,L) =\widehat{\Gamma}_{z'_1}\dots\widehat{\Gamma}_{z'_{\ell}} \sum\limits_{i=1}^{M\, D_2} \underset{z= \tilde{a}_i}{\Res}\, K_i(z_1,z) \Bigg(& \chi^{\ell}_{g-1,n+1+\ell}(z,\sigma_i(z),L,z'_1,\dots,z'_{\ell})\\
    &+ \sum\limits_{\substack{h+h'=g\\ C\sqcup C'=L\\ U\sqcup U'=\\\{z'_1,\dots,z'_{\ell}\}}}^{'} \chi_{h,1+|C\sqcup U|}(z,C,U)\, \chi_{h',1+|C'\sqcup U'|}(\sigma_i(z),C',U')\Bigg)
\end{split}
\]
(we could exchange the residue $\underset{z=\tilde{a}_i}{\Res}$ and the operators $\widehat{\Gamma}_{z'_j}$ since the integrand do not have poles at $z'_j= z, \, \sigma_i(z)$). Note that in this formula, the primed sum means that the terms $(h,C,U)=(0,\emptyset, \emptyset), \,(g,L,\{z'_1,\dots,z'_{\ell}\})$ are discarded from the sum. We recognise on the right hand side the topological recursion formula for $\chi_{g,n+\ell}$, so we get:
\[
\tilde{\chi}_{g,n}^{\ell}(z_1,\dots,z_n)= \widehat{\Gamma}_{z'_1}\dots\widehat{\Gamma}_{z'_{\ell}}\, \chi_{g,n+\ell}(z_1,\dots,z_n,z'_1,\dots,z'_{\ell}) = \chi_{g,n}^{\ell}(z_1,\dots,z_n),
\]
which ends the induction. 
\end{proof}

We can now finish the proof of our main theorem.
\begin{proof}[Proof of Theorem~\ref{thm:main:result}]
As a result of propositions \ref{prop:taylor:combi}, \ref{prop:taylor:TR} and \ref{prop:egalite:taylor}, we have the following facts:
\begin{itemize}
    \item $\omega_{g,n}(z_1,\dots,z_n)$ is analytic and well defined near $z_1,\dots,z_n=0$; $\tilde{\omega}_{g,n}(z_1,\dots,z_n)$ is well-defined for all $z_1,\dots,z_n\in\mathbb{CP}^1$.
    \item Where they are defined and for small enough $\alpha$, $\omega_{g,n}$ and $\tilde{\omega}_{g,n}$ admit Taylor expansions with respect to $\alpha$, which coincide. 
    \item By uniqueness of the expansion, we conclude that $\omega_{g,n}(z_1,\dots,z_n)$ is equal to $\tilde{\omega}_{g,n}(z_1,\dots,z_n)$ in a neighbourhood of zero. In particular, it can be globally analytically continued. 
    \item As a byproduct, the family $(\omega_{g,n})_{\substack{g\geq 0\\ n\geq 1}}$ satisfies the topological recursion for the spectral curve $\tilde{\mathcal{S}}$. 
\end{itemize}
The proof of Theorem \ref{thm:main:result} is complete.
\end{proof}

\section{Constellations and proofs for $W_{0,1}$ and $W_{0,2}$ using Albenque--Bouttier techniques}
\label{sec:constellations}

In this section we focus on the case where $G$ is a polynomial, i.e.~$J=\emptyset$ and $G(\cdot) = \prod_{i\in I}(1+\cdot u_i)$. It will be more convenient to use a different parametrisation, namely $\tilde{G}(\cdot) = \prod_{i\in I} (\bar{u}_i + \cdot) = \left(\prod_{i\in I} \bar{u}_i\right) G(\cdot)$, with $\bar{u}_i\equiv u_i^{-1}$. The effect of this parametrisation is that
\begin{equation}\label{eq:changeOfVariables}
    \tau^G(\mathbf{p}, \mathbf{q}; u_0, \dotsc, u_{m-1}; t) = \tau^{\tilde{G}}(\mathbf{p}, \mathbf{q}; u_0, \dotsc, u_{m-1}; \tilde{t}), \quad \text{with $\tilde{t} = \Bigl(\prod_{i\in I} u_i\Bigr) t$}.
\end{equation}

We first give some context for the relation of our work with the one of Albenque and Bouttier~\cite{AlbenqueBouttier2022} which is the main tool of this section. These two authors announced in talks a few years ago a purely combinatorial proof of Eynard's solution of the two-matrix model for the planar leading order (i.e.~what is here Theorem~\ref{thm:Discrat} for $(m,r)=(1,0)$) using a generalisation to hypermaps of the slice-decomposition of planar maps~\cite{BouttierGuitterContinuedFractions, BouttierGuitter2014}. 
Although we prove in this section a more general result (since we prove Theorem~\ref{thm:Discrat} for $r=0$ and any value of $m$), at the combinatorial level, the only tool we need is the Albenque--Bouttier approach of~\cite{AlbenqueBouttier2022}. Indeed, maybe surprisingly, the combinatorial objects underlying the model for arbitrary $m\geq1$ form in fact a \emph{subset} of the objects considered in \cite{AlbenqueBouttier2022} for $m=1$. 

However, because these objets need to be considered with more general weights (related to a notion of colour of vertices), it is necessary to go through the whole construction of~\cite{AlbenqueBouttier2022} in order to verify that it gives rise to the spectral curve we aim for. The main originality of this section is 
the embedding of the combinatorial objects into the case $m=1$, and the introduction of appropriate generating functions in the presence of colour weights.

In addition to the slice decomposition itself for $m=1$, \cite{AlbenqueBouttier2022} explains how to use it to derive the spectral curve in a closed form. It turns out that it is quite straightforward to introduce the colour weights in these calculations. For readers well acquainted with combinatorics of paths, calculations done in the present section (once Proposition~\ref{thm:FactorizationsConstellations} is known) essentially follow \cite{AlbenqueBouttier2022}, inserting colour weights where necessary. We therefore refer to~\cite{AlbenqueBouttier2022} for proofs of bijectivity of the decompositions and the details of the calculations. 



\subsection{Constellations}

A \emph{map} is a connected graph (with possible multiple edges or loops), embedded without edge-crossing in an oriented surface and considered up to orientation preserving homeomorphisms.
A \emph{bicolored map} is a map whose faces can be colored black or white such that a black face can share an edge only with a white face, and the other way around. Bicolored maps have a \emph{canonical orientation}, which orients edges so that black faces are to their left (and white faces to their right), see Figure \ref{fig:Faces}. Notice that there exists an oriented path between any two vertices.

\begin{defn}[$m$-Constellations]
Let $m\geq 1$. An $m$-constellation is a bicolored map equipped with its canonical orientation, such that 
\begin{itemize}
    \item each vertex has a color $c\in[0..m-1]$;
    \item if $v$ is a vertex of color~$c$, then any edge outgoing from $c$ points to a vertex of color $c-1\mod m$.
\end{itemize}
\end{defn}
\noindent As a consequence, black and white faces have degrees multiple of $m$.

\begin{figure}
    \centering
    \includegraphics[scale=.4]{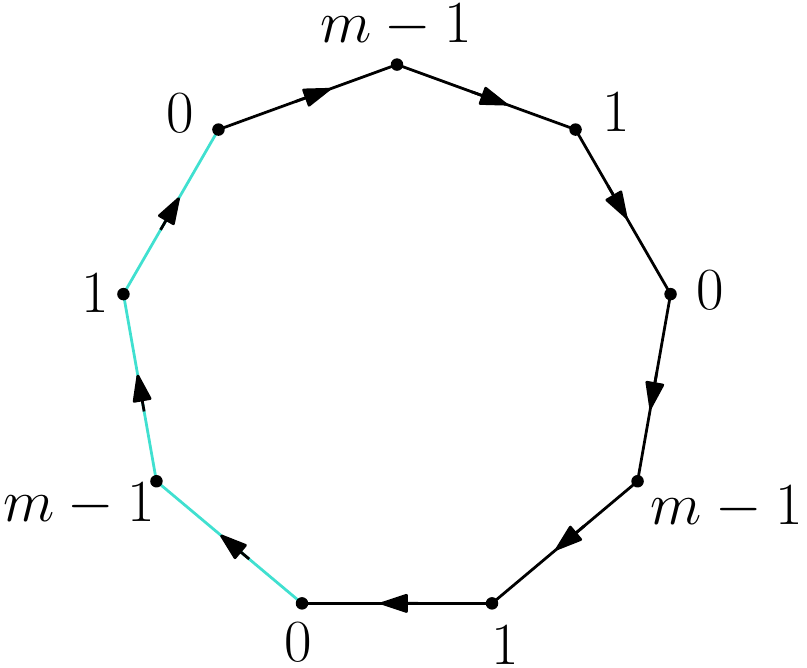} \hspace{2cm} \includegraphics[scale=.4]{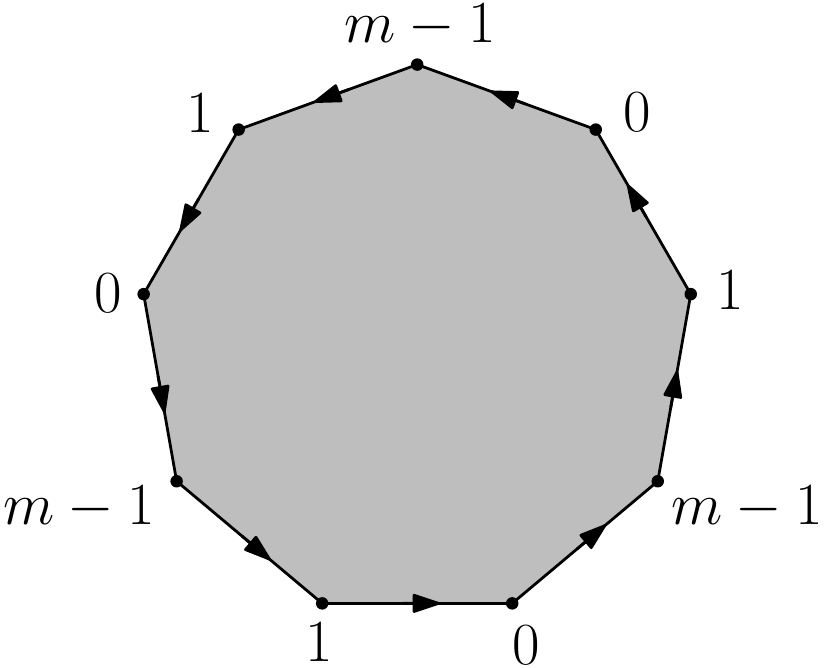} \hspace{2cm}
    \includegraphics[scale=.4]{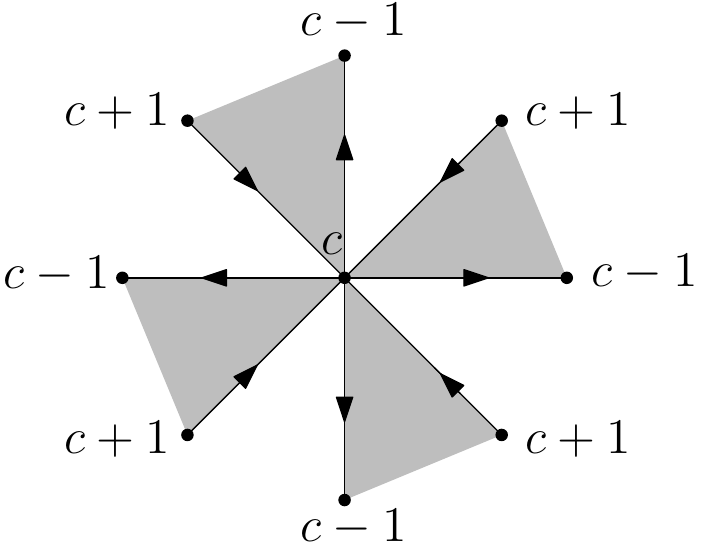}
    \caption{A white face with a marked right path, a black face (both of degree $3m$), and a vertex of color~$c$. Edges receive their canonical orientation so that the black faces are to their left.}
    \label{fig:Faces}
\end{figure}

A map is \emph{pointed} if it has a marked vertex. In a pointed bicolored map, there is a canonical function $\ell_{v^*}$, $v^*$ being the pointed vertex, from the set of vertices to $\mathbb{N}$, where $\ell_{v^*}(v)$ is the ``distance'', i.e.~the number of edges of a shortest oriented path from $v$ to $v^*$ (note that there is a single vertex mapped to 0, which is $v^*$). Those distances cannot decrease by more than 1 along an edge, i.e.~if there is an edge from $v$ to $v'$, then $\ell_{v^*}(v')\geq \ell_{v^*}(v)-1$. In an $m$-constellation, 
since the length of any oriented path is congruent to the difference in colours of its endpoints,
one has more specifically $\ell_{v^*}(v') \in \ell_{v^*}(v) -1 + m\mathbb{N}$.

A \emph{left path} (\resp \emph{right path}) is an oriented sequence of $m$ connected edges from a vertex of color $0$ to the next vertex of color $0$ along a black face (\resp white face). The size of an $m$-constellation is the number of left (equivalently, right) paths, see Figure \ref{fig:Faces}.

A labeled $m$-constellation of size $n$ is an $m$-constellation whose right paths are all labeled, from 1 to~$n$.
Recall that a factorisation of the identity in the symmetric group $\mathfrak{S}_n$ is called \emph{transitive} if its factors altogether generate a transitive subgroup of $\mathfrak{S}_n$.

\begin{prop} \label{thm:FactorizationsConstellations}
There is a one-to-one correspondence between labeled $m$-constellations of size $n$ and transitive factorisations of the form $\sigma_{-2}\sigma_{-1}\sigma_0 \dotsm \sigma_{m-1} = 1\in\mathfrak{S}_n$ which maps the number of vertices of color $c\in[0..m-1]$ to the number of cycles of $\sigma_c$, and cycles of $\sigma_{-2}$ (\resp $\sigma_{-1}$) of length $r$ to black (\resp white) faces of degree $mr$ (\resp $ms$).
\end{prop}

\begin{proof}
Given a labeled $m$-constellation of size $n$, we define the permutations $\sigma_{-2}, \sigma_{-1}, \dotsc, \sigma_{m-1}$ as follows. 
\begin{itemize}
    \item Each vertex of color $c=1, \dotsc, m-1$ encodes a cycle of $\sigma_c$, such that $\sigma_c$ maps a right path to the next one in the clockwise order at the vertex of $c$ where it is incident to.
    \item Each white face encodes a cycle of $\sigma_{-1}$ by writing down the right paths around it clockwise.
    \item Black faces determine the cycles of $\sigma_{-2}$ by writing down the labels of the right paths which have their edge oriented from the vertex of color 0 to that of color $m$ incident to them.
    \item As for the cycles of $\sigma_0$, they are defined by writing down in the clockwise order the labels of the right paths which have an edge between a vertex of color 1 and a vertex of color 0 incident at the latter.
\end{itemize}
Those rules are represented in Figure \ref{fig:PermutationCycles}. The permutations thus obtained factorise the identity in $\mathfrak{S}_n$, as can be seen in the following representation:
\begin{equation*}
    \begin{array}{c}\includegraphics[scale=.7]{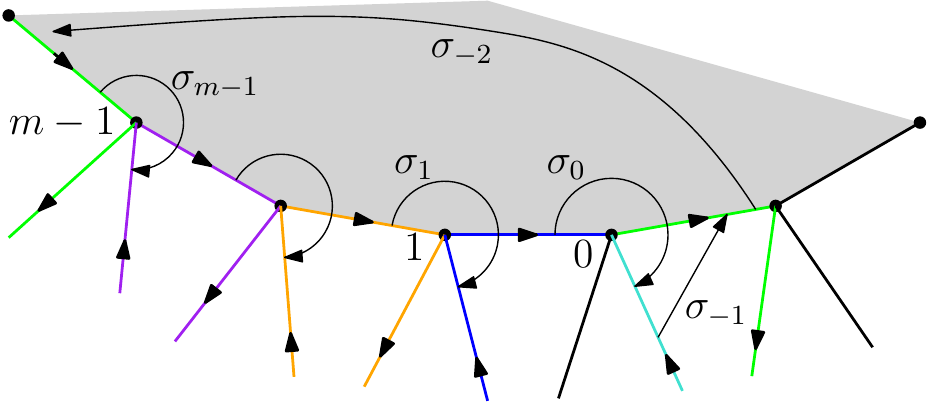} \end{array}
\end{equation*}
Starting at an edge $e$ between vertices of color 0 and $m-1$, the permutations $\sigma_{m-1}$ to $\sigma_1$ move it along the boundary of the black face it is incident to. The permutation $\sigma_0$ then moves to a right path which is not incident to this face, but this is ``corrected'' by $\sigma_{-1}$. It arrives at the next edge, in the counter-clockwise order, between vertices of color 0 and $m-1$ along that face. This edge is in turn mapped to $e$ by $\sigma_{-2}$.

Notice that the edges incident to a vertex of color 0 which are connected to a vertex of color $m-1$ are determined by $\sigma_{-1}$. A set of permutations $\sigma_{-1}, \sigma_0, \dotsc, \sigma_{m-1}$ therefore determines a labeled $m$-constellation, whose white faces are given by $\sigma_{-2}$.
\begin{figure}
    \centering
    \includegraphics[scale=.5]{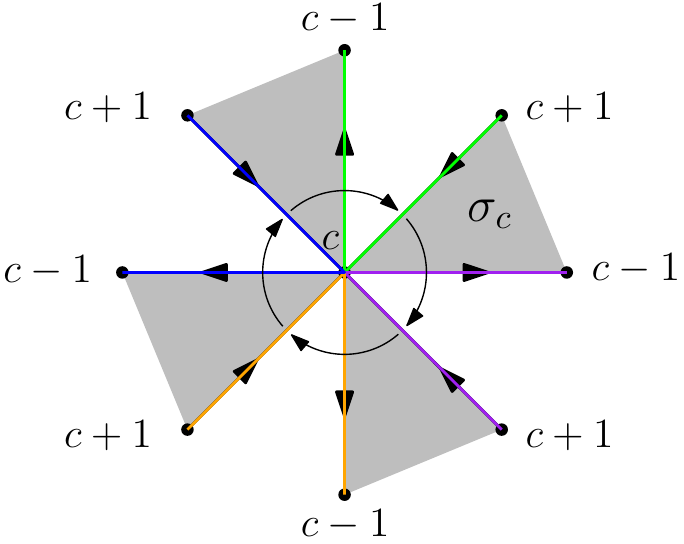} \hspace{.5cm}
    \includegraphics[scale=.4]{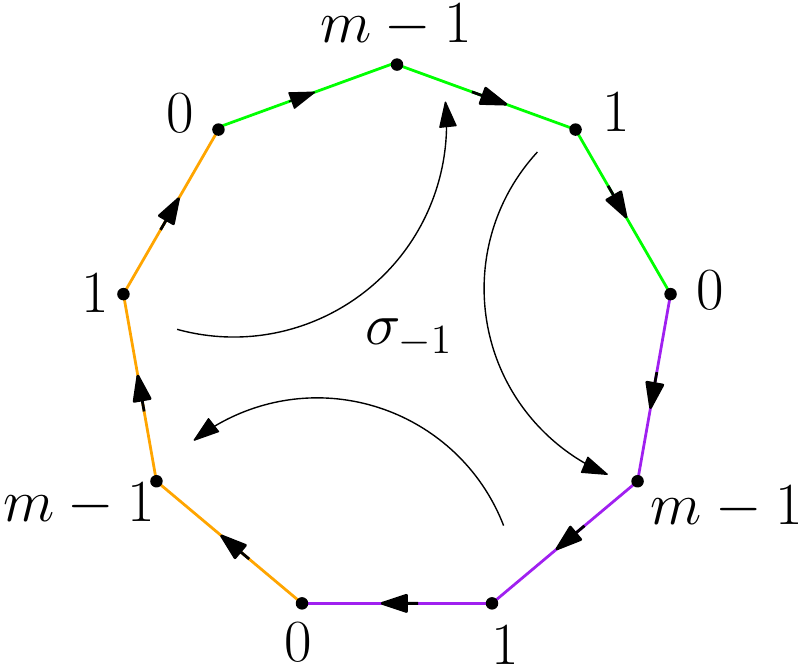}
    \hspace{.5cm}
    \includegraphics[scale=.4]{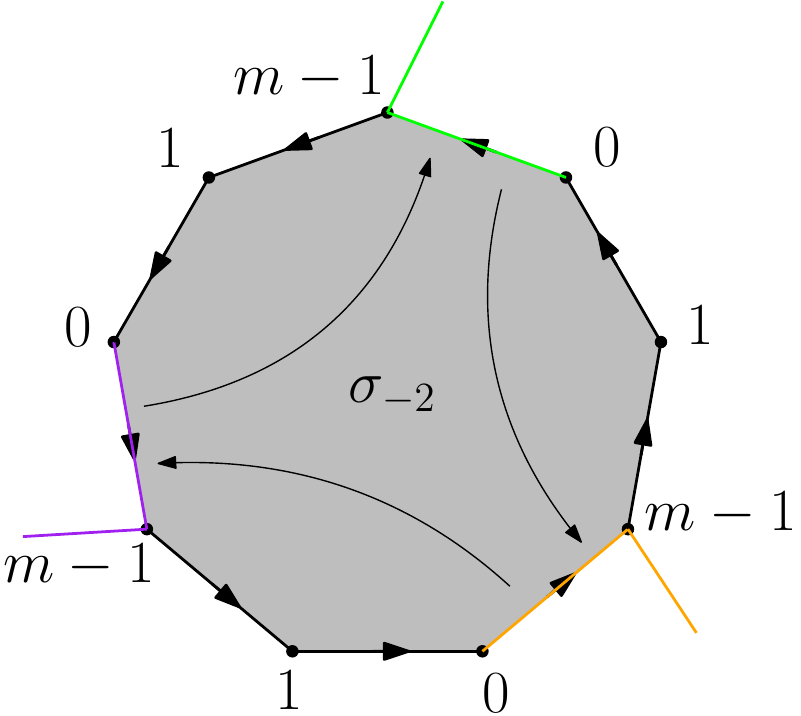}
    \hspace{.5cm}
    \includegraphics[scale=.5]{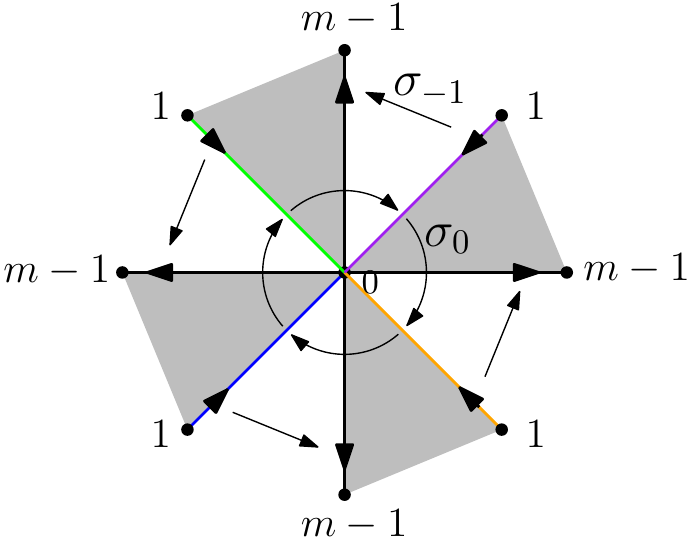}
    \caption{The actual colors represent different right paths. We have added the action of $\sigma_{-1}$ on the right paths incident to vertices of color 0.}
    \label{fig:PermutationCycles}
\end{figure}
\end{proof}
Evidently, one can remove the transitivity assumption from the factorisations in the above proposition and obtain a bijection with labelled \emph{non-necessarily connected} $m$-constellations (with obvious definition).

\begin{rem}
The previous proposition is classical in the combinatorics literature in the case where $\sigma_{-2}$ is the identity, i.e. when the associated $m$-constellation has only black faces of degree $m$ (or ``$m$-hyperedges''), see e.g.~\cite{BousquetSchaeffer2000, Chapuy2009}, but we are not aware of a previous use of the general case stated here. This presentation will be crucial for us since the slice decomposition need to be able to ''cut-out'' faces rather than vertices, and since we want to control \emph{two} full sets of degree parameters (in relation with our variables $\mathbf{p}$ and $\mathbf{q}$). Note also that these $m$-hyperedges are sometimes given different, but equivalent, graphical representations, for example as ``star vertices'' of degree $m$, such as in~\cite{AlexandrovChapuyEynardHarnad2020}.
\end{rem}

The weight $w(M)$ of an $m$-constellation is defined as
\begin{equation*}
    w(M) = \tilde{t}^{n(M)} \prod_{c=0}^{m-1} \bar{u}_c^{n_c(M)} \prod_{s=1}^{D_1} p_s^{f^{\circ}_s(M)} \prod_{r=1}^{D_2} q_r^{f^{\bullet}_r(M)},
\end{equation*}
where $n(M)$ is the number of right paths (which is the number of edges divided by $m$), $n_c(M)$ the number of vertices of color $c$, $f^\circ_s(M)$ the number of white faces of degree $ms$ and $f^\bullet_r(M)$ the number of black faces of degree $mr$.

A \emph{white rooted} (\resp \emph{black rooted}) constellation is an $m$-constellation with a marked right (\resp left) path. Equivalently, it corresponds to marking a corner for a fixed color $c\in[0..m-1]$ in a white (\resp black) face. The series $W_{g,n}(x_1, \dotsc, x_n)$ then corresponds to the generating series of labelled connected $m$-constellations of genus $g$,
\begin{itemize}
    \item with $n$ labeled, marked right paths such that no two of them lie in the same white face,
    \item counted with $x_i^{-f_i(M)-1}$ where $f_i(M)$ is the degree of the face with the $i$-th marked right path,
    \item with weight $w(M)$ except for the $p_{f_i(M)}$s of the white faces containing the marked right paths.
\end{itemize}

\subsection{Slices and elementary slices}

\begin{quote}{\it From now on, in Section~\ref{sec:constellations}, all constellations will be planar.}
\end{quote}

An $m$-constellation with a virtual boundary is a map satisfying the same definition as $m$-constellations except for one face which is neither black nor white. In particular, its boundary is not necessarily an oriented cycle.

\begin{defn}[Slices, ~\cite{AlbenqueBouttier2022}]
A slice is a planar $m$-constellation with a virtual boundary $f^*$ which has three marked corners, $l, r, o$, which split the boundary of $f^*$ into three connected paths as follows:
\begin{itemize}
    \item The \emph{right boundary} is oriented from $r$ to $o$ (avoiding $l$) and is the unique geodesic from $r$ to $o$.
    \item The \emph{left boundary} is oriented from $l$ to $o$ (avoiding $r$) and is a geodesic from $r$ to $o$.
    \item The \emph{base} between $l$ and $r$.
    \item The left and right boundaries intersect only at the \emph{origin}, which is the vertex where $o$ sits.
\end{itemize}
If the base is oriented from $l$ to $r$ (\resp from $r$ to $l$) we say that it is a \emph{white} (\resp \emph{black}) slice, see Figure \ref{fig:Slice}. If the base consists in a single edge, we say that it is an \emph{elementary} slice.
\end{defn}

\begin{figure}
    \centering
    \includegraphics[scale=.5]{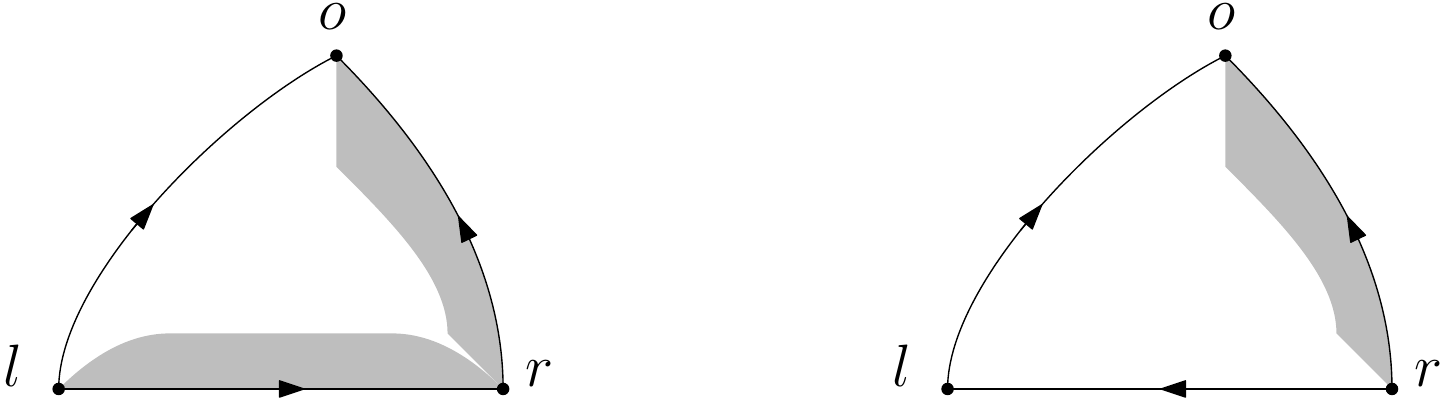}
    \caption{A white (left) and a black (right) slices with their marked corners $0, l, r$.}
    \label{fig:Slice}
\end{figure}

A slice is canonically pointed at its origin and we label vertices with $\ell_o(v)$ (abusing the notation by using the corner $o$ instead of the vertex as subscript). If we denote $\ell_o(l)$ and $\ell_o(r)$ the labels at $l$ and $r$, we call $\ell_o(r)-\ell_o(l)$ the \emph{increment} of white slices and $\ell_o(l)-\ell_o(r)$ the increment of black slices.

Let $\mathcal{A}^{(c)}_{n,k}$ (\resp $\mathcal{B}^{(c)}_{n,k}$) be the set of white (\resp black) slices of increment $k$, whose bases have length $n$ and where $l$ (\resp $r$) sits at a vertex of color $c$. Their generating series are defined as
\begin{equation*}
    A^{(c)}_{n,k} = \sum_{M\in \mathcal{A}^{(c)}_{n,k}} w(M),\qquad B^{(c)}_{n,k} = \sum_{M\in \mathcal{B}^{(c)}_{n,k}} w(M).
\end{equation*}
Let $\mathcal{A}^{(c)}_k$ (\resp $\mathcal{B}^{(c)}_k$) be the set of elementary white (\resp black) slices of increment $mk-1$ where $l$ (\resp $r$) sits at a vertex of color $c\in[0..m-1]$, with generating series $\tilde{A}^{(c)}_k$ (\resp $\tilde{B}^{(c)}_k$) defined similarly as above. The following proposition is essentially proved in \cite{AlbenqueBouttier2022} (we simply add the colour weights here).

\begin{prop} \label{thm:CompositeSlices}
Those series are given by
\begin{equation} \label{CompositeSlices}
\begin{aligned}
A_{n,k}^{(c)} &= [\alpha^k] \alpha^{-n} \prod_{r=1}^{n} \tilde{A}^{(c-r+1)}(\alpha^m),\\
B_{n,k}^{(c)} &= [\alpha^{-k}] \alpha^{n} \prod_{r=1}^{n} \tilde{B}^{(c-r+1)}(\alpha^m),
\end{aligned}
\end{equation}
with
\begin{equation*}
    \tilde{A}^{(c)}(z) = \bar{u}_c A^{(c)}(z \prod_{i\in I} u_i),\qquad \tilde{B}^{(c)}(z) = B^{(c)}(z \prod_{i\in I} u_i),
\end{equation*}
where the polynomials $A^{(c)}(z), B_k^{(c)}(z)$ (in $z$ and $1/z$ respectively) are given by Definition \ref{def:WkBk}, which in the polynomial case reduces to
\begin{equation} \label{eq:PolynomialAB}
    \begin{aligned}
    A^{(c)}(z) &= 1 + u_c \sum_{s=1}^{D_2} q_s t^s \left\{z^s \frac{\prod_{c'=0}^{m-1} B^{(c')}(z)^s}{B^{(c)}(z)}\right\}^{\geq},\\
    B^{(c)}(z) &= 1 + u_c \sum_{s=1}^{D_1} p_s \left[z^{-s} \frac{\prod_{c'=0}^{m-1} A^{(c')}(z)^s}{A^{(c)}(z)}\right]^{<}.
    \end{aligned}
\end{equation}
\end{prop}

We will see the generating series of slices as generating series of \L{}ukasiewicz paths appropriately weighted. Along an oriented edge of a slice
\begin{itemize}
    \item the color decreases by $1\mod m$,
    \item and the label $\ell_o(v)$ changes by $mk-1$ for $k\in\mathbb{N}$.
\end{itemize}
This explains why it is the following family of paths which will naturally appear.

\begin{defn}[Weighted $m$-paths]
A white $m$-path starting at color $c\in[0..m-1]$ is a path on $\mathbb{Z}^2$ with steps of type $\omega_p^\circ = (1,mp-1)$ for $p\geq 0$, with associated weight $\tilde{A}^{(c+r-1)}_p$ if the $r$-th step is of type $\omega_p^{\circ}$. A black $m$-path is a path on $\mathbb{Z}^2$ with steps of type $\omega_p^\bullet = (1,1-mp)$ for $p\geq 0$, with associated weight $\tilde{B}^{(c-r+1)}_p$ if the $r$-th step is of type $\omega_p^{\bullet}$. We say that the $r$-th step has color $c-r+1\mod m$.
\end{defn}

Those paths already appeared, without the colour-dependent weight system, in \cite{AlbenqueBouttier2022} and with black paths restricted to $p=0$ in the case of constellations with $q_k = \delta_{k,1}$ in \cite{AlbenqueBouttier:fpsac12}.

Let $\mathcal{L}^{(c)\circ}_{n,k}$ (\resp $\mathcal{L}^{(c)\bullet}_{n,k}$) be the set of $m$-paths which starts at $(0,0)$ with the color $c$ and ends at $(n,k)$ (\resp $(n,-k)$). Let $\tilde{A}^{(c)}(z) = \sum_{i\geq0} \tilde{A}_i^{(c)} z^i$ and $\tilde{B}(z) = 1 + \sum_{i\geq1} \tilde{B}_i^{(c)} z^{-i}$. Then the generating series of $\mathcal{L}^{(c)\circ}_{n,k}$ and $\mathcal{L}^{(c)\bullet}_{n,k}$ are, clearly,
\begin{equation} \label{eq:PathsSeries}
\begin{aligned}
    L^{(c)\circ}_{n,k} &= [\alpha^k] \prod_{r=1}^{n} \Bigl(\sum_{i\geq 0} \tilde{A}^{(c-r+1)}_i \alpha^{mi-1}\Bigr) = [\alpha^k] \alpha^{-n} \prod_{r=1}^{n} \tilde{A}^{(c-r+1)}(\alpha^m),\\
    L^{(c)\bullet}_{n,k} &= [\alpha^{-k}] \prod_{r=1}^{n} \Bigl(\sum_{i\geq 0} \tilde{B}^{(c-r+1)}_i \alpha^{1-mi}\Bigr) = [\alpha^{-k}] \alpha^{n} \prod_{r=1}^{n} \tilde{B}^{(c-r+1)}(\alpha^m).
\end{aligned}
\end{equation}

\begin{proof}[Proof of Proposition \ref{thm:CompositeSlices}]
We follow the slice decomposition of~\cite{AlbenqueBouttier2022}, while keeping track of the colours of vertices.

Let $M^{(c)}_{n,k}$ be a white slice of increment $k$, with $l$ sitting at a vertex of color $c$, and base of length $n$. Let us denote $v_1, \dotsc, v_n$ the sequence of vertices along the base from $l$ to $r$. Then the sequence of their labels can be encoded as a \L{}ukasiewicz path from $\mathcal{L}^{(c)\circ}_{n,k}$. If $M^{(c)}_{n,k}$ is a black slice, it works the same by writing the path from $r$ to $l$.

Furthermore, one can apply the slice decomposition to $M^{(c)}_{n,k}$. It consists in drawing from each vertex of the base the leftmost geodesics to the origin. After cutting the geodesics every time two of them meet, one obtains a concatenation of $n$ elementary white slices, whose bases are the edges of the base between $v_i$ and $v_{i+1})$, $i=1, \dotsc, n-1$, and whose increments are $\ell_o(v_{i+1}) - \ell_o(v_i)$, see Figure \ref{fig:SliceDecomposition}. 

\begin{figure}
    \centering
    \includegraphics[scale=.5]{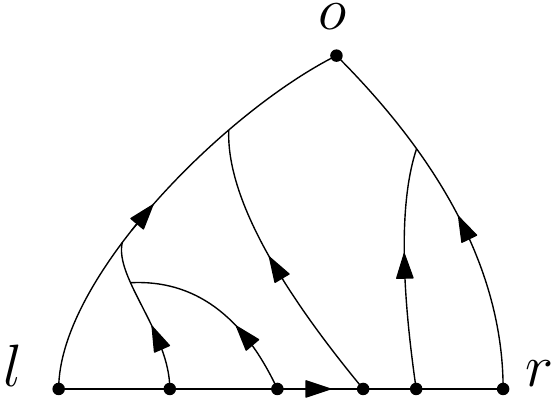}\hspace{2cm}
    \includegraphics[scale=.5]{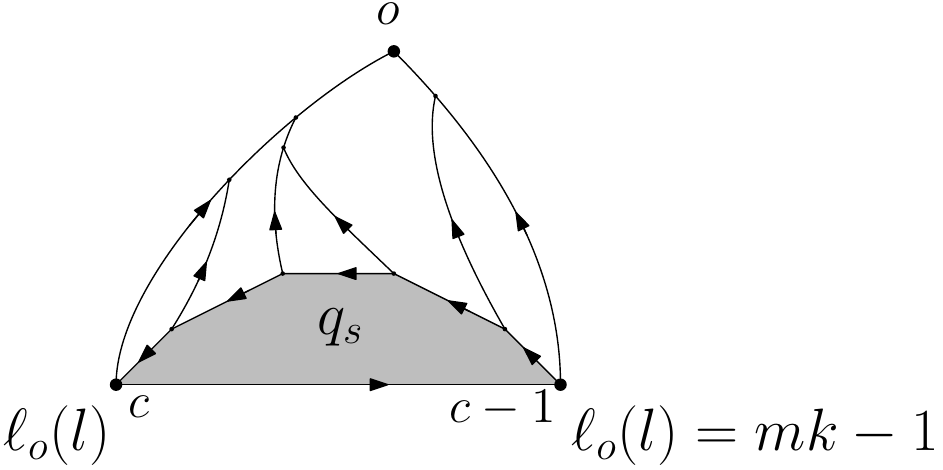}
    \caption{The slice decomposition of a white slice into elementary white slices (left). The slice decomposition applied to an elementary white slice after revealing the black face incident to the base (right).}
    \label{fig:SliceDecomposition}
\end{figure}

In terms of generating series, one can thus decorate the steps of the \L{}ukasiewicz paths of $\mathcal{L}^{(c)\circ}_{n,k}$ (respectively $\mathcal{L}^{(c)\bullet}_{n,k}$) with the series $\tilde{A}^{(c)}_k$s (respectively $\tilde{B}^{(c)}_k$). Notice that the edge of the base between $v_i$ and $v_{i+1}$ has color $c-i+1\mod m$. Therefore, if the $i$-th step is of type $(1,mk_i-1)$, it receives the weight $\tilde{A}^{(c-i+1)}_{k_i}$. This gives for $k\in\mathbb{Z}$
\begin{equation}
A_{n,k}^{(c)} =  L^{(c)\circ}_{n,k}
,\qquad
B_{n,k}^{(c)} =  L^{(c)\bullet}_{n,k}
.
\end{equation}
Here it is important to notice that elementary slices whose bases have compatible colours can be glued together along their full boundaries, since the colour along each oriented edge decreases by $1$ modulo $m$ (thus if the colours of two boundaries to be glued are compatible at the beginning, they are compatible throughout). We have thus proved \eqref{CompositeSlices}. 

To prove the relations in Definition \ref{def:WkBk} for $\tilde{A}_k^{(c)}$, consider an elementary white slice. For $k=0$, it can be reduced to a single edge, where the corners $o$ and $r$ are the same, with weight $\bar{u}_c$. If not, consider the black face to the left of its base. This face can have arbitrary degree $ms$ for $s\geq 1$. After removing the base, the white elementary slice becomes a black slice with base length $ms-1$. It is therefore found that
\begin{equation*}
    \tilde{A}_k^{(c)} = \bar{u}_c\delta_{k,0} + \sum_{s\geq 1} q_{s} \tilde{t}^{s} B^{(c-1)}_{ms-1,mk-1}.
\end{equation*}
This is illustrated in Figure \ref{fig:SliceDecomposition}. Similarly $\tilde{B}_k^{(c)} = \sum_{s\geq 1} p_{s} A^{(c-1)}_{ms-1,mk-1}$ for $k\geq 1$, while a black slice of increment $-1$ is by definition reduced to its base, implying $B_0^{(c)}=1$. One can thus write
\begin{equation} \label{eq:tildeAtildeB}
    \begin{aligned}
        \tilde{A}^{(c)}(z) &= \bar{u}_c + \sum_{s=1}^{D_2} q_s \tilde{t}^s \left\{ z^s \frac{\prod_{c'=0}^{m-1} \tilde{B}^{(c')}(z)^s}{\tilde{B}^{(c)}(z)}\right\}^{\geq},\\
        \tilde{B}^{(c)}(z) &= 1 + \sum_{s=1}^{D_1} p_s \left[ z^{-s} \frac{\prod_{c'=0}^{m-1} \tilde{A}^{(c')}(z)^s}{\tilde{A}^{(c)}(z)}\right]^{<}.
    \end{aligned}
\end{equation}
By comparing with Definition \ref{def:WkBk}, one finds $\tilde{A}^{(c)}(z) = \bar{u}_c A^{(c)}(z\prod_{i\in I} u_i)$ and $\tilde{B}^{(c)}(z) = B^{(c)}(z\prod_{i\in I} u_i)$.
\end{proof}

\subsection{Expression of $W_{0,1}$}

Here we prove Theorem \ref{thm:Discrat} in the case of polynomial $G$.

\begin{thm} \label{thm:PolynomialW01}
If $J=\emptyset$, then
\begin{equation*}
    W_{0,1}(x) = Y(Z(x)) - \sum_{s=1}^{D_1} p_s x^{s-1},\;\; \text{with}\quad 
    zX(z) = \prod_{c=1}^{m-1} A^{(c)}(z),\;\; X(z)Y(z) = \bar{u}_c (A^{(c)}(z) B^{(c)}(z)-1).
\end{equation*}
\end{thm}

The proof is twofold. The first step relies on a bijective argument originally due, in the case of maps, to Bouttier and Guitter \cite[Section 3.3]{BouttierGuitterContinuedFractions} which provides an expression for $W_{0,1}$. This argument was further used in \cite{AlbenqueBouttier:fpsac12} to derive $W_{0,1}$ in the case of constellations with $J=\emptyset$ and $q_k=\delta_{k,1}$ (in our notation). The second step is a calculation to match the bijective expression of $W_{0,1}$ with the RHS of the above theorem, i.e.~$Y(Z(x)) - \sum_{s=1}^{D_1} p_s x^{s-1}$. This two-step proof is due to \cite{AlbenqueBouttier2022} in the case of hypermaps, which we follow by merely adding color-dependent weights.




An \emph{$m$-excursion} is an $m$-path starting at $(0,0)$, ending at $(n,0)$ for some $n$ and staying above~$0$. Our $m$-excursions will always be weighted like $m$-paths, with an extra weight $x^{-1/m}$ per step. That is to say, the weights are $x^{-1/m}\tilde{A}^{(c')}_k$ for each step at color $c'$ of type $(1,mk-1)$, $k\geq 0$.

Notice that if an $m$-path crosses the horizontal line of height $h$ at the color $c$, then it can only cross it again with the same color. This is because along each step of the path the variation of abscissa and height are equal modulo $m$ (to $-1$), therefore they are also equal along any interval.

Let $Z^{(c)}$ be the generating series of $m$-excursions with an additional down step below 0. We can write schematically
\begin{equation*}
    Z^{(c)} = \begin{array}{c} \includegraphics[scale=.5]{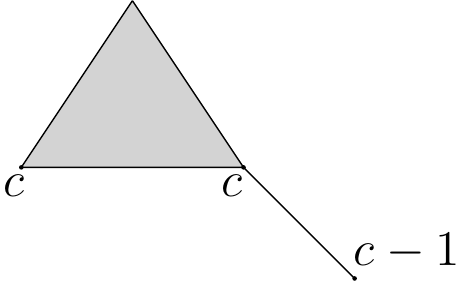} \end{array}
\end{equation*}
Let $\tilde{Z}$ then be the generating series of $m$-paths starting at $(0,0)$ with an arbitrary color $c$ and finishing at $(ms,-m)$ while staying above $-m$ until the last step, also counted with weight $x^{-1/m}\tilde{A}^{(c')}_k$ for each step at color $c'$ of type $(1,mk-1)$.

We perform a first passage decomposition. Such a path starts with an excursion at color, say $c$, until it goes below 0 for the first time with a down step. The associated weight is $Z^{(c)}$. The down step goes from color $c$ to $c-1$. The path then follows an excursion starting at color $c-1$ until it goes below $-1$ for the first time and so on. There are exactly $m$ excursions of this type to reach height $-m$. This gives
\begin{equation*}
    \tilde{Z} = \begin{array}{c} \includegraphics[scale=.5]{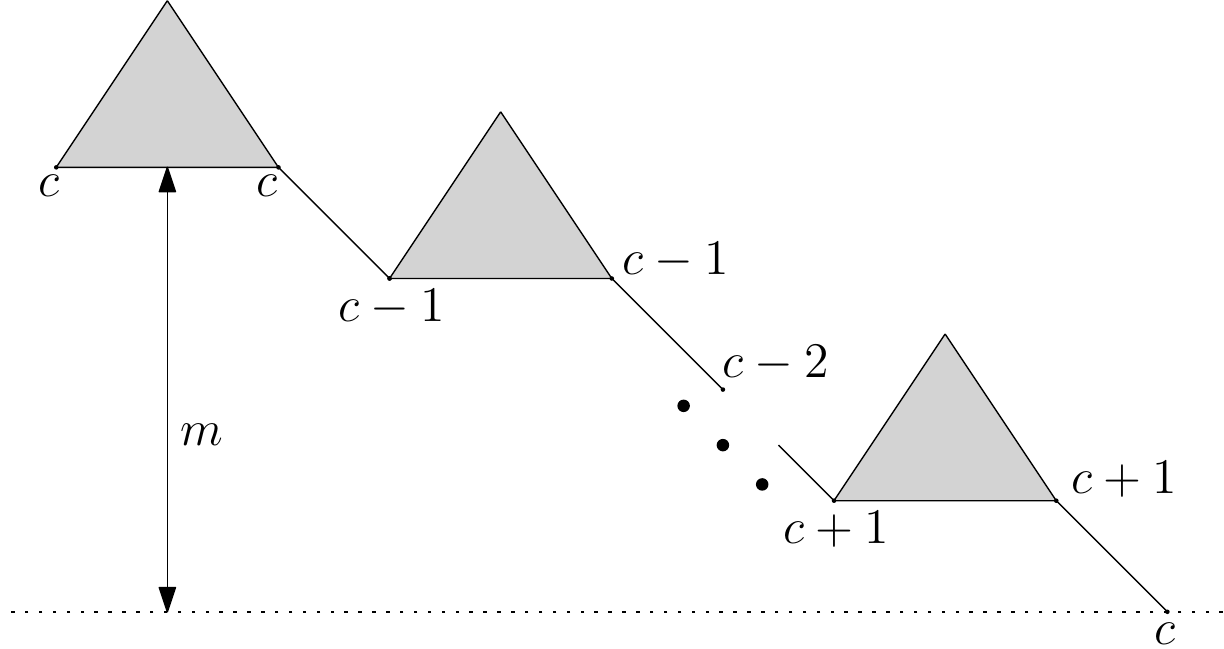}\end{array} =  \prod_{r=1}^{m} Z^{(c-r)}.
\end{equation*}

We then perform a first return decomposition to write an equation on $Z^{(c)}$. An excursion counted by $Z^{(c)}$ starts with an arbitrary step of type $(1,mk-1)$. If $k=0$, the $m$-excursion is over. If $k>0$, 
we perform again a last passage decomposition, decomposing the path into excursions above heights $c+mk-1, c+mk-2, \dots, c$.
These $mk$ excursions can be arranged in consecutive groups of $m$ in which each group is formed by an excursion starting at colour $i$ for all $i \in[0..m-1]$. The generating function for each group is precisely the object counted by $\tilde{Z}$ (up to circular permutation of excursions). Pictorially, we have

\begin{center} \includegraphics[scale=.5]{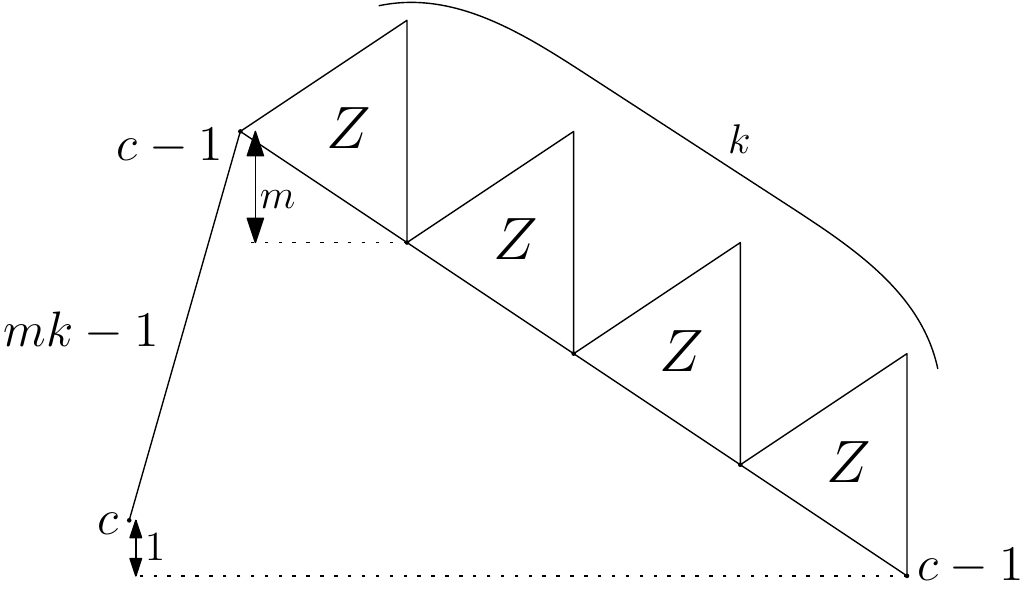} \end{center}

\noindent which translates into the equation:
\begin{equation*}
    Z^{(c)}(x) = x^{-1/m} \sum_{k=0}^{D_2} \tilde{A}_k^{(c)} 
    \tilde{Z}(x)^k
    = x^{-1/m} \tilde{A}^{(c)}(\tilde{Z}(x)).
\end{equation*}

We can in particular deduce that $\tilde{A}^{(c)}(\tilde{Z}(x))$ is the series of $m$-excursions with a final down step below 0 for which we ignore the weight $x^{-1/m}$ and that $\tilde{Z}(x)$ is defined by
\begin{equation} \label{eq:ZtildeSeries}
    \tilde{Z}(x) = \frac{1}{x} \prod_{c=0}^{m-1} \tilde{A}^{(c)}(\tilde{Z}(x)).
\end{equation}
The series $\tilde{Z}(x)$ will play in the bijective proofs below, the same role as $Z(x)$ in Section \ref{sec:MainResults} (the fact that we need to work with these two different quantities comes from the change of variable~\eqref{eq:changeOfVariables}).

With what we have done so far, it would be relatively easy to enumerate planar $m$-constellations which are both rooted and pointed. But the true power of the slice constructions (with origins in~\cite{BouttierGuitterContinuedFractions}) lies in the possibility to count maps which are \emph{rooted only}, via subtle difference arguments.
We apply this program here to $m$-constellations (again following~\cite{AlbenqueBouttier2022}). 
\begin{prop} \label{thm:BijectiveW01} We have
\begin{equation*}
    xW_{0,1}(x) = \tilde{A}^{(c)}(\tilde{Z}(x)) - \bar{u}_c - \tilde{A}^{(c)}(\tilde{Z}(x)) \sum_{s=1}^{D_1} p_s \sum_{k\geq 0} \tilde{Z}(x)^k [\alpha^{mk+1}] \prod_{r=1}^{ms-1} \Bigl(\sum_{i=0}^{D_2} \tilde{A}^{(c)}_i \alpha^{mi-1}\Bigr).
\end{equation*}
\end{prop}

\begin{proof}
In a rooted $m$-constellation, we recall that the root is a marked right path. It is incident to exactly one white face which we call the root face. For $c\in[0..m-1]$, a root vertex can also be designated to be the vertex of color $c$ on the root.

Let $\mathcal{M}^*$ be the set of connected, pointed rooted planar $m$-constellations with a white root, and $\mathcal{M}$ the set of connected, rooted planar $m$-constellations with a white root. Let $d\geq 0$ and $c\in[0..m-1]$. We introduce $\mathcal{M}^*_{d}$ the set of connected rooted, pointed $m$-constellations with pointed vertex $v^*$ and a root vertex $\vec{v}$ of color $c$, such that
\begin{itemize}
    \item $\ell_{v^*}(\vec{v}) = d$ where $\vec{v}$ is the root vertex,
    \item all vertices $v$ on the root face satisfy $\ell_{v^*}(v) \geq d$.
\end{itemize}
Notice that $\mathcal{M} = \mathcal{M}^*_{d=0}$ because $d=0$ enforces $v^*$ to coincide with the root vertex. Denoting
\begin{equation*}
    P^{(c)}_d(x) = \sum_{M\in \mathcal{M}^\bullet_d} x^{-f(M)} w(M),
\end{equation*}
with $f(M)$ being the degree of the root face, and $P^{(c)}(x) = \sum_{d\geq 0} P^{(c)}_d(x)$, we have (for any choice of $c$)
\begin{equation*}
   x W_{0,1}(x) = P^{(c)}(x) - \sum_{d\geq 1} P^{(c)}_d(x).
\end{equation*}
The two contributions can be evaluated independently.

First $P^{(c)}(x)$ counts $m$-constellations where the label of the root vertex is minimal among the labels of the root face. In terms of slices, they correspond to white slices of increment 0, whose base is formed by the edges of the root face, and such that the sequence of labels along the base forms an $m$-excursion (with no final down step below 0). This gives, as in \cite{AlbenqueBouttier2022},
\begin{equation*}
        P^{(c)}(x) = \left(\frac{\tilde{A}^{(c)}(\tilde{Z}(x))}{\tilde{A}^{(c)}_0}-1\right) \bar{u}_c.
\end{equation*}



The second term $\sum_{d\geq 1} P^{(c)}_d(x)$ consists of maps where the pointed vertex is not on the root face. 
By repeating the steps of \cite{AlbenqueBouttier2022} with colored vertices, it can be evaluated as follows.
Let $d>0$ be the label of the root vertex. It is then connected by an edge to at least one vertex of label $d-1$ which is not on the root face. Consider the rightmost edge of this type, and the white face $f$ to its right, which has a degree, say, $ms$ for $s\in [1..D_1]$. 
%
%
One then splits the root vertex into two vertices $v,v'$ so that $f$ is merged with the root face (after splitting, the vertex $v'$ has a new label of the form $d+mk$). By analyzing the sequence of label variations along the face obtained after splitting and using last passage decompositions, one obtains as in~\cite{AlbenqueBouttier2022}
\begin{equation*}
    \sum_{d\geq 1} P^{(c)}_d(x) = \tilde{A}^{(c)}(\tilde{Z}(x)) \sum_{k>0} \tilde{Z}(x)^k \sum_{s=1}^{D_1} p_s [\alpha^{mk+1}] \prod_{r=1}^{ms-1} \Bigl(\sum_{i=0}^{D_2} \tilde{A}^{(c-r)}_i \alpha^{mi-1}\Bigr)
    \end{equation*}
and therefore
\begin{equation} \label{eq:BijectiveW01intermediate}
    xW_{0,1}(x) = \left(\frac{\tilde{A}^{(c)}(\tilde{Z}(x))}{\tilde{A}^{(c)}_0}-1\right)\bar{u}_c - \tilde{A}^{(c)}(\tilde{Z}(x)) \sum_{k>0} \tilde{Z}(x)^k \sum_{s=1}^{D_1} p_s [\alpha^{mk+1}] \prod_{r=1}^{ms-1} \Bigl(\sum_{i=0}^{D_2} \tilde{A}^{(c-r)}_i \alpha^{mi-1}\Bigr).
\end{equation}
This is not quite the RHS of \eqref{thm:BijectiveW01} yet. To turn \eqref{eq:BijectiveW01intermediate} into the RHS of \eqref{thm:BijectiveW01}, we use
\begin{equation} \label{eq:SecondDecomposition}
    \tilde{A}^{(c)}_0 = \bar{u}_c + \tilde{A}^{(c)}_0\sum_{s=1}^{D_1} p_s [\alpha^{1}] \prod_{r=1}^{ms-1} \Bigl(\sum_{i=0}^{D_2} \tilde{A}^{(c-r)}_i \alpha^{mi-1}\Bigr).
\end{equation}
This equation can be proved by performing the slice decomposition on elementary slices of increment $-1$, after unraveling the white face to the right of the bottom edge of the left boundary as in \cite{AlbenqueBouttier2022}.
%
\end{proof}

\begin{proof}[Proof of Theorem \ref{thm:PolynomialW01}]
A direct calculation shows that $xY(Z(x)) - \sum_{s=1}^{D_1} p_s x^s$ yields the RHS of Proposition \ref{thm:BijectiveW01}. Indeed, we find
\begin{equation*}
    x Y(Z(x)) = \bar{u}_c A^{(c)}(Z(x)) B^{(c)}(Z(x)) - \bar{u}_c = \tilde{A}^{(c)}(\tilde{Z}(x)) \tilde{B}^{(c)}(\tilde{Z}(x)) - \bar{u}_c,
\end{equation*}
where the first equality is from Definition \ref{def:Hrat} and \ref{def:SpectralCurve}, and the second from Proposition \ref{thm:CompositeSlices} and observing that
\begin{equation*}
\tilde{Z}(x) = Z(x) \prod_{c=0}^{m-1} \bar{u}_c,
\end{equation*}
so that $\bar{u}_c A^{(c)}(Z(x))= \tilde{A}^{(c)}(\tilde{Z}(x))$ and $B^{(c)}(Z(x))= \tilde{B}^{(c)}(\tilde{Z}(x))$.

We then eliminate $\tilde{B}^{(c)}(\tilde{Z}(x))$ with \eqref{eq:tildeAtildeB}, and use $[f(z)]^{<} = f(z) - [f(z)]^\geq$ for a Laurent polynomial, to obtain
\begin{equation} \label{eq:Y}
   x Y(Z(x)) - \sum_{s=1}^{D_1} p_s x^s = \tilde{A}^{(c)}(\tilde{Z}(x)) - \bar{u}_c - \tilde{A}^{(c)}(\tilde{Z}(x)) \sum_{s=1}^{D_1} p_s \biggl[z^{-s} \frac{\prod_{c'=0}^{m-1} \tilde{A}^{(c')}(z)^s}{\tilde{A}^{(c)}(z)}\biggr]^\geq_{z=\tilde{Z}(x)}.
\end{equation}
Finally we simply rewrite the last term as
\begin{align*}
    \biggl[z^{-s} \frac{\prod_{c'=0}^{m-1} \tilde{A}^{(c')}(z)^s}{\tilde{A}^{(c)}(z)}\biggr]^\geq_{z=\tilde{Z}(x)} &= \sum_{k\geq 0} \tilde{Z}(x)^k [\alpha^{k}] \biggl[z^{-s} \frac{\prod_{c'=0}^{m-1} \tilde{A}^{(c')}(z)^s}{\tilde{A}^{(c)}(z)}\biggr] \\
    &= \sum_{k\geq 0} \tilde{Z}(x)^k [\alpha^{mk+1}] \prod_{r=1}^{ms-1} \Bigl(\sum_{i\geq 0} \tilde{A}^{(c)}_i \alpha^{mi-1}\Bigr).
\end{align*}
This turns the RHS of \eqref{eq:Y} into the RHS of Proposition \ref{thm:BijectiveW01}, thus proving Theorem \ref{thm:PolynomialW01}.
\end{proof}

\subsection{Expression of $W_{0,2}$ with slices}

Let $p>0$. A $p$-\emph{annular} constellation is a white rooted constellation with a marked black face of degree $mp$ such that any oriented cycle separating the root face from the marked face with the root face to its right has length larger than or equal to $mp$. A $p$-\emph{strict annular} constellation is a white rooted constellation with a marked white face of degree $mp$ such that any oriented cycle separating the root face from the marked face with the root face to its left has length larger than $mp$.

The following proposition is originally due to Bouttier and Guitter \cite[Section 7]{BouttierGuitter2014} in the context of irreducible maps, and it was generalised to hypermaps in \cite{AlbenqueBouttier2022}, which we again follow.

\begin{prop} \label{thm:Annular}
There is a bijection between white slices of increment $-mp$ and $p$-annular maps. There is a bijection between white slices of increment $mp$ and $p$-strict annular maps.
\end{prop}

Notice that a slice of increment a multiple of $m$ has base of length also a multiple of $m$.

One direction in those bijections is simple to explain. Consider a white slice of increment $-pm<0$ and glue the left and right boundaries together starting by identifying the corners $l$ and $r$ (note that this gluing respects the colours of vertices of the boundary since the increment is a multiple of $m$). Since the left boundary is longer than the right boundary, one obtains a white rooted $m$-constellation with a ``hole''. The edges around it all have white faces on their other side. The hole can therefore be colored black and this gives a white rooted $m$-constellation with an additional marked black face.

Let $v^*$ be the origin of the slice and $v$ the root vertex. After the gluing, one obtains a geodesic of length $\ell_o(r)$ from $v$ to $v^*$, while the degree of the marked face is minus the increment, $\ell_o(l) - \ell_o(r) = -mp$. Assume that there is a cycle separating the root face from the marked face, of length less than $mp$. Then this would create a path from $v$ to $v^*$ winding around the marked face which would be shorter than $\ell_o(l)$. In the slice, this in turn would correspond to a path from $l$ to $o$ of length shorter than $\ell_o(l)$, which is impossible. One thus obtains a $p$-annular map. The argument is similar for slices of increment $mp$, by exchanging the roles of the left and right boundaries. However, since the right boundary is the unique geodesic from $r$ to $o$ (as opposed to a geodesic for the left boundary), one obtains a $p$-strict annular map.

The other direction in the bijections (i.e.~from annular maps to slices) is explained in \cite[Section 7]{BouttierGuitter2014} and \cite{AlbenqueBouttier2022}.

Let $\mathcal{M}_{0,2}$ be the set of connected, planar $m$-constellations with two labeled, marked right paths which do not lie in the same white face. For $M\in\mathcal{M}_{0,2}$, we denote $F_1, F_2$ the two white faces which have the marked right paths, and $mf_1(M), mf_2(M)$ their degrees. The series $W_{0,2}(x_1,x_2)$ is the generating series of those objects, counted with $x_1^{-f_1(M)-1} x_2^{-f_2(M)-1}$.

\begin{prop}\label{prop:W02sum} We have
\begin{equation*}
    W_{0,2}(x_1, x_2) = \sum_{f_1, f_2} x_1^{-f_1-1} x_2^{-f_2-1} \sum_{p\geq 1} p \left([\alpha^{-p}] \alpha^{-f_1} \prod_{c=0}^{m-1} \tilde{A}^{(c)}(\alpha)^{f_1}\right) \left( [z^{p}] z^{-f_2} \prod_{c=0}^{m-1} \tilde{A}^{(c)}(z)^{f_2}\right).
\end{equation*}
\end{prop}

\begin{proof}
Let $M\in\mathcal{M}_{0,2}$. We consider the unique shortest cycle separating $F_1$ and $F_2$ which has $F_2$ on its left and is as close as possible to it. Say it has length $mp$ for $p\geq 1$. Let us cut $M$ along that cycle. By definition, we obtain to its left a $p$-strict annular map and to its right a $p$-annular map whose bases have respective degrees $mf_1(M), mf_2(M)$.

Using Proposition \ref{thm:Annular}, we now have to enumerate white slices of increment $-p$ with base of length $mf_1(M)$, and white slices of increment $p$ with base of length $mf_2(M)$. The former correspond to $m$-paths starting at $(0,0)$ and ending at $(mf_1(M), -mp)$, while the latter end at $(mf_2(M), mp)$. They are respectively
\begin{align*}
    [\alpha^{-mp}] \prod_{r=1}^{mf_1} \Bigl(\sum_{i\geq 0} \tilde{A}_i^{(c-r+1)} \alpha^{mi-1}\Bigr) &= [\alpha^{-p}] \alpha^{-f_1} \prod_{c=0}^{m-1} \tilde{A}^{(c)}(\alpha)^{f_1}\\
    \text{and}\quad 
    [\alpha^{mp}] \prod_{r=1}^{mf_2} \Bigl(\sum_{i\geq 0} \tilde{A}_i^{(c-r+1)} \alpha^{mi-1}\Bigr) &= [\alpha^{p}] \alpha^{-f_2} \prod_{c=0}^{m-1} \tilde{A}^{(c)}(\alpha)^{f_2}.
\end{align*}
Notice that they do not depend on the color chosen to start the paths at, as expected.

Given a $p$-annular constellation and a $p$-strict annular constellation, we can glue them along their marked faces, since they both have degree $mp$, one is white and the other black. There are moreover $p$ ways to do so (and not $mp$ because the colors of the vertices have to match for the gluing to be allowed).
\end{proof}

\subsection{Rational generating functions and final expression of~$W_{0,2}$}

By performing the (formal) summation over $f_1,f_2$, Proposition~\ref{prop:W02sum} rewrites
\begin{equation}\label{eq:W02new}
    W_{0,2}(x_1, x_2) = x_1^{-1}x_2^{-1}
   \sum_{p\geq 1} p \left([z^{-p}] F(x_1,z) \right)\left([z^{p}] F(x_2,z) \right),
    %
\end{equation}
where (note that in this section we consider series primarily in $\bar{x}$)
$$F(x,z)\coloneqq\frac{1}{1-\bar{x}\tilde{A}(z)/z} \ \  \in \mathbb{Q}[\mathbf{p},\mathbf{q},\mathbf{\bar{u}}][[t]][z,z^{-1}][[\bar{x}]]
$$
and $\tilde{A}(z)\coloneqq\prod_{c=0}^{m-1} \tilde{A}^{(c)}(z) \in \mathbb{Q}[\mathbf{p},\mathbf{q},\mathbf{\bar{u}}][[t]][z]$. We will now compute the quantity $[z^{p}] F(x,z)$ for $p\in \mathbb{Z}$ using standard arguments from combinatorics of paths and rational generating functions.
Write $F(x,z)= \left[F(x,z)\right]^\leq + \left[F(x,z)\right]^>$, where we recall that the negative and positive symbols are relative to the exponents of the variable $z$.
\begin{lem}  We have
\begin{align}
\left[F(x,z)\right]^< = \frac{D(x)}{1-\tilde{Z}(x)/z},
\end{align}
where $D(x)=-x\dfrac{\tilde{Z}'(x)}{\tilde{Z}(x)}$.
\end{lem}
\begin{rem}
The statement is equivalent to saying that for $p\geq 0$, we have
$[z^{-p}] F(x,z) = D(x) \tilde{Z}(x)^p.$ Readers familiar with combinatorics of paths will immediately recognise here an identity of \emph{last passage decomposition}, in the spirit of the previous sections. We leave our combinatorially experienced readers the pleasure to prove the lemma along such lines.
Instead, we will propose another proof based on rational generating functions and a discussion on \emph{small} and \emph{large} roots, which we find interesting on its own and which gives directly the value of $D(x)$ by algebraic computations.
\end{rem}
\begin{proof}
Since the field of Puiseux series is algebraically closed,  the polynomial equation 
\begin{align}\label{eq:roots}
1-\bar{x}\tilde{A}(z)/z =0
\end{align}
has $mD_2$ roots in $\hat{\mathbb{K}}((\bar{x}^*))$, where $\hat{\mathbb{K}}$ is an algebraic closure of $\mathbb{Q}(\mathbf{p}, \mathbf{q}, \mathbf{u})((t))$. One of those roots is the series $\tilde{Z}$ defined in~\eqref{eq:ZtildeSeries}, and it is the unique root which is a \emph{power} series in $\bar{x}$ (involving no negative exponent of $\bar{x}$; this is usually called a \emph{small root}),
$$
\tilde{Z} = c \bar{x} + O(\bar{x}^2),
$$
with $c=\prod_{i\in I} \bar{u}_i$.
Indeed, uniqueness is clear since the expansion can be computed recursively from the equation.
We let $\zeta_1,\dots,\zeta_{mD_2-1}\in \hat{\mathbb{K}}((\bar{x}^*))$ be the other roots of~\eqref{eq:roots}, and they are necessarily all \emph{large}, in the sense that they start with a negative (possibly fractional) exponent of $\bar{x}$.
We can then write, by factorizing the polynomial~\eqref{eq:roots} over its roots,
\begin{align}
F(x,z) = \frac{E}{(1-\tilde{Z}/z) \prod_{i=1}^{mD_2-1} (1- z/\zeta_i)},
\end{align}
for some series $E\in \hat{\mathbb{K}}((\bar{x}^*))$.
This equality is valid in $\hat{\mathbb{K}}[z,z^{-1}]((\bar{x}^*))$. Note that all factors $(1-\cdot)$ in the denominator involve a quantity which starts with a positive power of $\bar{x}$.
We can thus perform a (partial\footnote{At this stage we could perform a (full) partial fraction expansion of $[F(x,z)]^{\geq}$ and get a closed expression for $[z^{p}] F(x,z)$ for positive $p$ involving the large roots, but we will not need to do that.}) partial fraction decomposition of the form:
\begin{align}\label{eq:parfrac}
F(x,z) = \frac{D(x)}{(1-\tilde{Z}/z)} + \frac{C(x) z}{\prod_{i=1}^{mD_2-1} (1- z/\zeta_i)},
\end{align}
isolating the contribution of the unique small root, for some series $B(x),C(x)$ which are again in $\hat{\mathbb{K}}((\bar{x}^*))$.
The last equation clearly separates $F(x,z)$ into a nonpositive and positive part, so the only thing remaining to show is that $D(x)$ is given by the equation of the Lemma.
From~\eqref{eq:parfrac} we have 
$$
D(x) = \Big( (1-\tilde{Z}(x)/z) F(x,z)\Big)_{z=\tilde{Z}(x)}
$$
(note that the substitution is valid since the quantity in larger parentheses
involves no negative powers of $z$ in its expansion).
Therefore, going back to the definition of $F$ we find
$$D(x)=
\Big( \frac{z-\tilde{Z}(x)}{z-\bar{x} \tilde{A}(z)} \Big)_{z=\tilde{Z}(x)}
=
\frac{1}{1-\bar{x}\tilde{A}'(\tilde{Z}(x))}.
$$
Now, by differentiating the relation $\tilde{Z}(x) = \bar{x} \tilde{A}(\tilde{Z}(x))$ with respect to $x$
we obtain
$\tilde{Z}'(x) ( 1-\bar{x} \tilde{A}'(\tilde{Z}(x)) ) = - \bar{x} \tilde{Z}(x)
$,
which immediately gives 
$D(x) = - x \tilde{Z}'(x)/\tilde{Z}(x)$.
\end{proof}

Going back to~\eqref{eq:W02new} we can write (with $\tilde{Z}_i=\tilde{Z}(x_i)$ and $D_i=D(x_i)$, for $i=1,2$ )
\begin{align*}
    W_{0,2}(x_1, x_2) &= x_1^{-1}x_2^{-1}
   \sum_{p\geq 1} p D_1 \tilde{Z}_1^p \left([z^{p}] F(x_2,z) \right)\\
   &= x_1^{-1}x_2^{-1} D_1 \left(z\frac{ d}{dz} [F(x_2,z)]^>\right)_{z=\tilde{Z}_1} \\
 &= x_1^{-1}x_2^{-1} D_1 \left(z\frac{ d}{dz} \left(\frac{1}{1-\bar{x}_2\tilde{A}(z)/z} - \frac{D_2}{(1-\tilde{Z}_2/z)} \right)\right)_{z=\tilde{Z}_1} \\
  &= x_1^{-1}x_2^{-1} D_1 \left(\frac{\tilde{Z}_1}{\tilde{Z}'(x_1)} \frac{d}{dx_1} \left( \frac{1}{1-\bar{x_2}x_1} \right)
  + 
 \frac{D_2 \tilde{Z}_1\tilde{Z}_2}{(\tilde{Z}_1-\tilde{Z}_2)^2} \right)\\
  &=  \frac{x_1^{-1}D_1 \tilde{Z_1}}{\tilde{Z}'(x_1)} \frac{1}{(x_2-x_1)^2} 
  + 
 \frac{x_1^{-1}x_2^{-1} D_1 D_2 \tilde{Z}_1\tilde{Z}_2}{(\tilde{Z}_1-\tilde{Z}_2)^2},
     %
\end{align*}
where in the fourth equality we have used that $\tilde{Z}(x_1)=\bar{x_1} \tilde{A}(\tilde{Z}(x_1))$.
Substituting the expression  $D(x)=-x\frac{\tilde{Z}'(x)}{\tilde{Z}(x)}$, we obtain the same expression from $W_{0,2}$ as that given in Theorem~\ref{thm:cylinder}, except that it involves the function $\tilde{Z}(x_1)$ and $\tilde{Z}(x_2)$. Due to homogeneity  we can replace them with $Z(x_1)$ and $Z(x_2)$ (note, going back to definitions, that $\tilde{Z}(x) = \left(\prod_{i\in I} \bar{u}_i\right) Z( x)$). This concludes the proof of this theorem.

\section{Proofs for $W_{0,1}$ and $W_{0,2}$ for rational $G(z)$}
\label{sec:rational}

\subsection{Proof of  Theorem~\ref{thm:Discrat} }

In the case of polynomial $G(z)$ (i.e.~$r=|J|=0$, or equivalently $M=m$), Theorem~\ref{thm:Discrat} is proved in Section~\ref{sec:constellations} by combinatorial techniques. We will now deduce the general case, i.e. $G(z)$ rational, from the polynomial case.
In order to prove
\begin{equation}\label{eq:toprove}
W_{0,1}(x) + \sum_{k=1}^{D_1}p_k x^{k-1} = \frac{1}{x} H^c(\tilde{Z}(x)),
\end{equation}
the main task will be to show that the coefficients on both sides depend ``nicely'' in the parameters $\mathbf{u}$. Together with the fact that this equation holds in the polynomial case, and with the fact that we are able to take $m$ arbitrarily large, this will be enough to conclude.

\subsubsection{Analysis of the coefficients}

We start with the left-hand side of~\eqref{eq:toprove}.
\begin{lem}
Let $k\geq 0$ and let $P^{(m,r)}_k\coloneqq [t^k] W_{0,1}(x) + \sum_{k=1}^{D_1}p_k x^{k-1}$.
Then $P^{(m,r)}_k$ is a polynomial (with coefficients in $\mathbb{Q}[\mathbf{p}, \mathbf{q},\bar{x}]$) in the variables $\mathbf{u}_I$ and $\mathbf{u}_J$, which is symmetric in the $\mathbf{u}_I$, symmetric in the $\mathbf{u}_J$, and which has homogeneous degree at most $2k-1$ in these variables.
\end{lem}
\begin{proof}
Symmetry is clear.
The fact that $W_{0,1}$ is extracted from the function $F_0$ implies that for a monomial of the form $t^n p_\lambda q_\mu x^{-\ell-1} \prod_i u_i^{d_i}\prod_i v_i^{e_i}$ one has
$$
\sum d_i + \sum e_i = \ell(\lambda)+\ell(\mu) +1 + 2g -2,
$$
with here $g=0$ (of course this is only the Riemann--Hurwitz relation for the associated topological objects). Therefore $\sum d_i + \sum e_i \leq 2k-1$. 
\end{proof}

In what follows we let $e_\ell$ denote the $\ell$-th elementary symmetric function in a set of variables. We declare it to have degree $\ell$.
\begin{lem}\label{lem:projectiveLimit}
The quantity $P^{(m,r)}_k$ defined in the previous lemma is a polynomial (with coefficients in $\mathbb{Q}[\mathbf{p}, \mathbf{q},\bar{x}]$) in the elementary symmetric functions $e_1(\mathbf{u}_I),\dots,e_{2k-1}(\mathbf{u}_I)$ of the $\mathbf{u}_I$ and  $e_1(\mathbf{u}_J),\dots,e_{2k-1}(\mathbf{u}_J)$  of the $\mathbf{u}_J$. 
This polynomial is independent of the values of $m$ and $r$, i.e.~we can write
$$P^{(m,r)}_k
=
P^{(\infty)}_k \Bigl(\big(e_\ell(\mathbf{u}_I)\big)_{\ell< 2k}, \big(e_\ell(\mathbf{u}_J)\big)_{\ell< 2k}\Bigr), 
$$
for some polynomial $P^{(\infty)}_k$.
\end{lem}
\begin{proof}
The first assertion is a direct consequence of the symmetry and the degree bound of the previous lemma. 
The fact that it is independent of $m$ and $r$ 
follows from considering the map which sets the last variable in $\mathbf{u}_I$ or $\mathbf{u}_J$ to zero. Namely, write, making explicit the dependency of $\tilde G$ in $m=|I|$ and $r=|J|$,
$$
{G}^{(m,r)}(z) \coloneqq  \frac{\prod_{i=0}^{m-1} ( 1+z u_i)}{\prod_{j=1}^{r} (1+ z u_{j+m-1})},
$$
then clearly, for $r\geq 1$,
$$
{G}^{(m,r)}\Big|_{u_{m+r-1}=0} =  \tilde{G}^{(m,r-1)}(z) .
$$
On coefficients of $W_{0,1}$ this identity implies:  
$$P^{(m,r)}_k(e_1,\dots,e_{2k},f_1,\dots,f_{2k}) \big|_{\substack{e_i = e_i(u_0,\dots,u_{m-1}) \\ f_i = f_i(u_{m},\dots,u_{M-2},0)} }
=P^{(m,r-1)}_k(e_1,\dots,e_{2k},f_1,\dots,f_{2k}) \big|_{\substack{e_i = e_i(u_0,\dots,u_{m-1}) \\ f_i = f_i(u_{m},\dots,u_{M-2})}}. 
$$
If both $m$ and $r$ are larger than $2k$, this implies that $P^{(m,r)}_k = P^{(m,r-1)}_k$ (since the first $2k$ elementary symmetric functions in more than $2k$ variables are algebraically independent).
The same argument (considering now the application setting $u_m$ to zero) implies that $P^{(m,r)}_k = P^{(m-1,r)}_k$, when $m$ and $r$ are both larger than $2k$.

One can thus set $P^{(\infty)}_k\coloneqq P^{(m,r)}_k$ for a fixed pair $(m,r)$ large enough, and from the previous discussion this choice will be valid for all $(m,r)$.
\end{proof}

We now turn to the right-hand side of~\eqref{eq:toprove}.
We start by analysing~\eqref{eq:defWratconv2}-\eqref{eq:defBratconv2} in detail. 
From these equations it is clear that at each order in $t$, the coefficients of $A^{(c)}$ and $B^{(c)}$ are polynomials in the $\mathbf{u},\mathbf{p}, \mathbf{q}$. Moreover, if we declare the variables
$u_i$ to have degree $1$ and the variables $p_i, q_j$ to have degree $-1$, these polynomials are all homogeneous of degree $0$ (the degree of $t$ is declared zero, and so is the degree of $x$).
Now, from equation~\eqref{eq:ZasExcursionsrat}, it is clear that 
 the coefficients of $\tilde{Z}(x)$ at any order are also polynomials in the $\mathbf{u},\mathbf{p}, \mathbf{q}$, of homogeneous degree $0$.
Therefore, the coefficients at any order in $t$, in the quantity 
$$\frac{u_c}{x} \tilde{H}^c(\tilde{Z}(x))
= \frac{1}{x} \left(A^{(c)}(\tilde{Z})B^{(c)}(\tilde{Z})-1\right)
$$
are polynomials in the $\mathbf{u},\mathbf{p}, \mathbf{q}$, of homogeneous degree $0$.

We can now analyse the right-hand side of~\eqref{eq:toprove}. We will focus for the moment on the case where $m>0$ and $c=0$.
\begin{lem}
Let $k\geq 0$ and let 
$$Q^{(m,r)}_k\coloneqq [t^k] \frac{1}{x}\left(A^{(0)}(\tilde{Z})B^{(0)}(\tilde{Z})-1\right).
$$
Then $Q^{(m,r)}_k$ is a polynomial (with coefficients in $\mathbb{Q}[\mathbf{p}, \mathbf{q},\bar{x}]$) in the variables $u_i$ and $v_j$. It is symmetric in the $\mathbf{u_I}\setminus\{u_0\}$, symmetric in the $\mathbf{u}_J$, and it has homogeneous degree at most $2k$ in these variables.
\end{lem}
\begin{proof}
It follows from the discussion preceeding the lemma that $Q^{(m,r)}_k$ is a polynomial in the variables $u_i$ and $v_j$, and that any monomial of the form 
$t^k p_\lambda q_\mu x^{-\ell-1} \prod_i u_i^{d_i}\prod_i v_i^{e_i}$ 
appearing in this polynomial has homogeneous degree $0$, which is to say
$$
\sum d_i + \sum e_i -\ell(\lambda)-\ell(\mu) = 0.
$$
Therefore $\sum d_i + \sum e_i \leq 2k$.

The symmetry stated is clear from definitions.
\end{proof}
\begin{lem}
The quantity $Q^{(m,r)}_k$ defined in the previous lemma is a polynomial (with coefficients in $\mathbb{Q}[\mathbf{p}, \mathbf{q},\bar{x}]$) in $u_0$, in the elementary symmetric functions $e_1(\mathbf{u}_I\setminus\{u_0\}),\dots,e_{2k}(\mathbf{u}_I\setminus\{u_0\})$ of $\mathbf{u}_I\setminus\{u_0\}$ and in the elementary symmetric functions $e_1(\mathbf{u}_J),\dots,e_{2k}(\mathbf{u}_J)$ of the $\mathbf{u}_J$. 
This polynomial is independent of $m$ and $r$, 
i.e. we can write
$$Q^{(m,r)}_k
=
Q^{(\infty)}_k (u_0, \big(e_\ell(\mathbf{u}_I\setminus\{u_0\})\big)_{\ell\leq 2k}, \big(e_\ell(\mathbf{u}_J)\big)_{\ell\leq 2k}) 
$$
for some polynomial $Q^{(\infty)}_k$.
\end{lem}
\begin{proof}
The first assertion is a direct consequence of the symmetry and the degree bound of the previous lemma. The fact that it is independent of $m$ and $r$ follows from considering the map which sets the last variable in $\mathbf{u}_I$ or $\mathbf{u}_j$ to zero as in the proof of Lemma~\ref{lem:projectiveLimit}. To see this, it suffices to observe that when $u_i=0$, the polynomials $A^{(i)}$ and $B^{(i)}$ are both equal to $1$, which is clear from~\eqref{eq:defWratconv2}-\eqref{eq:defBratconv2}. This implies 
$$
Q^{(m,r)}_k \big|_{u_{m-1} =0} = Q^{(m-1,r)}_k \ \ , \ \ Q^{(m,r)}_k \big|_{u_{M-1} =0} = Q^{(m,r-1)}_k,
$$
which is enough to conclude.
\end{proof}

\subsubsection{Conclusion of the proof}

The four above lemmas tell us that the polynomials $P_k^{(m,r)}$ and $Q_k^{(m,r)}$ are both functions of $\mathbf{u}_I=u_0,\dots,u_{m-1}$ and $\mathbf{u}_J=u_m,\dots,u_{m+r-1}$, and depend on them polynomially with nice symmetry properties. Now, observe that if $m\geq r+1$ and 
$$u_m=u_1, u_{m+1}=u_2, \dots, u_{m+r-1}=u_{r}$$
(in which case  the variables $\mathbf{u}_J$ are a subset of the variables $\mathbf{u}_I$), the rational function
$$
{G}(z) =  \frac{\prod_{i\in I} ( 1+z u_i)}{\prod_{j \in J} (1+ z u_j)}
$$
is in fact a polynomial. In this case, from Section~\ref{sec:constellations}, we already know that Theorem~\ref{thm:Discrat} (equivalently,~\eqref{eq:toprove}) is correct, or in other words we already know that 
$$P^{(\infty)}_k (\big(e_\ell(\mathbf{u}_I)\big)_{\ell< 2k}, \big(e_\ell(\mathbf{u}_J)\big)_{\ell< 2k}) =
\bar{u}_0Q^{(\infty)}_k (u_0, \big(e_\ell(\mathbf{u}_I\setminus\{u_0\})\big)_{\ell\leq 2k}, \big(e_\ell(\mathbf{u}_J)\big)_{\ell\leq 2k}) 
$$
in this case (see Remark~\ref{rem:artificialPoles}).
Now, we have the following lemma:
\begin{lem}\label{lemma:algebraic}
Let $P(X_1,\dots,X_{2k}; Y_1,\dots,Y_{2k} )$ and $Q(\alpha; X_1,\dots,X_{2k}; Y_1,\dots,Y_{2k})$ be two polynomials. Consider variables $u_0,u_1,\dots$ and $v_1,\dots,v_r$. Assume that there is $m_0> r$ such that the following is true:
as soon as $m\geq m_0$ and $v_1=u_1, \dots, v_r=u_r$, we have the 
equality
$$
P(e_i(u_0,\dots,u_m)_{1\leq i< 2k}; e_i(v_1,\dots,v_r)_{1\leq i< 2k} )=
\bar{u}_0Q(u_0; e_i(u_1,\dots,u_m)_{1\leq i\leq 2k}; e_i(v_1,\dots,v_r)_{1\leq i\leq 2k} ).
$$
Then in fact that equality holds  without specialisation of the variables $v_j$ and for any value of $m$.
\end{lem}
\begin{proof}
Consider the polynomial (in variables $u_0,\dots,u_m, v_1,\dots v_r$)
$$
u_0 P(e_i(u_0,\dots,u_m)_{1\leq i< 2k}; e_i(v_1,\dots,v_r)_{1\leq i< 2k} )-
Q(u_0; e_i(u_1,\dots,u_m)_{1\leq i\leq 2k}; e_i(v_1,\dots,v_r)_{1\leq i\leq 2k} ).
$$
Set $v_1=u_1, \dots, v_{r-1}=u_{r-1}$. 
The obtained polynomial is symmetric in variables $u_r,u_{r+1},\dots,u_m$. Moreover by assumption it vanishes for $v_r=u_r$, so it is divisible by $(v_r-u_r)$, and by symmetry it is divisible by $v_r-u_i$ for all $r\leq i\leq m$. Taking $m$ larger than the maximum degree of $P,Q$ plus $r$, this implies that this polynomial is null. Either $r=1$ and we are done, or we have decreased the value of $r$ by one and we can perform induction. 
\end{proof}

The last lemma and the discussion above imply that Theorem~\ref{thm:Discrat} holds in the (general) case of rational $G$, as long as $c=0$ and $I\neq \emptyset$. Moreover, the last lemma also shows that in this case the quantity $\bar{u}_0Q(u_0; e_i(u_1,\dots,u_m)_{1\leq i\leq 2k}; e_i(v_1,\dots,v_r)_{1\leq i\leq 2k} )$, which represents an arbitrary coefficient in 
$\frac{1}{x} H^{(0)}(Z(x))$, is in fact symmetric in $\mathbf{u}_I$ (including $u_0$).
Therefore we have
$$
\frac{1}{x} H^{(0)}(Z(x)) = \frac{1}{x} H^{(i)}(Z(x)),
$$
for any $i\in I$, and Theorem~\ref{thm:Discrat} holds in fact for any $c\in I$ (still requiring that $I\neq \emptyset$).

\smallskip

We now address the case $c\in J$ (which includes the case where $I=\emptyset$). In this case,  we can artificially introduce in the function $G(\cdot)$ an artificial extra  factor $(u_c+\cdot)$ in the numerator, and an artificial extra factor $(u_c+\cdot)$ in the denominator (in doing so we will have a repeated $u_c$ in the denominator but this is not a problem). After doing this, we are back to the case where $I$ is not empty, and from the previous discussion, we can thus apply Theorem~\ref{thm:Discrat} with our chosen value of $c$. 
We conclude by using Remark~\ref{rem:artificialPoles}, which tells us that the addition of artificial poles has not changed the definition of all quantities of interest $A^{(i)}, B^{(i)}, H^{(i)}$.

This concludes the proof of Theorem~\ref{thm:Discrat} in all cases.

\subsection{Proof of Theorem~\ref{thm:cylinder} for rational $G$}

In the case of polynomial $G$, Theorem~\ref{thm:cylinder} was proved in Section~\ref{sec:constellations}. To deduce the case of rational $G$ from it, one proceeds exactly as we did for $W_{0,1}$.

First, 
it is direct from definitions that coefficients of the function $W_{0,2}$ at order $t^k$ are polynomials in $u_i,v_i$ of homogeneous degree at most $2k$. Moreover, from the degree discussions of the last section, the quantity
\begin{align}\label{eq:toprove2}
\frac{Z'(x_1)Z'(x_2)}{(Z(x_1)-Z(x_2))^2} - \frac{1}{(x_1-x_2)^2}
\end{align}
has degree (for the conventional notion of degree introduced in the last section) zero. This implies that any monomial 
$$t^k p_\lambda q_\mu x^{-\ell-1} \prod_i u_i^{d_i}\prod_i v_i^{e_i}$$
appearing at order $k$ satisfies
$
\sum d_i + \sum e_i -\ell(\lambda)-\ell(\mu) = 0,
$
so that $\sum d_i + \sum e_i \leq 2k$.

Therefore the coefficients at order $t^k$ in $W_{0,2}$ and in~\eqref{eq:toprove2} are both polynomials in $\mathbf{u}_I, \mathbf{u}_J$, symmetric in each set of variables, of degree bounded by $2k$. Since we know that Theorem~\ref{thm:cylinder} is true in the case of polynomial $G$, and since these polynomials satisfy the same stability properties when substituting a variable to zero as the one analysed in the previous section, the proof is concluded exactly as in the previous section (using a simpler variant of  Lemma~\ref{lemma:algebraic} which does not require the separate role of the variable $u_0$).

\subsection{Symmetries of the polynomial $H^{(c)}$}
\label{sec:Hsymmetry}

Although this is not needed for our results, we find satisfactory to give a proof of the properties of $H^{(c)}$ stated in Remark~\ref{rem:Hsymmetry}.

Now that we have proven Theorem~\ref{thm:Discrat}, we know that
$H^c(\tilde{Z}(x)) = xW_{0,1}(x) + \sum_{k=1}^{D_1}p_kx^{k}
$.
Since we have $\tilde{Z}(x) = \bar{x} + O(t)$, the change of variables $\bar{x} \leftrightarrow \tilde{Z}(x)$ is invertible, and we have 
$$
H^{(c)}(z) = X(z) W_{0,1}(X(z) ) + \sum_{k=1}^{D_1}p_kX(z)^{k},
$$
in $\mathbb{Q}[\mathbf{p},\mathbf{q},\mathbf{u}][z,z^{-1}][[t]]$,
where we recall $X(z)= z \frac{\prod_{i\in I} A^{(i)}(\tilde{z})}{\prod_{j\in J} A^{(j)}(\tilde{z})}$.
It is clear from the last expression that $H^{(c)}(z)$ involves no negative power of $u_c$, is symmetric in the variables $\mathbf{u}_I$, symmetric in the $\mathbf{u}_J$, and that it is independent of the chosen value of $c \in I\cup J$ (since $W_{0,1}(x)$ and $X(z)$ have these properties, directly from their definitions).

\section{Appendix: Weighted Hurwitz numbers}
\label{sec:WHN}

In this section we quickly recall the interpretation of the coefficients $H$ and $H^\circ$ (i.e., the coefficients of functions $W_{g,n}$) in the general case of a rational weight-function $G$.
What follows is essentially a reminder and rephrasing of the theory of weighted Hurwitz numbers of~\cite{Guay-PaquetHarnad2017}, which we find useful to make compactly accessible.

We let $d \geq 1$ and we consider the group algebra $\mathbb{C}[\mathfrak{S}_d]$ of the symmetric group $\mathfrak{S}_d$. For $\mu \vdash d$, we let 
$$C_\mu\coloneqq\sum_{\substack{\sigma\\ type(\sigma)=\mu}} \sigma  \ \ \  \in \mathbb{C}[\mathfrak{S}_n]
$$
be the formal sum of the elements in the conjugacy class of permutations of cycle-type $\mu$. For $i=1,\dots,d$, we let $J_i$ be the $i$-th Jucys--Murphy element:
$$
J_i \coloneqq \sum_{j=1}^{i-1} (j,i)  \ \ \ \in \mathbb{C}[\mathfrak{S}_d],
$$
where $(j,i)$ is the transposition exchanging $j$ and $i$. The $J_i$ commute with each other.
It is clear by induction on $n$ that
\begin{align}\label{eq:factorJM}
C(u) = \prod_{i=1}^d (1+ u J_i) = \sum_{\sigma \in \mathfrak{S}_d} u^{d-\ell(\sigma)} \sigma,
\end{align}
where $\ell(\sigma)$ is the number of cycles of $\sigma$. This implies in particular that any symmetric function of the $J_1,\dots,J_d$ lies in the center of the group algebra.

For our application we also need to formally expand the inverse of $C(u)$, we have
$$
D(u) \coloneqq\frac{1}{C(u)} =  \frac{1}{\prod_{i=1}^d (1+ u J_i) }= \sum_{k \geq 0} (-u)^k h_k(J_1,\dots,J_d),
$$
where $h_k$ is the $k$-th homogeneous symmetric function. Note that one can write explicitly
$$
h_k(J_1,\dots,J_n) 
=\sum_{\substack{i_1\leq \dots \leq i_k \\ j_1< i_1, \dots, j_k <i_k}} (j_1,i_1) \dots (j_k,i_k),
$$
where each summand is what we call a \emph{monotone run of $k$ transpositions}, because the top elements of successive transposition are ordered in \emph{monotone} (i.e., in this context, nondecreasing) order. Note also that it is possible to expand $C(u)$ in a similar way through elementary symmetric functions, thus giving rise to a similar formula involving \emph{strictly monotone runs} of transpositions.

We have:
\begin{prop}[\cite{Guay-PaquetHarnad2017}]
The number $H(\lambda,\mu,\ell_0, ..., \ell_{m-1}; \ell_m, \dots, \ell_{M-1})$ defined in the introduction as coefficient of the function $\tau^G$ is equal to the coefficient of 
$u_0^{\ell_0}\dots u_{M-1}^{\ell_{M-1}}$
in
\begin{align}\label{eq:CSncoeff}
[\mathbf{1}]C_\lambda C_\mu C(u_0)\dots C(u_{m-1}) D(u_m)\dots D(u_{M-1}),
\end{align}
where $[\mathbf{1}]$ extracts the coefficient of the identity in the group algebra. 
\end{prop}
\begin{proof}[Sketch of proof]
The proof is conceptual enough so we can sketch it.
For a partition $\alpha\vdash d$, consider the irreducible module $V^\alpha$ of $\mathfrak{S}_d$. 
The group algebra can be decomposed as $\mathbb{C}[\mathfrak{S}_d]=\bigoplus_\alpha End(V^\alpha)$. Now, a famous (and beautiful) theorem of Jucys and Murphy (\cite{Jucys, Murphy}) states that any \emph{symmetric} polynomial $f(J_1,\dots,J_d)$ in the Jucys--Murphy elements acts on the irreducible module $V^\alpha$ as the scalar $f(c(\Box), \Box \in \alpha)$, i.e.~as the same symmetric function evaluated on the contents on the partition $\alpha$. Moreover, the conjugacy class $C_\mu$, being central, also acts as a scalar, namely as the normalised character $\frac{|C_\mu|}{\dim V^\alpha}\chi^\alpha(\mu)$. Putting all this together, the coefficient of $[\mathbf{1}]$ in the expression~\eqref{eq:CSncoeff} can be expressed as a sum over all irreducibles, an more precisely it is found to be equal to
\begin{align}\label{eq:characters}
\sum_{\alpha\vdash d} \frac{ |C_\mu| \cdot |C_\lambda|}{d!} \chi^\alpha(\mu) \chi^{\alpha}(\lambda) \prod_{\Box\in \alpha} 
\frac{\prod_{i\in I} 1+u_i c(\Box)}{\prod_{j\in J} 1+u_j c(\Box)},
\end{align}
where the only thing we have done is to compute the trace (coefficient of $1$) additively over the decomposition $\bigoplus_\alpha End(V^\alpha)$. Now remember the change of basis formula between Schur functions and powersum symmetric functions:
$$
s_\alpha(\mathbf{p}) = \sum_{\lambda}   \frac{|C_\lambda|}{d!} \chi^\alpha(\lambda) p_\lambda.
$$
Therefore, multiplying~\eqref{eq:characters} by $p_\lambda q_\mu$ and summing over $\lambda, \mu$, one obtains the expression 
\begin{align*}
d!\sum_{\alpha\vdash d} s_\alpha(\mathbf{p}) s_\alpha(\mathbf{q}) \prod_{\Box\in \alpha} 
\frac{\prod_{i\in I} 1+u_i c(\Box)}{\prod_{j\in J} 1+u_j c(\Box)},
\end{align*}
and summing over $\alpha$ times $t^{d}/d!$ gives precisely the function $\tau^G$.
\end{proof}

From what precedes, the above expression equivalently says that  $H(\lambda,\mu,\ell_0,\dots,  \ell_{m-1}; \ell_m, \dots, \ell_{M-1})$ is equal to $(-1)^{\ell_{m}+\dots+\ell_{M-1}}$ times the number of factorisations of the identity in $\mathfrak{S}_d$ of the form:
\begin{align}\label{eq:factorun}
\sigma_{-2} \sigma_{-1} \sigma_0 \dots \sigma_{m-1} \underline{\rho^{(m)}} \dots \underline{\rho^{(M-1)}}=\mathbf{1},
\end{align}
where $\sigma_{-2}$ has type $\lambda$, $\sigma_{-1}$ has type $\mu$, $\sigma_0,\dots,\sigma_{m-1}$ have respectively $d-\ell_0, \dots, d-\ell_{m-1}$ cycles, and where $\rho^{(1)}, \dots, \rho^{(M-1)}$ are monotone runs of respective lengths $\ell_{m}, \dots, \ell_{M-1}$ (as in the introduction the underlined notation $\underline{\rho}$ denotes the products of elements in the run $\rho$, which itself is a tuple of transpositions).
Such a tuple 
$$(\sigma_{-2}, \sigma_{-1} ,\sigma_0, \dots ,\sigma_{m-1}, \rho^{(m)}, \dots \rho^{(M-1)})$$
made by $m+2$ permutations, $r$ monotone runs, whose total product is the identity,
is what we call an \emph{$(m,r)$-factorisation}.
An $(m,r)$-factorisation is \emph{transitive} (or \emph{connected}) if the subgroup generated by all the permutations $\sigma_i$ and all the transpositions appearing in the runs $\rho^{(j)}$ acts transitively on $[d]$.
The \emph{genus} $g\geq0$ of a connected factorisation is defined by the formula
$$
\sum_{i=-2}^{m-1}(d-\ell(\sigma_i)) + \sum_{j=m}^M \ell_j  = 2d+2g-2, 
$$
where $\ell_j$ is the length of the run $\rho^{(j)}$.

The notions of connectedness and genus are consistent with the interpretation of these factorisations as branched covers of the sphere, see e.g.~\cite{LandoZvonkin2004} or again~\cite{Guay-PaquetHarnad2017}. In the case $r=0$, it is common in the combinatorics literature to interpret $(m,0)$-factorisations as certain coloured graphs embedded on surfaces called \emph{constellations} (either $m$- or $(m+1)-$ or $(m+2)-$ constellations depending on references), and this is recalled in Section~\ref{sec:constellations}. By analogy with this case, we sometimes abusively use the terminology \emph{white/black faces} to mention the cycles of $\sigma^{-2}, \sigma^{-1}$, respectively, and think as the quantity $n-\ell_i$ for $i\in I\cup J$ as the number of \emph{vertices of colour $j$}.

We leave as an open problem the construction of a fully combinatorial (and natural) description of $(m,r)$-factorisations as a model of embedded graphs (which could naturally be called $(m,r)$-constellations), as well as the design as an effective theory of slice decompositions for these objects, that would give bijective proofs of our expression of $W_{0,1}$ in the case of rational $G$.

\bigskip

The next statement is the basis of all the absolute convergence assumptions made in this paper. We are not aware of an existing proof of it in the literature.
Let $\hat F_g (t)= F_g\Big|_{p_i\equiv1, q_i\equiv 1, u_i \equiv 1}$ be the generating function of \emph{all} $(m,r)$-factorisations of genus $g$, where the only variable not specialised to $1$ is $t$ (whose exponent controls the size, i.e. the index $d$ of the underlying symmetric group $\mathfrak{S}_d$). We have
\begin{lem}[Absolute convergence]\label{lemma:absoluteConvergence}
There is a constant $c_{m,r}>0$, independent of $g$, such that the series $\hat F_g$ has radius of convergence at least $c_{m,r}$.
\end{lem}
\begin{proof}
The proof is inspired from the proof of the main result in~\cite{GGPN:convergence}, which covers the case $(m,r)=(0,1)$.
Consider a transitive $(m,r)$-factorisation of genus $g$,
$$
\sigma_{-2} \sigma_{-1} \sigma_0 \dots \sigma_{m-1} \underline{\rho^{(m)}} \dots \underline{\rho^{(M-1)}} =1,
$$
where we recall that the $\sigma_i$ are permutations and the $\rho^{(j)}$ are monotone runs of transpositions. Expanding naively~\eqref{eq:factorJM} (or remembering its natural proof), any permutation $\sigma$ has a \emph{unique} factorisation into a \emph{strictly monotone}  run of transpositions of length $n-\ell(\sigma)$. Making all transpositions explicit in the notation, our $(m,r)$-factorisation can be written as a product
\begin{align}\label{eq:factoToSort}
(j_1, i_1) (j_2, i_2) \dots  \dots (j_R, i_R)  = 1
\end{align}
with $R=2d+2g-2$ factors (and as usual $j_k < i_k$ for all $k$).
Consider the sequence $\mathbf{i}=(i_1,i_2,\dots,i_R)$, this sequence can be divided into $m$ stricly increasing runs of successive length $n-\ell(\sigma_1), \dots, n-\ell(\sigma_{m-1})$, and $r$ nondecreasing runs of successive length $\ell_m,\dots,\ell_{M-1}$.

Now, there is a canonical way to sort the sequence $\mathbf{i}$ in nondecreasing order acting on it with a sequence of simple swaps of the form $s_k = (k,k+1) \in \mathfrak{S}_R$ (for example use lexicographic leftmost minimal displacements). It is explained in~\cite[Section 2]{GGPN:convergence}, how the action of simple swaps can be promoted from $\{1,\dots,R\}$ to the full factorisation~\eqref{eq:factoToSort}, \textit{via} the Hurwitz action. The sequence of swaps thus transforms~\eqref{eq:factoToSort} into a factorisation of the identity whose new ``$\mathbf{i}$-sequence'' is the sorted version of $\mathbf{i}$, and in particular, into a monotone factorisation. In order to invert this process it is enough to be able to recover the (inverse) swap sequence, which, because we sorted canonically, only requires the knowledge of the final factorisation and of the mapping $\{1,\dots,R\}\longrightarrow \{0,\dots,m+r-1\}$ which associates to each transposition the run from which it was originating (this is because the initial runs are all nondecreasing). Moreover, the sorted factorisation is still transitive and of genus $g$.  See~\cite[Section 2]{GGPN:convergence} for details on the Hurwitz action.

Therefore, the total number of $(m,r)$-factorisations of size $d$ and genus $g$ is at most $(m+r)^{2d+2g-2}$ times the number of monotone factorisations of the identity of same size and genus, and the radius of convergence of $\hat{F}_g$ is at least $(m+r)^{-2}\rho$, where $\rho$ is the radius of convergence of the generating function of monotone factorisations of the identity. This number is known to be positive (in fact $\rho=\frac{2}{27}$, see~\cite{GGPN:convergence} again, or \cite{GouldenGuayPaquetNovak:polynomiality}), and the proof is complete.
\end{proof}

\begin{rem}
The rough lemma above is needed \emph{a priori} for our proofs, but once our results are obtained they enable one \emph{a posteriori} to compute very precisely the radius of convergence of $\hat F_g$ \textit{via} algebraic equations.
Indeed for general $(m,r)$, taking $D_1=D_2=1$, $\mathbf{u}_I\equiv1$ and  $\mathbf{u}_J\equiv -1$,   our spectral curve~\eqref{eq:spectralX}-\eqref{eq:spectralY} takes the form of the two equations
$U= 1+ t\frac{U^m}{V^r}\left(\frac{m-1}{U}+\frac{r}{V}\right)$, $V= 1-t \frac{U^m}{V^r}\left(\frac{m}{U}+\frac{r+1}{V}\right)$, where $U=A^{(i)}_0$ and $V=A^{(j)}_0$ for $i\in I$ and $j\in J$. Looking for the critical point of this system and eliminating variables, one finds a degree three polynomial equation (with coefficients depending on $m$ and $r$) for the critical values of $U$ and $V$, from which the critical value $t=t_{m,r}$ of $t$ can be expressed.
This is the radius of convergence of the series of all factorisations
$$
\sigma_0 \dots \sigma_{m-1} \underline{\rho^{(m)}} \dots \underline{\rho^{(M-1)}} =1,
$$
on a fixed surface. For $r=0$ the equation is explicitly solvable for general $m$ and we recover the value $t_{m,0}=\frac{(m-2)^{m-2}}{(m-1)^m}$ given in~\cite[Corollary 9.8]{Chapuy2009} (note the shift of convention by $1$ in the value of $m$), and for $(m,r)=(0,1)$ we recover the value $t_{0,1}=\frac{2^1}{3^3}=\frac{2}{27}$ from~\cite{GGPN:convergence}. 
\end{rem}

\section{Appendix: the case of a rational-exponential function $G$}
\label{sec:exp}

On the day this paper was made public on the server arxiv,  Bychkov, Dunin-Barkowski, Kazarian and Shadrin made public the paper~\cite{BDBKS4}, which proves results strongly related to ours, with completely different techniques. These authors also consider the case where the underlying weight function $G$ is a rational function times an exponential. Although we had not considered this case at first, the purpose of this section is to show that this case is, in fact, also covered by our techniques. We will only state the results here, indicating the (few) modifications to introduce to handle this case, leaving certain details to the reader.

We now consider a function $G$ of the form 
\begin{align}\label{eq:tildeGexp}
G(\cdot)  =  \frac{\prod_{i=0}^{m-1} ( 1+\cdot\, u_i)}{\prod_{j=m}^{M-1} (1+ \cdot\, u_j)} e^{u_{-1} \cdot }=  \frac{\prod_{i \in I} ( 1+\cdot\, u_i)}{\prod_{j\in J} (1+ \cdot\, u_j)} e^{u_{-1} \cdot },
\end{align}
and we define the functions $\tau, F_g, W_{g,n}$ as in~\eqref{eq:tau}, \eqref{eq:Fg}, \eqref{eq:Wgn}.

These functions can be interpreted as generating functions very similarly as in the introduction, except that we are now counting tuples 
$$(\sigma_{-2}, \sigma_{-1}, \rho^{(-1)}, \sigma_0,\dots,\sigma_{m-1}, \rho^{(m)}, \dots, \rho^{(M-1)})$$
with an additional element $\rho^{(-1)}$. Here
the $\sigma_i$ and the $\rho^{(j)}$ are as in the introduction (and receive the same weights), except that $\rho^{(-1)}$ is now a \emph{non-necessarily monotone run} of transpositions, i.e.~an arbitrary tuple of transpositions in $\mathfrak{S}_d$. If this run has length $\ell$, it receives a weight $\frac{u_{-1}^\ell}{\ell!}$.
Moreover, as before,  the total product is equal to the identity:
$$\sigma_{-2} \sigma_{-1} \underline{\rho^{(-1)}}\sigma_0\dots \sigma_{m-1} \underline{\rho^{(m)}} \dots \underline{\rho^{(M-1)}} = \mathbf{1}_{\mathfrak{S}_d}.$$
This interpretation is obtained similarly as in Section~\ref{sec:WHN}, noting that the group algebra element $\exp( u_{-1}T)$ where $T$ is the sum of all transpositions in $\mathfrak{S}_d$ (which is equal to the sum of all Jucys--Murphy elements) acts on the irreducible module $V^\lambda$ as $\exp ( u_{-1}\sum_{\Box\in\lambda}c(\Box))$.

\medskip

\subsection{Spectral curve and planar generating functions}

To write the spectral curve of the model, we consider the following variants of Equations~\eqref{eq:defWratconv2}-\eqref{eq:defBratconv2}, which define polynomials $A^{(c)}(z)$ in $z$ and $B^{(c)}(z)$ in $z^{-1}$, with $c\in I\cup J$, and polynomials $\eta(z)$ in $z$ and $\theta(z)$ in $z^{-1}$,
\begin{align}
  \label{eq:defWratconv2exp}
    A^{(c)}(z) &= 1 + u_c \sum_{s=1}^{D_2} q_{s} t^s   \biggl\{z^s \frac{\prod_{i\in I}  B^{(i)}(z)^s}{\prod_{j\in J}  B^{(j)}(z)^s}\frac{e^{su_{-1}\theta(z)}}{B^{(c)}(z)}\biggr\}^{\geq},\\
    \label{eq:defBratconv2exp}
	B^{(c)}(z) &= 1 + u_c \sum_{s=1}^{D_1} p_s \biggl[z^{-s} \frac{\prod_{i \in I} A^{(i)}(z)^s}{\prod_{j \in J} A^{(j)}(z)^s}\frac{e^{s u_{-1}\eta(z)}}{A^{(c)}(z)}\biggr]^{<},
\end{align}
and
\begin{align} \label{eq:defAlpha}
    \eta(z) &= \sum_{s=1}^{D_2} q_{s} t^s   \biggl\{z^s \frac{\prod_{i\in I}  B^{(i)}(z)^s}{\prod_{j\in J}  B^{(j)}(z)^s}e^{s u_{-1}\theta(z)}\biggr\}^{\geq}\\
    \label{eq:defBeta}
    \theta(z) &= \sum_{s=1}^{D_1} p_s \biggl[z^{-s} \frac{\prod_{i \in I} A^{(i)}(z)^s}{\prod_{j \in J} A^{(j)}(z)^s}e^{s u_{-1}\eta(z)}\biggr]^{<}.
\end{align}
These polynomials (whose coefficients are as before formal power series in $t$) have, as before, degree $D_2$ and $D_1$, respectively.
We then introduce the series $Z(x)\in \mathbb{Q}[\mathbf{p},\mathbf{q},\mathbf{u},\bar{x}][[t]]$ defined by the following equation
\begin{align} \label{eq:ZasExcursionsratexp}
	Z(x) = \bar{x} e^{u_{-1}\eta(Z)} \frac{\prod_{i\in I} A^{(i)}(Z)}{\prod_{j\in J} A^{(j)}(Z)},
\end{align}
and we define the Laurent polynomials
$H^{(c)}(z) \in \mathbb{Q}[\mathbf{p},\mathbf{q},\mathbf{u},\bar{\mathbf{u}}][[t]][z,z^{-1}]$ by 
\begin{equation}\label{eq:defHexp}
H^{(c)}(z) \coloneqq \bar{u}_c\left( A^{(c)}(z) B^{(c)}(z) - 1 \right),
\end{equation}
with $c\in I \cup J$.

Once these new quantities are introduced, the spectral curve of the model can be formulated as before.
\begin{defn}[Spectral curve of our model, rational-exponential case] \label{def:SpectralCurveexp}
	We consider the system of equations	defined by
	\begin{align}\label{eq:spectralXexp}
		z X(z)&=  e^{u_{-1}\eta(z)}\frac{\prod_{i  \in I} A^{(i)}(z)}{\prod_{j \in J} A^{(j)}(z)}
		, \\ \label{eq:spectralYexp}
		X(z) Y(z) &= H^{(c)}(z),
\end{align}
	where we recall that the quantities $\eta$, $A^{(i)}$ and $H^{(c)}$ are given by~\eqref{eq:defWratconv2exp}-\eqref{eq:defBeta} and~\eqref{eq:defHexp}, and where $c$ is any value in $I\cup J$.
\end{defn}

Then we have:
\begin{thm}[Disk generating function $W_{0,1}$, rational exponential-case] \label{thm:Discratexp}
	The disk generating function $W_{0,1}(x)$ is, up to an explicit shift, given by the parametrisation $(Y(z),X(z))$ given above. Namely, we have:
	$$W_{0,1}(x) + \sum_{k=1}^{D_1}p_kx^{k-1} = Y(Z(x)),$$
	in $\mathbb{Q}[\mathbf{p},\mathbf{q},\mathbf{u},x,\bar{x}][[t]]$.
\end{thm}

\begin{thm}[Cylinder generating function $W_{0,2}$, rational-exponential case]\label{thm:cylinderexp}
	The cylinder generating function $W_{0,2}$ is
	$$W_{0,2}(x_1,x_2)= \frac{Z'(x_1)Z'(x_2)}{(Z(x_1)-Z(x_2))^2} - \frac{1}{(x_1-x_2)^2},$$
	in $\mathbb{Q}[\mathbf{p},\mathbf{q}, \mathbf{u}](x_1,x_2)[[t]]$.
\end{thm}

\begin{proof}[Sketch of the proof of Theorems~\ref{thm:Discratexp} and~\ref{thm:cylinderexp}]
We introduce a new integer parameter $N$ and we introduce the rational function
$$
	G_N(\cdot)  =   \frac{\prod_{i \in I} ( 1+\cdot\, u_i)}{\prod_{j\in J} (1+ \cdot\, u_j)} \left(1+\frac{ \cdot\,u_{-1}}{N}\right)^N.
$$
	Then one has $G_N \longrightarrow G$ when $N$ goes to infinity, coefficient by coefficient. This easily implies that $\tau^{G_N} \longrightarrow \tau^{G}$ and, with self-explanatory notation, $W_{g,n}^{G_N}\longrightarrow W_{g,n}$, again for each coefficient (indeed, each coefficient of the function $W_{g,n}$ depends on only finitely many evaluations of the function $G$).

	We can then introduce functions $A^{(c)}$ and $B^{(c)}$ as in~\eqref{eq:defWratconv2}-\eqref{eq:defBratconv2}, where $c$ belongs to $I' \cup J$, where $I'=I \cup K$ with $K=\{-2,\dots,-N-1\}$ and $u_i=u_{-1}/N$ for $i\in K$. The equations write, for $c\in I\cup J$
\begin{align}\label{eq:july1}
	A^{(c)}(z) &= 1 + u_c \sum_{s=1}^{D_2} q_{s} t^s   \biggl\{z^s \frac{\prod_{i\in I}  B^{(i)}(z)^s}{\prod_{j\in J}  B^{(j)}(z)^s}(B^{(-1)}(z))^{Ns} \frac{1}{B^{(c)}(z)}\biggr\}^{\geq},\\
    B^{(c)}(z) &= 1 + u_c \sum_{s=1}^{D_1} p_s \biggl[z^{-s} \frac{\prod_{i \in I} A^{(i)}(z)^s}{\prod_{j \in J} A^{(j)}(z)^s} (A^{(-1)}(z))^{Ns}\frac{1}{A^{(c)}(z)}\biggr]^{<},
\end{align}
	and for $c \in K$ (it suffices to consider $c=-2$ as all $c\in K$ give the same value; we abusively use the upper index $^{(-1)}$ for this common value)
\begin{align}
	A^{(-1)}(z) &= 1 + \frac{u_{-1}}{N} \sum_{s=1}^{D_2} q_{s} t^s   \biggl\{z^s \frac{\prod_{i\in I}  B^{(i)}(z)^s}{\prod_{j\in J}  B^{(j)}(z)^s}(B^{(-1)}(z))^{Ns} \frac{1}{B^{(-1)}(z)}\biggr\}^{\geq},\\\label{eq:july2}
	B^{(-1)}(z) &= 1 + \frac{u_{-1}}{N} \sum_{s=1}^{D_1} p_s \biggl[z^{-s} \frac{\prod_{i \in I} A^{(i)}(z)^s}{\prod_{j \in J} A^{(j)}(z)^s} (A^{(-1)}(z))^{Ns}\frac{1}{A^{(-1)}(z)}\biggr]^{<},
\end{align}
    Set $\eta(z) = \lim_{N\to\infty} N(A^{(-1)}(z)-1)/u_{-1}$ and $\theta(z) = \lim_{N\to\infty} N(B^{(-1)}(z)-1)/u_{-1}$.
	It is easy to see, using again that $(1+\frac{u}{N})^{Ns}\rightarrow e^{us}$, that the  generating functions $A^{(c)}$, $B^{(c)}$ thus defined converge, coefficient by coefficient when $N$ goes to infinity, to the quantities defined in~\eqref{eq:defWratconv2exp}-\eqref{eq:defBratconv2exp} while $\eta(z)$ and $\theta(z)$ satisfy~\eqref{eq:defAlpha},~\eqref{eq:defBeta}.

This in turns implies that the quantity $Z$ defined by
\begin{align}\label{eq:july3}
	Z(x) = \bar{x} \left(A^{(-1)}(Z)\right)^{N} \frac{\prod_{i\in I} A^{(i)}(Z)}{\prod_{j\in J} A^{(j)}(Z)},
\end{align}
	converges (for the same notion of convergence) to the quantity defined in~\eqref{eq:ZasExcursionsratexp}.

 Now we can apply Theorems~\ref{thm:Discrat} and~\ref{thm:cylinder} for the function $G_N$ (which is rational!), thus expressing $W_{0,1}$ and $W_{0,2}$ in terms of the quantities  $A^{(c)}$, $B^{(c)}$ and $Z$ of~\eqref{eq:july1}--\eqref{eq:july2}, \eqref{eq:july3} (all these quantities depend implicitly on $N$).
	Theorems~\ref{thm:Discratexp} and~\ref{thm:cylinderexp} follow by taking the limit of these expressions when $N$ tends to infinity and applying the convergence $W_{g,n}^{G_N}\longrightarrow W_{g,n}$ observed at the beginning of the proof.
\end{proof}

\subsection{Topological recursion}

The topological recursion holds in this more general case, with the spectral curve~\eqref{eq:spectralXexp}-\eqref{eq:spectralYexp}. The proof given in Section~\ref{sec:toprecproof} works again, up to clarifying the analytic assumptions and checking the case $\alpha=0$. Namely: 
\begin{itemize}
	\item Introducing the scaling $p_i \mapsto \alpha p_i$ as in Section~\ref{sec:toprecproof}, the spectral curve for $\alpha=0$ becomes
\[	\begin{split}
	z X_0(z)&=  e^{u_{-1}\sum_{r=1}^{D_2} q_r t^r z^{r}}\frac{\prod_{i  \in I} \left(1 + u_i\sum_{r=1}^{D_2} q_r t^r z^{r}\right)}{\prod_{j  \in J} \left(1 + u_j\sum_{r=1}^{D_2} q_r t^r z^{r}\right)}
=G(Q(t z)),
	\\
X_0(z) Y_0(z) &=
 \sum_{i=1}^{D_2}t^i  q_i z^{i}
	= Q(tz),
\end{split}\]
where $Q(z)=\sum_{r= 1}^{D_2} q_{r}z^{r}$. It coincides with the spectral curve given in~\cite{BDBKS2} for this case. We can then, again, use the fact that TR is known for $\alpha=0$.
The equation for the initial ramification points now writes:
\[
-a^{-1} +u_{-1}\sum\limits_{k=1}^{D_2} k\, q_k a^{k-1}+ \sum_{i\in I} \frac{\sum_{k=1}^{D_2} k q_k a^{k-1}}{1/u_i + \sum_{k=1}^{D_2} q_k a^k}-\sum_{j\in J} \frac{\sum_{k=1}^{D_2} k q_k a^{k-1}}{1/u_j + \sum_{k=1}^{D_2} q_k a^k} =0.
\]
This polynomial equation has $(M+1)D_2$ solutions $a_1,\dots,a_{(M+1)D_2}$.
	\item The fact that the ramification points -- solutions of the polynomial equation $\frac{X'}{X}=0$ -- have an expansion of the form $b_i\sim \frac{a_i}{t}$ for $i=1,\dots,(M+1)D_2$ follows similarly as in Section~\ref{sec:tr:rampoints}, observing that the series $A^{(c)}_k$ and $B^{(c)}_k$ still have expansions of the form 
\[A^{(c)}_0 = 1+O(t),\qquad A^{(c)}_{k} = u_c q_k t^{k} + O(t^{k+1}),\qquad B^{(c)}(z) = 1+u_c\sum\limits_{s=1}^{D_1}p_s z^{-s}+O(t). \]	
Those expansions also enable one to check that $Y'$ is generically nonzero at the ramification points.
\end{itemize}
Once all these details are checked, the analytic assumptions are the same as in Section~\ref{sec:tr:rampoints} and the proof of Section \ref{sec:toprecproof} goes through the same way. Observe that in this case $X$ is no longer a polynomial but $X'/X$ and $XY$ still are. These two conditions imply that $\omega_{0,1}=Y \dd X$ is meromorphic, which places us in the usual setting of topological recursion.

\subsection{Absolute convergence}
The only remaining check to carry out is the analogue of Lemma~\ref{lemma:absoluteConvergence} which shows that all the generating functions considered are absolutely convergent.
The proof is very similar to the one of that lemma. Just observe that a factorisation of the form
$$\underline{\rho^{(-1)}}\sigma_0\dots \sigma_{m-1} \underline{\rho^{(m)}} \dots \underline{\rho^{(M-1)}} = \mathbf{1}_{\mathfrak{S}_d},$$
where $\sigma_{0},\dots, \sigma_{m-1}$ are strictly monotone runs of transpositions, $\rho^{(m)}$, $\rho^{(M-1)}$ are monotone runs, and $\rho^{(-1)}$ is an arbitrary run of length $\ell$, can be reordered into a monotone run using the Hurwitz action as in the proof of Lemma~\ref{lemma:absoluteConvergence}. The number of ways to invert this process is at most
\[
{2n \choose \ell} \ell! (m+r)^{2n},
\]
where the factor ${2n\choose \ell}$ accounts for the choice of the transpositions which will be moved towards to the run $\rho^{(-1)}$, and $\ell!$ for their ordering. The factor $(m+r)^{2n}$ is as in the proof of Lemma~\ref{lemma:absoluteConvergence}. Since the run $\rho^{(-1)}$ is counted with a weight $\frac{1}{\ell!}$ in the generating function, the remaining weight is at most 
\[
2^{2n} (m+r)^{2n},
\]
which is enough to conclude.
\bibliographystyle{alpha}
\bibliography{biblio}

\end{document}